\documentclass[a4paper,12pt,openany,leqno]{article} 
\usepackage{calc}
\usepackage[all]{xy}
\usepackage[centertags]{amsmath}
\usepackage{latexsym}
\usepackage{amsfonts}
\usepackage{cases}
\usepackage{array}
\reversemarginpar \DeclareMathAlphabet{\mathpzc}{OT1}{pzc}{m}{it}
\usepackage{amssymb}
\usepackage{amsthm}
\usepackage{multicol}
\usepackage{rotating}
\usepackage{multirow}
\usepackage[T1]{fontenc}
\usepackage{lmodern}
\usepackage{color}
\usepackage{fancyhdr}
\usepackage [dvips]{epsfig}
\fancypagestyle{plain}{ 
\fancyhead{} 
}

\usepackage{newlfont}

\newlength{\defbaselineskip}
\newcommand{\setlinespacing}[1]%
           {\setlength{\baselineskip}{#1 \defbaselineskip}}

\usepackage{titlesec}%
\renewcommand{\thesection}{\arabic{section}}
\titleformat{\section}{\footnotesize\scshape\filcenter}{\thesection.}{1em}{}

\titleformat{\subsection}{\small\scshape}{\thesubsection.}{1em}{}
\newtheorem{proposition}{\textbf{Proposition}}[section]
\theoremstyle{plain}
\newtheorem{remark}{\textbf{Remark}}[section]
\theoremstyle{plain}
\newtheorem{theorem}{\textbf{Theorem}}[section]
\theoremstyle{plain}

\theoremstyle{plain}

\theoremstyle{plain}
\newtheorem{lemma}{\textbf{Lemma}}[section]
\theoremstyle{plain}

\theoremstyle{plain}

\theoremstyle{plain}

\theoremstyle{plain}

\usepackage[latin1]{inputenc}

\usepackage[latin1]{inputenc}

\usepackage[english]{babel}

\usepackage{graphicx}
\usepackage{t1enc}
\makeatletter
\numberwithin{equation}{section}
\usepackage[english]{babel}
\usepackage{graphicx}
\usepackage{t1enc}
\begin{document}
\title{\small{\textbf{STOCHASTIC-PERIODIC HOMOGENIZATION OF MAXWELL'S\\ EQUATIONS WITH NONLINEAR AND PERIODIC\\  CONDUCTIVITY}}}
\author{\footnotesize{JOEL FOTSO TACHAGO$^{\ddagger}$ AND HUBERT NNANG$^{\dagger}$}}

\maketitle
\begin{abstract}
Stochastic-periodic homogenization is studied for the Maxwell equations with nonlinear and periodic electric conductivity. It is shown by the stochastic-two-scale convergence method that the sequence of solutions of a class of highly oscillatory problems converges to the solution of a homogenized Maxwell equation.
\footnote{\begin{flushleft}
		Date: December, 2023.
	\end{flushleft}  }
\end{abstract}

\section{INTRODUCTION}
\par We present homogenization of Maxwell equations with non linear conductivity using stochastic two scale convergence, that is the asymptotic behavior as $\mathbb{R}_{+}^{*} \ni \varepsilon \rightarrow 0$ of $\varepsilon-$ problem below: 
\begin{equation}\label{1}
\left\{ \begin{array}{l c l}
 \partial_{t}D_{\varepsilon}(x, \omega, t)+ J_{\varepsilon}(x, w, t) &=& $curl$ \;H^{\varepsilon}(x, w, t)+ F^{\varepsilon}(x, \omega, t) \;$in$\; \Omega \times \Lambda \times ]0,\; T[\\
 \partial_{t}B_{\varepsilon}(x, \omega, t)&=& -$curl$\; E^{\varepsilon}(x, \omega, t) \; $in$\; \Omega \times \Lambda \times ]0, \; T[\\
\mathrm{div} B_{\varepsilon}(x, \omega, t) &=& 0\; $in$\; \Omega \times \Lambda \times ]0, \; T[ \\
\mathrm{div} D_{\varepsilon}(x, \omega, t) &=& Q^{\varepsilon}(x, \omega, t)\; $in$ \;\Omega \times \Lambda \times ]0, \; T[\\
E_{\varepsilon}(x, \omega, 0)=E^{0}_{\varepsilon}(x, \omega) &,& H^{\varepsilon}(x, \omega, 0)= H^{0}_{\varepsilon}(x, \omega)\; $in$ \;\Omega \times \Lambda\\
\gamma\wedge E_{\varepsilon}(x, \omega, t) &=& 0 \;$on$\; \partial\Omega \times \Lambda \times ]0, \;T[,
\end{array}
\right.
\end{equation}
with 
\begin{equation}\label{2}
\begin{array}{c}
B^{i}_{\varepsilon}(x, \omega, t)=\mu_{ij}\left(x, \hbox{\LARGE$\tau$}\left(\frac{x}{\varepsilon}\right)\omega, \frac{x}{\varepsilon^{2}}\right)H^{j}_{\varepsilon}(x, \omega, t)\\
J^{i}_{\varepsilon}(x, \omega, t)= \sigma_{i}\left(x, \hbox{\LARGE$\tau$}\left(\frac{x}{\varepsilon}\right)\omega, \frac{x}{\varepsilon^{2}}, E_{\varepsilon}(x, \omega, t)\right)\\
D^{i}_{\varepsilon}(x, \omega, t)=\eta_{ij}\left(x, \hbox{\LARGE$\tau$}\left(\frac{x}{\varepsilon}\right)\omega, \frac{x}{\varepsilon^{2}}\right)E^{j}_{\varepsilon}(x, \omega, t),
\end{array}
\end{equation}
where, $T$ is a fixed strictly positive real, $\Omega$ is a connected bounded open subset of $\mathbb{R}^{3}$ with regular boundaries $\partial\Omega$, $\gamma$ is the unit outgoing normal vector on $\partial\Omega$, $\Lambda$ a probability space, $\mu_{ij}$ and $\eta_{ij}\; (1\leq i,j\leq 3)$ describing magnetic permeability  and permitivity respectively are functions from $\overline{\Omega}\times \Lambda \times \mathbb{R}^{3}_{z}$ into $\mathbb{R}^{9}$ satisfying the following conditions:
\begin{equation}
\mu_{ij}, \eta_{ij} \in \mathcal{C}\left(\overline{\Omega}; L^{\infty}\left(\Lambda; L^{\infty}_{per}(Z)\right)\right), 1\leq i,j\leq 3,
\end{equation}
where $Z=(0, 1)\times \left(0, 1\right)\times \left(0,1\right),$ and $L^{\infty}_{per}(Z)$ is the space of all functions in $L^{\infty}\left(\mathbb{R}^{3}\right)$ that are $Z$-periodic, for all $x\in \overline{\Omega}$ and almost everywhere $(a.e)\; \omega\in\Lambda$
\begin{equation}
\mu_{ij}(x, \omega, z)=\mu_{ji}(x, \omega, z), \; \eta_{ij}(x, \omega, z)=\eta_{ji}(x, \omega, z)
\end{equation}
and there exists two constants $c_{1}, c_{2}>0$ such that for all $x\in \overline{\Omega}$ and for almost all $(\omega, z)\in \Lambda \times Z$, 
\begin{equation}
\begin{array}{l c l}
max\left(|\mu_{ij}(x, \omega, z)\lambda^{j}|, |\eta_{ij}(x, \omega, z)\lambda^{j}|\right)&\leq& c_{1}|\lambda|,\\
min\left(\mu_{ij}(x, \omega, z)\lambda^{j}\lambda^{i}, \eta_{ij}(x, \omega, z)\lambda^{j}\lambda^{i}\right)&\geq& c_{2}|\lambda|^{2}
\end{array}
\end{equation}
for all $\lambda=\left(\lambda^{i}\right) \in \mathbb{R}^{3}. \;\; \sigma=(\sigma_{i})$, describes conductivity, each $\sigma_{i}, 1\leq i\leq 3$ being a function from $\overline{\Omega}\times \Lambda \times \mathbb{R}^{3}_{z}\times \mathbb{R}^{3}_{\xi}$ to $\mathbb{R}$ such that:
\begin{enumerate} 
\item[(i)] for all $x\in \overline{\Omega}, \xi\in \mathbb{R}^{3}$ and for almost all $\omega\in \Lambda$, $\sigma_{i}(x, \omega, ., \xi)$ is measurable and $Z$-periodic, $Z=\left(0, 1\right)\times \left(0, 1\right)\times \left(0, 1\right)$;
\item[(ii)] for all $x\in \overline{\Omega},\xi\in \mathbb{R}^{3}$ and for almost all $z\in \mathbb{R}^{3}, \sigma_{i}(x, ., z,\xi)$ is measurable;
\item[(iii)] for all $x\in \overline{\Omega}$ and for almost all $(\omega, z)\in \Lambda \times \mathbb{R}^{3}_{z},\; \sigma_{i}(x,\omega,z,0)=0$;
\item[(iv)] for almost all $(\omega, z)\in \Lambda \times \mathbb{R}^{3}_{z},\; \sigma_{i}(., \omega, z, .)$ is continuous on $\overline{\Omega}\times \mathbb{R}^{3}_{\xi}$ and there exists a constant $c_{1}>0$ such that 
\begin{equation}\label{6}
\left|\sigma(x, \omega, z, \xi)-\sigma(x, \omega, z, \xi')\right|\leq c_{1}|\xi-\xi'|;
\end{equation}
for all $\left(x, \xi, \xi'\right)\in \overline{\Omega}\times \mathbb{R}^{3}\times \mathbb{R}^{3}$ and for almost all $(\omega, z)\in \Lambda \times \mathbb{R}^{3}_{z}$;
\item[(v)] there exists a constant $\delta\geq 1$ such that for all $\xi_{1}, \xi_{2}\in \mathbb{R}^{3}$, for all $x\in \overline{\Omega}$ and for almost all $(\omega, z)\in \Lambda \times \mathbb{R}^{3}$:
\begin{equation}\label{7}
\left(\sigma(x, \omega, z, \xi_{1})-\sigma(x, \omega, z, \xi_{2}), \xi_{1}-\xi_{2}\right)\geq \delta|\xi_{1}-\xi_{2}|^{2}
\end{equation}
\item[(vi)] for all $\xi_{1},\xi_{2}\in \mathbb{R}^{3}$, for all $x\in \overline{\Omega}$ and for almost all $(\omega, z)\in \Lambda\times \mathbb{R}^{3}$:
\begin{equation}\label{8}
(D_{\xi}\sigma(x, \omega, z, \xi_{1})\xi_{2}, \xi_{2})\geq 0.
\end{equation}
\end{enumerate}
\par Finally, initial values $E^{0}_{\varepsilon}=\big(E^{i0}_{\varepsilon}\big)$ and $H^{0}_{\varepsilon}=\left(H^{i0}_{\varepsilon}\right)$ are taken strongly convergent in $H_{\mathrm{curl}}\left(\Omega; L^{2}(\Lambda)\right)$ to $E^{0}=\left(E^{i0}\right)$ and $H^{0}=\left(H^{i0}\right)$ (respectively), $F_{\varepsilon}=\left(F^{i}_{\varepsilon}\right)$ is a given source in $H_{\mathrm{div}}\left(Q; L^{2}(\Lambda)\right)$ such that the sequence $\left(F_{\varepsilon}\right)_{0<\varepsilon\leq 1}$ converges strongly in $H_{\mathrm{div}}\big(Q; L^{2}(\Lambda)\big)$ to $F=\left(F^{i}\right)$, with in addition $\partial_{t}F_{\varepsilon}$, bounded in  $L^{\infty}\big(0, T; \mathbb{L}^{2}(\Omega\times \Lambda)\big)$ and $\partial^{2}_{t}F_{\varepsilon}$ bounded in $L^{2}\big(0, T; \mathbb{L}^{2}(\Omega\times \Lambda)\big)$.
\par Such problem arises for example from antenna phenomenon where it is established that Ohm's law is not linear see \cite{2, 16}, and references therein.
Coefficient depending on $\big(x,\hbox{\LARGE{$\tau$}}\big(\frac{x}{\varepsilon}\big)\omega, \frac{x}{\varepsilon^{2}}\big)$ reflects on one hand the non homogeneous environment where the phenomenon is taking place, that is the difference in size of quantities involved rending impossible any tentative of direct numerical approximation \cite{30}, justifying therefore the study of asymptotic behavior. On the other hand, it expresses the fact we are taking in account and simultaneously 
predictable and unpredictable aspects. Indeed, a parameter can not be both predictable and unpredictable; moreover from experiment, phenomena do not 
repeat identically in some aspects and repeat identically in some, may be due to experimentation conditions or may be due to lack of knowledge on the phenomenon \cite{14, 29}.
\par Homogenization of Maxwell's equations have been studied by many authors using diverse techniques. In \cite{33, 34} is used Tartar method (\cite{35}). In \cite{37, 38} 
Nguetseng two scale convergence is used (\cite{24}), and recently in \cite{17}, (Woukeng et $al.$ \cite{31}) stochastic two sale convergence method is used for linear
 system. To the best of our knowledge, homogenization of the Maxwell's equations with non linear conductivity was first studied in \cite{38}. It has surely contributed in bringing us close to reality. Nevertheless it seem perfectible as stochastic aspects are considered negligible there. In order to improve understanding of such situation, we consider stochastic aspect and  model them by $(\Lambda, \mathcal{M})$ a probability space on which there is a dynamic system \hbox{\LARGE{$\tau$}} \cite{29}. To the best of our knowledge, existing work in this direction restrict only to stochastic aspects with ergodic dynamical systems \cite{29}, which may 
with no doubt still leave shaded on understanding such phenomena. For sake of simplicity, $\Lambda$ is such that spaces in here are separable. For general
 case, we refer to (\cite{21} page 162 footnote comment). Recent development in homogenization theory allow to go further \cite{31, 32} by considering non ergodic 
dynamical systems and couple deterministic-stochastic aspects.
\par Remaining of this paper is organized as follows: section 2 deals with existence of solutions and section 3 with homogenization results. Except contrary is mentioned, spaces take real values, as nonlinear phenomena are modeled in $\mathbb{R}$ and not in $\mathbb{C}$, (cf\cite{32}) and $dx$ refer to Lebesgue 
measure in $\mathbb{R}^{3}.$
\section{FUNDAMENTALS OF STOCHASTIC-TWO-SCALE CONVERGENCE}
\par For sake of self content and completeness we repeat here basic concepts on stochastic-two-scale convergence \cite{17}.
\subsection{Dynamical systems.} Let $\left(\Lambda, \mathcal{M}, \mu\right)$ be a probability space, and let $\mathcal{L}$ be the $\sigma$-algebra of Lebesgue measurable sets in $\mathbb{R}^{N}$ (integer $N\geq 1$). By an $N-dimensional \; dynamical \; system \; on \; \Lambda$, we meant any family of invertible mappings, $\hbox{\LARGE{$\tau$}}(y) : \Lambda \rightarrow \Lambda,\; y\in \mathbb{R}^{N}$, satisfying the following conditions:
\begin{itemize} 
\item[(i):] $(group\; property)$\; $\hbox{\LARGE{$\tau$}}(O)=Id$ and $\hbox{\LARGE{$\tau$}}(x+y)=\hbox{\LARGE{$\tau$}}(x)\circ\hbox{\LARGE{$\tau$}}(y)$ for all $x, y \in \mathbb{R}^{N}$, ("$\circ$" being the usual composition of mappings, and $O$ the origin in $\mathbb{R}^{N}$);
\item[(ii):]$(invariance)$ for all $y\in \mathbb{R}^{N}$, the mapping $\hbox{\LARGE{$\tau$}}(y)$ is measurable and preserves the $\mu$-measure, i.e., $\mu\left( \hbox{\LARGE{$\tau$}}(y)F\right)=\mu(F)$ for every $F\in \mathcal{M}$;
\item[(iii):] $(measurability)$ for any $F\in \mathcal{M}$, the set $\left\{(y, \omega): \hbox{\LARGE{$\tau$}}(y)\omega \in F\right\}$ is measurable with respect to $\mathcal{L}\otimes \mathcal{M}$, (the $\sigma$-algebra generated by the family of Cartesian products $L\times M$, $L\in \mathcal{L}$ and $M\in \mathcal{M}$).
\end{itemize}
\par Given $1\leq p\leq \infty$, a dynamics system $\hbox{\LARGE{$\tau$}}=\left\{ \hbox{\LARGE{$\tau$}}(y): \Lambda \rightarrow \Lambda, y\in \mathbb{R}^{N} \right\}$ induces a strongly continuous $N$-parameter group of isometries $\mathcal{U}(y): L^{p}(\Lambda) \rightarrow L^{p}(\Lambda), y\in \mathbb{R}^{N}$, defined by $\left( \mathcal{U}(y)f\right)(\omega)=f\left(\hbox{\LARGE{$\tau$}}(y)\omega\right),\; f\in L^{p}(\Lambda)$; for details, see e;g;, \cite{18, 26}. In particular the strong continuity is expressed by $\|\mathcal{U}(y)f-f\|_{L^{p}(\Lambda)}\rightarrow 0$ as $|y|\rightarrow 0$. The $i$th stochastic derivative $D^{\omega}_{i}: L^{p}(\Lambda)\rightarrow L^{p}(\Lambda)\; (1\leq i\leq N)$ is an unbounded linear mapping of domain
\[\mathbf{D}_{i}=\left\{f\in L^{p}(\Lambda) : \lim_{\iota\rightarrow 0} \frac{\mathcal{U}(\iota e_{i})f-f}{\iota}\;\hbox{ exists  in } L^{p}(\Lambda) \right\}\]
(($e_{i}$) being the canonical basis of $\mathbb{R}^{N}$) such that for all $f\in \mathbf{D}_{i}$,
\[ D^{\omega}_{i}f(\omega)=\lim_{\iota \rightarrow 0}\frac{f\left(\hbox{\LARGE{$\tau$}}(\iota e_{i})\omega\right)-f(\omega)}{\iota} \hbox{ for almost all }     \omega\in \Lambda\]
Higher order stochastic derivatives can be defined analogously by setting $D^{\alpha}_{\omega}=D^{\alpha_{1},\omega}_{1}\circ\cdots\circ D^{\alpha_{N},\omega}_{N}$ for every multi-index $\alpha=(\alpha_{1},\cdots \alpha_{N})\in \mathbb{N}^{N}$, where $D^{\alpha_{i},\omega}_{i}=D^{\omega}_{i}\circ\cdots \circ D^{\omega}_{i}\; (\alpha_{i}$-times). This being so, put $\mathcal{D}_{p}(\Lambda)=\cap_{i=1}^{N}\mathbf{D}_{i}$, define
\[\mathcal{D}^{\infty}_{p}(\Lambda)=\left\{f\in L^{p}(\Lambda) : D^{\alpha}_{\omega}f\in \mathcal{D}_{p}(\Lambda) \hbox{ for all } \alpha \in \mathbb{N}^{N} \right\},\]
and thanks to \cite{3}, it can be shown that $\mathcal{D}^{\infty}_{\infty}(\Lambda)\equiv \mathcal{C}^{\infty}(\Lambda)$ is dense in $L^{p}(\Lambda)$ for $1\leq p< \infty$. Endowed with a suitable locally convex topology, $\mathcal{C}^{\infty}(\Lambda)$ is a Fr\'{e}chet space. Any continuous linear form on $\mathcal{C}^{\infty}(\Lambda)$ is referred to as a stochastic distribution; and the space $\mathcal{C}^{\infty}(\Lambda)'$ of all stochastic distributions is endowed with the strong dual topology. The stochastic weak derivative of multi-index $\alpha\; (\alpha\in \mathbb{N}^{N})$ of $T\in \mathcal{C}^{\infty}(\Lambda)'$ is the stochastic distribution $D^{\alpha}_{\omega}T$ given by $\langle D^{\alpha}_{\omega}T, \varphi\rangle=(-1)^{|\alpha|}\langle T, D^{\alpha}_{\omega}\varphi \rangle$  $\left(\varphi \in\mathcal{C}^{\infty}(\Lambda)\right)$. Moreover, by the above density, $L^{p}(\Lambda)\; (1\leq p<\infty)$ is a subspace of $\mathcal{C}^{\infty}(\Lambda)'$ with continuous embedding. So, the stochastic derivative of $f\in L^{p}(\Lambda)$ exists with $\langle D^{\alpha}_{\omega}f, \varphi \rangle = (-1)^{|\alpha|}\int_{\Lambda}fD^{\alpha}_{\omega}\varphi d\mu$ for all $\varphi \in \mathcal{C}^{\infty}(\Lambda)$. Therefore, we define the Sobolev spaces on $\Lambda$ as follows (see \cite{31}):
\[W^{1,p}(\Lambda))=\left\{f\in L^{p}(\Lambda): D^{\omega}_{i}f\in L^{p}(\Lambda)\; (1\leq i \leq N) \right\},\]
where $D^{\omega}_{i}f$ is taken in the distribution sense on $\Lambda$. Equipped with the norm 
\[ \|f\|_{W^{1, p}(\Lambda)}=\left(\|f\|^{p}_{L^{p}(\Lambda)}+\sum_{i=1}^{N}\|D^{\omega}_{i}f\|^{p}_{L^{p}(\Lambda)} \right)^{\frac{1}{p}},\;\; f\in W^{1,p}(\Lambda),\]
$W^{1,p}(\Lambda)$ is a Banach space. Instead of $W^{1,p}(\Lambda)$ we will be concern with one of its semi-normed subspace. To this end, let us recall that:
\begin{enumerate}
\item[\textbf{(iv):}](property) Given $f\in \mathcal{D}^{\infty}_{1}(\Lambda)$ and for a.e. $\omega \in \Lambda$, the function $y\rightarrow f(\hbox{\LARGE{$\tau$}}(y)\omega) $ lies in $\mathcal{C}^{\infty}(\mathbb{R}^{N})$ and $D^{\alpha}_{y}f(\hbox{\LARGE{$\tau$}}(y)\omega)=(D^{\alpha}_{\omega}f)(\hbox{\LARGE{$\tau$}}(y)\omega)$ for any 
$\alpha \in \mathbb{N}^{N}$.
\item[\textbf{(v):}] (definition) By $f\in L^{p}(\Lambda),\; 1\leq p\leq \infty$, is \hbox{\LARGE{$\tau$}}-invariant is meant for any $y\in \mathbb{R}^{N}, \; f(\hbox{\LARGE{$\tau$}}(y)\omega)=f(w)$ for a.e. $\omega \in \Lambda$.
\end{enumerate}
Denoting by $L^{p}_{nv}(\Lambda)$, the set of all \hbox{\LARGE{$\tau$}}-invariant functions in $L^{p}(\Lambda)$, the dynamical system \hbox{\LARGE{$\tau$}} is termed to be ergodic if $L^{p}_{nv}(\Lambda)$ is the set of constant functions.
\begin{enumerate}
\item[\textbf{(vi):}] (property) Given $f\in L^{1}(\Lambda)$, one has $f\in L^{1}_{nv}(\Lambda)$ if and only if $D^{\omega}_{i}f=0$ for each $1\leq i \leq N$.
\end{enumerate}
According to (vi) , we consider $\mathcal{C}^{\infty}(\Lambda)$ provided with the semi-norm
\[\|f\|_{\#,p}=\left(\sum_{i=1}^{N}\|D^{\omega}_{i}f\|^{p}_{L^{p}(\Lambda)}\right)^{\frac{1}{p}},\; f\in \mathcal{C}^{\infty}(\Lambda),\; 1<p< \infty.\]
So topologized, $\mathcal{C}^{\infty}(\Lambda)$ is in general non-separated and non-complete. We denote by $W^{1,p}_{\#}(\Lambda)$ the separated completion of $\mathcal{C}^{\infty}(\Lambda)$ and by $I_{p}$ the canonical mapping of $\mathcal{C}^{\infty}(\Lambda)$ into its separated completion (see, e.g., \cite[Chapter II]{6}  and \cite[Page 29]{15}). Furthermore, as pointed out in \cite{31} (see also \cite[Chapter II]{6}), the distributional stochastic derivative $D^{\omega}_{i}, \; 1\leq i\leq N$, viewed as a mapping of $\mathcal{C}^{\infty}(\Lambda)$ into $L^{p}(\Lambda)$ extends to a unique continuous linear mapping, still denoted $D^{\omega}_{i}$, of $W^{1,p}_{\#}(\Lambda)$ into $L^{p}(\Lambda)$ such that $D^{\omega}_{i}I_{p}(v)=D^{\omega}_{i}v$ for 
$v\in \mathcal{C}^{\infty}(\Lambda)$ and 
\[\|u\|_{W^{1,p}_{\#}(\Lambda)}=\left( \sum_{i=1}^{N}\|D^{\omega}_{i}u\|^{p}_{L^{p}(\Lambda)}\right)^{\frac{1}{p}} \hbox{ for } u \in W^{1,p}_{\#}(\Lambda)\]
\par Moreover $W^{1,p}_{\#}(\Lambda)$ is a reflexive Banach space by the fact that the stochastic gradient $D_{\omega}=(D^{\omega}_{1}, \cdots, D^{\omega}_{N})$ sends isometrically $W^{1,p}_{\#}(\Lambda)$ into $L^{p}(\Lambda)^{N}$. By duality, the operator div$_{\omega} : L^{p'}(\Lambda)^{N} \rightarrow \left( W^{1,p}_{\#}(\Lambda)\right)^{'}\; (p'=\frac{p}{p-1})$ defined by 
\[ \langle \mathrm{div}_{\omega}u,v \rangle = -\langle u, D_{\omega}v\rangle=-\sum_{i=1}^{N} \int_{\Lambda}u_{i}D^{\omega}_{i}v d\mu\]
$(u=(u_{i}) \in L^{p'}(\Lambda)^{N},\; v\in W^{1,p}_{\#}(\Lambda))$ naturally extends the stochastic divergence operator in $\mathcal{C}^{\infty}(\Lambda)$; and we have the following fundamental lemma:

\begin{lemma}\label{l2.1}
Let $\mathrm{\textbf{v}}=(v_{i}) \in L^{p}(\Lambda)^{N}$ such that $\sum\limits_{i=1}^{N}\int_{\Lambda}\textbf{v$\cdot$g}d\mu=0$ \hbox{ for all } \textbf{g} $\in\mathcal{V}^{\omega}_{div}=\left\{\textbf{f}=(f_{i}) \in \mathcal{C}^{\infty}(\Lambda)^{N}\; : \mathrm{div}_{\omega}\textbf{f}=0 \right\},$ then $\mathrm{curl}_{\omega}\textbf{v}=O$ and there exists $u\in W^{1,p}_{\#}(\Lambda)$ such that $\textbf{v}=D_{\omega}u$

\end{lemma}
\begin{proof}
See \cite{8} Lemma 2.3 (or \cite{31} Proposition 1)
\end{proof}
\par In the sequel, we will put $W^{1,2}(\Lambda)=H^{1}(\Lambda)$ and $W^{1,2}_{\#}(\Lambda)=H^{1}_{\#}(\Lambda).$
\subsection{\textbf{Stochastic-two-scale convergence.}}We need the notion of weak stochastic-two-scale convergence in $L^{2}(Q\times\Lambda\times\Lambda_{0}),$ that is, given a fundamental sequence $E=(\varepsilon_{n})$ (integers $n\geq 0$) with $\varepsilon_{n}>0$ and $\varepsilon_{n}\rightarrow 0$ as $n\rightarrow \infty$, given $(v_{\varepsilon})_{\varepsilon\in E}\subset L^{2}(Q\times \Lambda\times \Lambda_{0})$ and given $v_{0} \in L^{2}\left(Q\times \Lambda\times\Lambda_{0}; L^{2}_{per}(Y)\right)$, where
\begin{equation}\nonumber
\begin{array}{l}
Q=\Omega\times (0,T),\; Y=Z\times \Theta \hbox{ with } \Theta=(0,1), \hbox{ and }\\
L^{2}_{per}(Y)=\left\{v\in L^{2}_{loc}\left(\mathbb{R}^{4} \right): \; v \; is\; Y\hbox{ -periodic}  \right\},
\end{array}
\end{equation}
the sequence $(v_{\varepsilon})_{\varepsilon\in E}$ is weakly stochastic-two-scale ($stoch$-2s in short) convergent in $L^{2}(Q\times \Lambda\times\Lambda_{0})$ to $v_{0}$ if, as $E\ni\varepsilon \rightarrow 0$, and for all $f\in L^{2}\left(Q; \mathcal{C}^{\infty}(\Lambda\times \Lambda_{0};\mathcal{C}_{per}(Y))\right),$             
\begin{equation}\label{p}
\begin{array}{l}
\int\!\int_{Q\times\Lambda\times\Lambda_{0}}v_{\varepsilon}f^{\varepsilon}dx dt d\mu d\mu^{0} \longrightarrow \\
\int\!\int\!\int_{Q\times\Lambda\times\Lambda_{0}\times Y}v_{0}(x, t, \omega,\omega_{0}, z, \tau)f(x, t, \omega, \omega_{0}, z, \tau)dx dt d\mu d\mu^{0}dz d\tau,
\end{array}
\end{equation}
where $\mathcal{C}_{per}(Y)=\left\{ v\in \mathcal{C}(\mathbb{R}^{4}) : v\;is\; Y\hbox{ -periodic }\right\}$ and
\begin{equation}
f^{\varepsilon}=f^{\varepsilon}(x, t, \omega, \omega_{0})=f\left(x, t, \hbox{\LARGE{$\tau$}} \left(\frac{x}{\varepsilon}\right)\omega, \hbox{\LARGE{$\tau$}}_{0}\left(\frac{t}{\varepsilon}\right)\omega_{0}, \frac{x}{\varepsilon^{2}}, \frac{t}{\varepsilon^{2}} \right)
\end{equation}
$(x, t, \omega, \omega_{0}) \in Q\times \Lambda\times\Lambda_{0},\; \varepsilon >0, \hbox{\LARGE{$\tau$}}$ being a 3-dimensional dynamical system on $\Lambda$ associated to a fixed invariant probability measure $\mu$ and $\hbox{\LARGE{$\tau$}}_{0}$ a 1-dimensional dynamical system on $\Lambda_{0}$ associated to a fixed invariant probability measure $\mu^{0}$. In particular, (\ref{p}) makes sense for $f\in L^{2}\left(Q\times \Lambda\times\Lambda_{0}; \mathcal{C}_{per}(Z)\right)\otimes \mathcal{C}_{per}(\Theta)$, or for $f\in \mathcal{C}\left(\overline{Q}; L^{2}\left(\Lambda\times\Lambda_{0}; L^{\infty}_{per}(Y) \right)\right)$.
\par Let us recall a fundamental compactness result, the proof can be found in \cite{31}.
\begin{lemma}
Let $E$ be a fundamental sequence, and let $(v_{\varepsilon})_{\varepsilon\in E}$ be a bounded sequence in $L^{2}(Q\times\Lambda\times\Lambda_{0})$. Then, a subsequence $E'$ can be extracted from $E$ such that the sequence $(v_{\varepsilon})_{\varepsilon \in E'}$ is a weakly stoch-2s convergent subsequence in $L^{2}(Q\times\Lambda\times\Lambda_{0})$.
\end{lemma}
\par Let $H^{1}_{\#}(\Lambda_{0})$ be the separated completion of $\mathcal{C}^{\infty}(\Lambda_{0})$. Following the proof in \cite{11,  12}, we have:
\begin{proposition}
Let $(v_{\varepsilon})_{\varepsilon\in E}$ be a bounded sequence in $L^{2}(Q\times\Lambda\times\Lambda_{0})$ such that $\left(\frac{\partial v_{\varepsilon}}{\partial t}  \right)_{\varepsilon \in E}$ is bounded in $L^{2}(Q\times\Lambda\times\Lambda_{0})$. Then, there exists a subsequence $E'\subset E$ and three functions $v_{0}\in H^{1}\left(0, T; L^{2}\left(\Omega\times\Lambda; L^{2}_{nv}\left(\Lambda_{0}; L^{2}_{per}(Z)\right)\right)\right), \; v^{s}_{1} \in L^{2}\left(Q\times\Lambda; L^{2}_{per}\left(Z; H^{1}_{\#}(\Lambda_{0}) \right)\right)$ and $v^{\sigma}_{1}\in L^{2}\left(Q\times\Lambda\times\Lambda_{0}; L^{2}_{per}\left( Z; H^{1}_{\#}(\Theta)\right) \right)$, such that as $E'\ni \varepsilon \rightarrow 0,$
\begin{equation}
v_{\varepsilon}\rightarrow v_{0}\; in \; L^{2}(Q\times\Lambda\times\Lambda_{0}) \hbox{ -weak stoch-2s }
\end{equation}
and 
\begin{equation}
\frac{\partial v_{\varepsilon}}{\partial t}\rightarrow \frac{\partial v_{0}}{\partial t}+ D^{\omega_{0}}v_{1}^{s} + \frac{\partial v_{1}^{\sigma}}{\partial \tau} \; in \; L^{2}(Q\times\Lambda\times\Lambda_{0}) \hbox{ -weak  stoch-2s}.
\end{equation}
\end{proposition}
\par Let $F$ be a Hilbert space, we define Hilbert space
\begin{equation}\nonumber
\begin{array}{l}
H_{\mathrm{curl}}(\Omega; F)=\left\{ v=(v_{i})\in L^{2}(\Omega; F)^{3}; \mathrm{curl} v \in L^{2}(\Omega; F)^{3} \right\} \hbox{ and} \\
 H_{\mathrm{div}}(\Omega; F)=\left\{v=(v_{i})\in L^{2}(\Omega; F)^{3}; \mathrm{div} v\in L^{2}(\Omega; F)  \right\}
\end{array}
\end{equation}
endowed with norms
\begin{equation}\nonumber
\begin{array}{l}
\|v\|_{\mathrm{curl}}=\left(\|v\|^{2}_{L^{2}(\Omega; F)^{3}}+\|\mathrm{curl}\; v\|^{2}_{L^{2}(\Omega; F)^{3}}  \right)^{\frac{1}{2}},\; v\in H_{\mathrm{curl}}(\Omega; F),\\
\|v\|_{\mathrm{div}}=\left(\|v\|^{2}_{L^{2}(\Omega; F)^{3}}+\|\mathrm{div}\; v\|^{2}_{L^{2}(\Omega; F)}  \right)^{\frac{1}{2}},\; v\in H_{\mathrm{div}}(\Omega; F),
\end{array}
\end{equation}
respectively. Since $\Omega$ is bounded then $\mathcal{C}^{1}(\overline{\Omega}; F)^{3}$ is dense in both spaces $H_{\mathrm{curl}}(\Omega; F)$ and $H_{\mathrm{div}}(\Omega; F)$. Furthermore, the mappings $v\rightarrow \gamma \wedge v$ (usual vector product) from $\mathcal{C}^{1}\big(\overline{\Omega}; F\big)^{3}$ into $\mathcal{C}^{1}(\partial \Omega; F)^{3}$ and $v\rightarrow v\cdot n$ (usual Euclidean product) form $\mathcal{C}^{1}\big(\overline{\Omega}; F\big)^{3}$ into 
$\mathcal{C}^{1}(\partial \Omega; F)$ extend by continuity to continuous linear mappings, still denoted $v\rightarrow v\wedge n$ and $v\rightarrow v \cdot n$, from $H_{\mathrm{curl}}(\Omega; F)$ into $L^{2}(\partial \Omega; F)^{3}$ and from $H_{\mathrm{div}}(\Omega; F)$ into $L^{2}(\partial \Omega; F)$, respectively. For simplicity, we will put $L^{2}(0, T; H_{\mathrm{curl}}(\Omega; F))\equiv H_{\mathrm{curl}}(Q; F)$ and $L^{2}(0, T; H_{\mathrm{div}}(\Omega; F))\equiv H_{\mathrm{div}}(Q; F).$
\begin{proposition}
Let $(v_{\varepsilon})_{\varepsilon \in E} = \big(v_{\varepsilon}^{1}, v^{2}_{\varepsilon}, v^{3}_{\varepsilon}\big)_{\varepsilon \in E}$ be a bounded sequence in 
$H_{\mathrm{curl}}\big(Q; L^{2}(\Lambda \times \Lambda_{0})\big)$. Then there exist a subsequence $E'\subset E$, functions
$$ v_{0}\in L^{2}\left(Q\times \Lambda_{0}; L^{2}_{per}(\Theta) \right)^{3}, \; v_{0}^{s}\in L^{2}\left(Q\times \Lambda_{0}; L^{2}_{per}\left(\Theta; H^{1}_{\#}(\Lambda)\right)\right)$$
$$and\;  v^{\sigma}_{0}\in L^{2}\left(Q\times \Lambda \times \Lambda_{0}; L^{2}_{per}\left(\Theta; H^{1}_{\#}(Z)\right)\right),$$ with 
$v_{0}+D_{\omega}v_{0}^{s}\in H_{\mathrm{curl}}\left(Q; L^{2}_{per}\left( \Theta; L^{2}(\Lambda_{0})  \right)   \right),$ such that, as $E' \ni \varepsilon \rightarrow 0$, one has 
\begin{equation}
v_{\varepsilon} \rightarrow v_{0}+D_{\omega}v_{0}^{s}+D_{z}v_{0}^{\sigma}\; in\; L^{2}(Q\times \Lambda \times \Lambda_{0})^{3}\hbox{ -weak stoch-2s}.
\end{equation}
Moreover, we have for $\varphi = (\varphi_{i}) \in \mathcal{D}(Q)^{3}, \; g_{0}\in \mathcal{C}^{\infty}(\Lambda_{0}), \; g\in \mathcal{C}^{\infty}(\Lambda)\cap L^{2}_{nv}(\Lambda)$ and $\psi \in \mathcal{C}_{per}(\Theta),$
\begin{equation}
\begin{array}{l}
\lim_{E'\ni \varepsilon \rightarrow 0} \int\int_{Q\times \Lambda \times \Lambda_{0}}\mathrm{curl} v_{\varepsilon}\cdot\varphi g(\omega)g_{0}\left( \tau_{0}\left( \frac{t}{\varepsilon}\right)\omega_{0}\right)\psi\left(\frac{t}{\varepsilon^{2}}\right)dx dt d\mu d\mu^{0}\\
= \int\int\int_{Q\times \Lambda\times \Lambda_{0}\times \Theta}\mathrm{curl}(v_{0}+D_{\omega}v^{s}_{0})\cdot\varphi g(\omega)g_{0}(\omega_{0})\psi(\tau)dx dt d\mu d\mu^{0}d\tau.
\end{array}
\end{equation}
\end{proposition}

\begin{proposition}
Let $(v_{\varepsilon})_{\varepsilon \in E} = \big(v_{\varepsilon}^{1}, v^{2}_{\varepsilon}, v^{3}_{\varepsilon}\big)_{\varepsilon \in E}$ be a bounded sequence in 
$H_{\mathrm{div}}\big(Q; L^{2}(\Lambda \times \Lambda_{0})\big)$. Then there exist a subsequence $E'\subset E$, a function $ v_{0}\in H_{\mathrm{div}}\left(\Omega; L^{2}\left(\Lambda\times \Lambda_{0}; L^{2}_{per}(\Theta)\right) \right)$, and a function $v_{1} \in L^{2}\left(Q\times \Lambda\times \Lambda_{0}; L^{2}_{per}\left(\Theta; L^{2}_{\#}(Z)^{3}\right)\right)$
with $\mathrm{div}_{\omega}v_{0}=0$ and $\mathrm{div}_{z}v_{1}=0$ such that, as $E'\ni \varepsilon \rightarrow 0$, one has
\begin{equation}
v_{\varepsilon} \rightarrow v_{0}+v_{1}\; in\; L^{2}(Q\times\Lambda\times \Lambda_{0})^{3}\hbox{ -weak stoch-2s},
\end{equation}
and for $\varphi \in \mathcal{D}(Q),\; g_{0} \in \mathcal{C}^{\infty}(\Lambda_{0}), \; g\in\mathcal{C}^{\infty}(\Lambda)\cap L^{2}_{nv}(\Lambda)$ and 
$\psi\in \mathcal{C}_{per}(\Theta)$
\begin{equation}
\begin{array}{c}
\lim_{E'\ni \varepsilon \rightarrow 0} \int\int_{Q\times \Lambda \times \Lambda_{0}}\mathrm{div}\; v_{\varepsilon}\varphi g(\omega)g_{0}\left( \hbox{\LARGE{$\tau$}}_{0}\left( \frac{t}{\varepsilon}\right)\omega_{0}\right)\psi\left(\frac{t}{\varepsilon^{2}} \right)dx dt d\mu d\mu^{0}\\
= \int\int\int_{Q\times \Lambda\times \Lambda_{0}\times \Theta}\mathrm{div}\; v_{0}\varphi\otimes g \otimes g_{0}\otimes \psi(x, t, \omega, \omega_{0}, \tau)dx dt d\mu d\mu^{0}d\tau,
\end{array}
\end{equation}
where $L^{2}_{\#}(Z)=\left\{v\in L^{2}_{per}(Z); \int_{Z}vdz=0  \right\}.$
\end{proposition}

\section{EXISTENCE AND UNIQUENESS OF SOLUTION}
\par We  prove the existence and uniqueness of the result and provide some estimations necessary for the study of convergence of the solutions of problem (\ref{1})-(\ref{7}). In what follows the space $\Lambda$ is such that the Sobolev spaces in use are separable. 

\subsection{\textbf{Some fundamental results}.} The following result is fundamental for the existence result for  fixed $\varepsilon > 0$, of a couple of solution
of problem (\ref{1})-(\ref{7}) and characterization of stochastic two scale convergence limit of nonlinear conductivity.  

\begin{lemma}
Let $a:\mathbb{R}\longrightarrow \mathbb{R}$, a continuous function verifying $|a(y)|\leq c\big(|y|^{p-1}\big), \; p>1,\; c\in \mathbb{R}_{+}^{*}$ for all $y$. Let $p'$ be such that $\frac{1}{p}+\frac{1}{p'}=1$ then $a(\cdot)$ from $L^{p}(\mathbb{R}) \longrightarrow L^{p'}(\mathbb{R})$ is continuous, $L^{p}(\mathbb{R})$ and $L^{p'}(\mathbb{R})$ with their respective norms.
\end{lemma}
\begin{proof}
Let $\left(u_{n}\right)_{n\in \mathbb{N}}$ be such that 
\begin{equation}\label{17}
u_{n}\longrightarrow u\;\hbox{ in }\; L^{p}(\mathbb{R}).
\end{equation}
Let us show that $$a(u_{n})\longrightarrow a(u),\; in\; L^{p'}(\mathbb{R}).$$
$|a(y)|\leq c\big(|y|^{p-1}\big),\; p>1,\; c\in\mathbb{R}_{+}^{*}$ and (\ref{17}) leads to $a(u_{n})$ is bounded in $L^{p'}(\mathbb{R})$. Let $\sigma_{i},\; 1\leq i\leq 4$ be one to one and non-decreasing functions from $\mathbb{N}$ to $\mathbb{N}$. Then $\big(a(u_{\sigma_{1}(n)})\big)_{n}$ is a subsequence of $(a(u_{n}))_{n}$ bounded in $L^{p'}(\mathbb{R})$. $\big(a(u_{\sigma_{1}(n)})\big)_{n}$ therefore has a sub-sequence $\big(a(u_{\sigma_{2}(\sigma_{1}(n))})\big)_{n}$ converging to $f$. From (\ref{17}), the subsequence $\big(u_{\sigma_{2}(\sigma_{1}(n))}\big)_{n}$ of $(u_{n})_{n}$ converges  to $u$ in $L^{p}(\mathbb{R})$. Theorem 4.9 in page 94 of \cite{10} gives the existence of a subsequence $\big(u_{\sigma_{3}(\sigma_{2}(\sigma_{1}(n)))}\big)_{n}$ such that 
$$u_{\sigma_{3}(\sigma_{2}(\sigma_{1}(n)))} \longrightarrow u \; a.e \; in \; \mathbb{R}.$$
$a: \mathbb{R} \longrightarrow \mathbb{R}$ being continuous, $a\left(u_{\sigma_{3}(\sigma_{2}(\sigma_{1}(n)))}\right) \longrightarrow a(u) \; a.e\; in\; \mathbb{R}$. As $\big(a(u_{\sigma_{2}(\sigma_{1}(n))})\big)_{n}$ converges to $f \; in \; L^{p'}(\mathbb{R}),\; \big(a(u_{\sigma_{3}(\sigma_{2}(\sigma_{1}(n)))})\big)_{n}$ converges to $f$ in $L^{p'}(\mathbb{R})$.Theorem 4.9 in page 94 of \cite{10} yields the existence of a subsequence $\big(a(u_{\sigma_{4}(\sigma_{3}(\sigma_{2}(\sigma_{1}(n))))})\big)_{n}$ such that 
$$a(u_{\sigma_{4}(\sigma_{3}(\sigma_{2}(\sigma_{1}(n))))}) \longrightarrow f \; a.e\; in\; \mathbb{R}$$
and as $a\left(u_{\sigma_{3}(\sigma_{2}(\sigma_{1}(n)))}\right) \longrightarrow a(u) \; a.e\; in\; \mathbb{R}$, we obtain $f=a(u)$ and $\big(a(u_{\sigma_{2}(\sigma_{1}(n))})\big)_{n}$ converges to $a(u) \; in\; L^{p'}(\mathbb{R})$.
\par Therefore, any subsequence of $(a(u_{n}))_{n}$ admits a subsequence converging in $L^{p'}(\Omega)$ to $a(u)$; we conclude that $(a(u_{n}))_{n}$ converges in $L^{p'}(\mathbb{R})$ to $a(u)$ and $a(\cdot)$ is continuous from $L^{p}(\mathbb{R})$ to $L^{p'}(\mathbb{R})$.
\end{proof}
\subsection{\textbf{Trace results and consequences.}} We refer to \cite{22, 27, 32} for details. Properties of $\sigma((i)-(iv))$ allow to justify traces we are in need of.
\par \textbf{First case:} $u\in \big[\mathcal{C}\left(\overline{Q}\right) \otimes \mathcal{C}^{\infty}(\Lambda)\otimes \mathcal{C}^{\infty}_{per}(Z) \big]^{3}$.
\par As $u\in \big[\mathcal{C}\left(\overline{Q}\right) \otimes \mathcal{C}^{\infty}(\Lambda)\otimes \mathcal{C}^{\infty}_{per}(Z) \big]^{3}$ then for fix 
$(x, t, z)$, the function $\omega \longrightarrow u(x, t,\omega, z)$ is an element of $\mathcal{C}^{\infty}(\Lambda)^{3}$, and with properties $(ii)$ and $(iv)$ of $\sigma$; for all $\left(x', t'\right)$ and almost all $\omega$, the function $\left(\omega', z' \right) \longrightarrow \sigma_{i}\left(x', \omega', z', u(x, t, \omega, z)  \right)$ is  measurable on $\Lambda_{\omega'}\times \mathbb{R}_{z'}^{3},$ with $$\big\|\sigma_{i}\big(x',.,.,u(x,t,\omega,z)\big)\big\|_{L^{\infty}\left(\Lambda_{\omega'} \times \mathbb{R}_{z'}^{3}\right)} \leq c_{1}||u(x,t,.,z)||_{L^{\infty}\left(\Lambda_{\omega'}\times \mathbb{R}_{z'}^{3}  \right)}<\infty.$$ 
Then for all $\big(x,t,z,x',t'  \big) \in \overline{Q}\times \mathbb{R}^{3}\times\overline{Q}$ for almost all $\big( \omega', z' \big) \in \Lambda_{\omega'} \times \mathbb{R}^{3}_{z'}$, the function $\omega \longrightarrow \sigma_{i}\big( x', \omega',z',u(x,t,\omega,z) \big) \in \mathcal{C}^{\infty}\left( \Lambda_{\omega}\right).$ For all $(x, t)$ and for almost all $\omega$ as $u(x, t, \omega, .)\in \mathcal{C}^{\infty}_{per}(Z)$ property $(iv)$ of $\sigma$ leads to $\big(x,t,\omega,z, x',t',\omega',z' \big) \longrightarrow \sigma_{i}\big(x',\omega', z', u(x,t,\omega,z) \big)$ from $\overline{Q}\times \Lambda \times \mathbb{R}^{3}\times \overline{Q}\times\Lambda\times \mathbb{R}^{3}$ to $\mathbb{R}$ is an element of $\mathcal{C}\big( \overline{Q}\times \overline{Q}\times \mathbb{R}^{3}_{z}; \mathcal{C}^{\infty}\big(\Lambda; L^{\infty}\big( \Lambda\times \mathbb{R}^{3}_{z'} \big)\big)\big)$. Fixing 
$\big( x, t, x', t' \big)$;
$$\big(\omega, z,\omega',z' \big) \longrightarrow \sigma_{i}\big(x',\omega', z', u(x,t,\omega,z) \big) \in \mathcal{C}^{\infty}\big(\mathbb{R}^{3}_{z};
\mathcal{C}^{\infty}\big(\Lambda; L^{\infty}\big(\Lambda\times \mathbb{R}^{3}_{z'}\big)\big)\big)$$
and the trace are defined naturally by $(\omega, z)\longrightarrow \sigma_{i}\big(x', \omega, z, u(x,t,\omega,z) \big) \equiv \sigma_{i}(x',.,., u(x,t,.,.)) \in L^{\infty}\big(\Lambda \times\mathbb{R}^{3}_{z}\big)$.
$$\big(x,t,x',t'\big) \longrightarrow \sigma_{i}\left(x',.,.,u(x,t,.,.) \right)\in \mathcal{C}\big(\overline{Q}\times\overline{Q}; L^{\infty}\big( \Lambda \times \mathbb{R}^{3}_{z}\big)\big)$$
and the trace $(x,t)\longrightarrow \sigma_{i}(x,.,.,u(x,t,.,.)) \in \mathcal{C}\big(\overline{Q};L^{\infty}\big(\Lambda \times \mathbb{R}^{3}_{z}\big) \big)$ is well defined.\\ Therefore  $\sigma^{\varepsilon}_{i}\left(-, u^{\varepsilon} \right): Q\times \Lambda \longrightarrow \mathbb{R},\; (x, t, \omega) \longrightarrow \sigma_{i}\big(x, \hbox{\LARGE{$\tau$}}\left(\frac{x}{\varepsilon}\right)\omega, \frac{x}{\varepsilon^{2}}, u\left(x, t, \hbox{\LARGE{$\tau$}}\left(\frac{x}{\varepsilon}\right)\omega, \frac{x}{\varepsilon^{2}}\right)\big)$ is meaningful as element of $L^{\infty}\left( Q\times \Lambda \right)$ and due to (\ref{6}) we have:
\begin{equation}
\begin{array}{l}
\big\|\sigma^{\varepsilon}_{i}(-, u^{\varepsilon})-\sigma^{\varepsilon}_{i}(-, v^{\varepsilon})\big\|_{L^{2}(Q\times \Lambda)}\leq \big\|u^{\varepsilon}- v^{\varepsilon}\big\|_{\mathbb{L}^{\infty}(Q\times \Lambda)} \leq\|u-v\|_{\mathcal{C}\big(\overline{Q}; L^{\infty}\left(\Lambda; L^{\infty}_{per}(Z)\right)\big)^{3}}.
\end{array}
\end{equation}
Hence $\big[\mathcal{C}\big(\overline{Q}\big) \otimes \mathcal{C}^{\infty}(\Lambda)\otimes \mathcal{C}^{\infty}_{per}(Z)\big]^{3} \longrightarrow L^{\infty}\big(Q\times \Lambda\big), u\longrightarrow \sigma^{\varepsilon}_{i}\left(-, u^{\varepsilon} \right) $ extends by continuity to a function still denoted with same notation from $\mathcal{C}\big(\overline{Q}; L^{2}\big(\Lambda; L^{\infty}_{per}(Z)\big)\big)$ to $L^{2}\left(Q\times \Lambda \right)$.
\par \textbf{Second case:} $u\in \big[\mathcal{C}\big(\overline{Q}\otimes \mathcal{C}^{\infty}(\Lambda)\big)\big]^{3}$.
\par Proceed as in the first case and realize that the function $\big(x, t, \omega, x', t', \omega', z  \big) \longrightarrow \sigma_{i}\big(x', \omega', z, u(x, t, \omega) \big)$ from $\overline{Q}\times \Lambda\times \overline{Q}\times \Lambda\times \mathbb{R}^{3}$ to $\mathbb{R}$ is an element of\\ $\mathcal{C}\big(\overline{Q}\;\times\; \overline{Q}; \mathcal{C}^{\infty}\big(\Lambda;\; L^{\infty}\big(\Lambda\times \mathbb{R}^{3}\big)\big)  \big)$. Fixing 
$\big(x, t, x', t'  \big); \big(\omega, \omega', z  \big)\longrightarrow \sigma_{i}\big(x', \omega', z, u(x,t,\omega) \big) \in \mathcal{C}^{\infty}\big(\Lambda; L^{\infty}\big(\Lambda\times \mathbb{R}^{3}\big)\big)$ and we have the trace $(\omega, z) \longrightarrow \sigma_{i}\big(x', \omega, z, u(x, t, \omega)  \big) \equiv \sigma_{i}\big(x',.,.,u(x,t,.) \big)\in L^{\infty}\left(\Lambda \times \mathbb{R}^{3}  \right).$ We also have $\big(x, t, x', t'  \big) \longrightarrow \sigma_{i}\big(x',.,.,u(x,t,.)\big) \in \mathcal{C}\big(\overline{Q}\; \times \; \overline{Q}; \; L^{\infty}\big(\Lambda \times \mathbb{R}^{3}\big)  \big)$ and naturally define the trace $(x, t)\longrightarrow \sigma_{i}(x,.,.,u(x,t,.)) \equiv \sigma_{i}\left(-,u\right)\in \mathcal{C}\big(\overline{Q}; \; L^{\infty}\big(\Lambda \times \mathbb{R}^{3}  \big)  \big)$. Then $\sigma_{i}^{\varepsilon}\left(-, u \right):Q\times \Lambda \longrightarrow \mathbb{R}, \left(x, t, \omega \right) \longrightarrow \sigma_{i}\big(x, \hbox{\LARGE{$\tau$}}\left(\frac{x}{\varepsilon}\right)\omega, \frac{x}{\varepsilon^{2}}, u(x, t, \omega) \big)$ is defined as an element of $L^{\infty}\left(Q\times \Lambda \right)$. We deduce from (\ref{6}) that 
$$\big\|\sigma^{\varepsilon}_{i}(-, u)-\sigma^{\varepsilon}_{i}(-,v) \big\|_{L^{2}(Q\times \Lambda)}\leq c_{1}\|u-v\|_{\mathbb{L}^{2}(Q\times \Lambda)}, \hbox{ where } \mathbb{L}^{2}(\Omega\times \Lambda) =L^{2}(\Omega\times \Lambda)^{3}.$$
Therefore, 
$$\big[\mathcal{C}\big(\overline{Q}\big)\otimes \mathcal{C}^{\infty}(\Lambda)\big]^{3}\longrightarrow L^{\infty}(Q\times \Lambda),\; u\longrightarrow \sigma^{\varepsilon}_{i}(-,\; u)$$
extends by continuity to a function still denoted the same way from $\mathbb{L}^{2}\left( Q\times \Lambda\right)$ to $L^{2}\left( Q\times \Lambda\right)$. In what follows we set $\mathcal{M}^{\varepsilon}_{1}(u)=\sigma^{\varepsilon}(-, u)$ for $u\in \mathbb{L}^{2}\left(Q\times\Lambda \right)$ and $\mathcal{M}^{\varepsilon}(U)=\left(\mathcal{M}^{\varepsilon}_{1}(u_{1}), 0  \right)$ for $U=(u_{1},\; u_{2}) \in \mathbb{L}^{2}\left(Q\times \Lambda\right)^{2}$
\par We also have the following result:
\begin{lemma}\label{l3.2}
Let $\varphi \in \mathcal{D}(Q)^{3},\; \phi \in \mathcal{D}\left(\Omega\times ]0,\; T[; \mathcal{C}^{\infty}\left(\Lambda\right)\right),\; \psi \in \mathcal{D}\big(\Omega\times ]0,\; T[; \big(\mathcal{C}^{\infty}(\Lambda); \mathcal{C}^{\infty}_{per}(Z)  \big)  \big), \; \\ \theta \in \mathcal{C}\big(\overline{Q}; \big(\mathcal{C}^{\infty}(\Lambda);\; \mathcal{C}^{\infty}_{per}(Z)  \big)  \big)^{3}$, define the test function $v^{\varepsilon}=(v^{\varepsilon}_{i}),\; 1\leq i\leq 3$ by:
$$v^{\varepsilon}_{i}=\varphi_{i}+\varepsilon \partial_{x_{i}}(\phi)^{\varepsilon}+ \varepsilon^{2}\partial_{x_{i}}(\psi)^{\varepsilon}+ r(\theta_{i})^{\varepsilon};$$
where $r$ is a non null real. We have: $\sigma_{i}(.,.,.,v(.,.,.,.)) \in \mathcal{C}\big(\overline{Q}; \; L^{2}\big(\Lambda; L^{\infty}_{per}(Z)\big)\big)$ with 
$$\sigma^{\varepsilon}_{i}(.,.,.,v^{\varepsilon}(.,.,.,.))\longrightarrow \sigma_{i}\left(.,.,.,\varphi_{i}+\partial^{\omega}_{i}(\phi)+\partial_{z_{i}}(\psi)+r(\theta_{i} \right)\; in \; L^{2}(Q\times \Lambda\times Z)-strong.$$
\end{lemma}
\begin{proof}
Because,
\begin{equation}\nonumber
\begin{array}{l}
\partial_{x_{i}}(\phi)^{\varepsilon}(x, \omega, t)=\partial_{x_{i}}\phi\left(x, \hbox{\LARGE{$\tau$}}\left(\frac{x}{\varepsilon} \right)\omega, t \right)+ \frac{1}{\varepsilon}\partial^{\omega}_{i}\phi\left(x, \hbox{\LARGE{$\tau$}}\left(\frac{x}{\varepsilon} \right)\omega, t \right);\\
\partial_{x_{i}}(\psi)^{\varepsilon}(x, \omega, t)=\partial_{x_{i}}\psi\left(x, \hbox{\LARGE{$\tau$}}\left(\frac{x}{\varepsilon} \right)\omega,\frac{x}{\varepsilon^{2}}, t \right)+ \frac{1}{\varepsilon}\partial^{\omega}_{i}\psi\left(x, \hbox{\LARGE{$\tau$}}\left(\frac{x}{\varepsilon} \right)\omega,\frac{x}{\varepsilon^{2}}, t \right)+\frac{1}{\varepsilon^{2}}\partial_{z_{i}}\psi\left(x, \hbox{\LARGE{$\tau$}}\left(\frac{x}{\varepsilon}\right)\omega, \frac{x}{\varepsilon^{2}}, t \right);
\end{array}
\end{equation}
then
\begin{equation}\nonumber
\begin{array}{l}
\varepsilon\partial_{x_{i}}(\phi)^{\varepsilon}(x, \omega, t)=\varepsilon\partial_{x_{i}}\phi\left(x, \hbox{\LARGE{$\tau$}}\left(\frac{x}{\varepsilon}\right)\omega, t \right)+ \partial^{\varepsilon}_{i}\phi\left(x, \hbox{\LARGE{$\tau$}}\left(\frac{x}{\varepsilon}\right)\omega, t\right); \\
\varepsilon^{2}\partial_{x_{i}}(\psi)^{\varepsilon}(x, \omega, t)=\varepsilon^{2}\partial_{x_{i}}\psi\left(x, \hbox{\LARGE{$\tau$}}\left(\frac{x}{\varepsilon} \right)\omega,\frac{x}{\varepsilon^{2}}, t \right)+ \varepsilon\partial^{\omega}_{i}\psi\left(x, \hbox{\LARGE{$\tau$}}\left(\frac{x}{\varepsilon} \right)\omega,\frac{x}{\varepsilon^{2}}, t \right)+\partial_{z_{i}}\psi\left(x, \hbox{\LARGE{$\tau$}}\left(\frac{x}{\varepsilon}\right)\omega, \frac{x}{\varepsilon^{2}}, t \right).
\end{array}
\end{equation}
From condition (\ref{6}) we get as $\varepsilon \longrightarrow 0$,
$$\sigma^{\varepsilon}_{i}(., ., .,v^{\varepsilon}(., ., ., .))\longrightarrow \sigma_{i}\left(., ., ., \varphi_{i}+\partial^{\omega}_{i}(\phi)+\partial_{z_{i}}(\psi)+r(\theta_{i})\right)$$ in $L^{2}\left(Q\times \Lambda \times Z  \right)-$strong.
\end{proof}
\subsection{\textbf{A priori estimates.}} We recall some notations and facts and refer to \cite{17} for details. Set $U^{\varepsilon}=\{E^{\varepsilon}, \; H^{\varepsilon}\}$, $G^{\varepsilon}=\{F^{\varepsilon},\; 0\}$, define the following:
$$\mathcal{N}^{\varepsilon}: \mathbb{L}^{2}\left(\Omega \times \Lambda  \right)^{2}\longrightarrow \mathbb{L}^{2}\left(\Omega \times \Lambda \right)^{2}, \; \Phi=\{\varphi, \theta\}\mapsto \mathcal{N}^{\varepsilon}\Phi=\{\eta^{\varepsilon}\varphi, \mu^{\varepsilon}\theta\}.$$
Set 
\begin{equation}\nonumber
\begin{array}{l}
\mathcal{A}\Phi=(-\mathrm{curl}\; \phi,\; \mathrm{curl}\;\varphi) \hbox{ for all } \Phi=(\varphi, \phi) \in \mathcal{D}(\mathcal{A}), \hbox{ with}\\
\mathcal{D}(\mathcal{A})=\bigg\{(\varphi, \phi) \in \mathbb{L}^{2}\left(\Omega\times \Lambda \right)^{2}; (\mathrm{curl}\; \varphi,\; \mathrm{curl} \;\phi)\in \mathbb{L}^{2}\left( \Omega\times \Lambda\right)^{2} \hbox{ and }  \gamma \wedge \varphi=O\; on\; \partial\Omega\times \Lambda  \bigg\}.
\end{array}
\end{equation}
The first two equations of (\ref{1}) imply the following variational formulation
\begin{equation}\nonumber
\begin{array}{l}
\big[\mathcal{N}^{\varepsilon}\partial_{t}U_{\varepsilon}(t), \; V \big]+ \big[\mathcal{A}U_{\varepsilon}(t), \; V\big]+ \big[\mathcal{M}^{\varepsilon}U_{\varepsilon}(t),\; V \big]= \big[G_{\varepsilon}(t),\; V\big],\; 0<t<T, \\
\hbox{ for all } V\in \mathcal{D}(\mathcal{A}), \hbox{ where } \mathcal{N}^{\varepsilon}, \; \mathcal{A}, \; G_{\varepsilon} \hbox{ are  defined as above; see also \cite{17}} \\
\hbox{ and } \mathcal{M}^{\varepsilon} \hbox{ is as defined in  the  previous section.}
\end{array}
\end{equation}
\par Taking $V=U_{\varepsilon}(t),\; 0<t<T,$ we get from $\mathcal{N}^{\varepsilon}, \; \mathcal{A}$ and $\mathcal{M}^{\varepsilon}$ properties that:
$$\frac{1}{2}\partial_{t}\left(E_{\varepsilon}(t),\; D_{\varepsilon}(t)\right)+\frac{1}{2}\partial_{t}\left(H_{\varepsilon}(t),\; B_{\varepsilon}(t)\right) +\delta \|E_{\varepsilon}(t)\|^{2}_{L^{2}(\Omega \times \Lambda)^{3}}\leq \big[G_{\varepsilon}(t),\; U_{\varepsilon}(t)\big].$$
We integrate from $0$ to $t$ and we obtain the following estimation
\begin{equation}\nonumber
\begin{array}{l}
\frac{1}{2}\left(E_{\varepsilon}(t),\; D_{\varepsilon}(t)\right) +\frac{1}{2}\left(H_{\varepsilon}(t), \; B_{\varepsilon}(t)\right) + \delta\int_{0}^{t}\|E_{\varepsilon}(s)\|^{2}_{L^{2}(\Omega\times \Lambda)^{3}}ds \leq \\
\sqrt{\frac{2}{\delta}}\bigg(\int^{T}_{0}\|F_{\varepsilon}(t)\|^{2}_{L^{2}(\Omega\times \Lambda)^{3}}dt  \bigg)^{\frac{1}{2}} \times \sqrt{\frac{\delta}{2}}\bigg(\int^{T}_{0} \|E_{\varepsilon}(t)\|^{2}_{L^{2}(\Omega\times \Lambda)^{3}}dt  \bigg)^{\frac{1}{2}}\\ + \frac{1}{2}\left(E_{0}^{\varepsilon}(t),\; D^{\varepsilon}_{0}(t) \right)+ \frac{1}{2}\left( H^{\varepsilon}_{0}(t), \; B^{\varepsilon}_{0}(t)\right). 
\end{array}
\end{equation}
Therefore
\begin{equation}\nonumber
\begin{array}{l}
\frac{1}{2}\left(E_{\varepsilon}(t),\; D_{\varepsilon}(t)\right) +\frac{1}{2}\left(H_{\varepsilon}(t), \; B_{\varepsilon}(t)\right) +\delta\int_{0}^{t}\|E_{\varepsilon}(s)\|^{2}_{L^{2}(\Omega\times \Lambda)^{3}}ds \leq \\
\frac{2}{\delta}\int_{0}^{T}\|F_{\varepsilon}(t)\|^{2}_{L^{2}(\Omega\times \Lambda)^{3}}dt +\frac{1}{2}\left(E^{\varepsilon}_{0}(t),\; D_{0}^{\varepsilon}(t)\right) +\frac{1}{2}\left(H_{0}^{\varepsilon}(t),\; B_{0}^{\varepsilon}(t)\right)
\end{array}
\end{equation}
So
\begin{equation}\label{19}
\begin{array}{l}
\left(E_{\varepsilon}(t),\; D_{\varepsilon}(t)\right) +\left(H_{\varepsilon}(t), \; B_{\varepsilon}(t)\right) + \delta\int_{0}^{t}\|E_{\varepsilon}(s)\|^{2}_{L^{2}(\Omega\times \Lambda)^{3}}ds \leq \\
\frac{4}{\delta}\int_{0}^{T}\|F_{\varepsilon}(t)\|^{2}_{L^{2}(\Omega\times \Lambda)^{3}}dt +\left(E^{\varepsilon}_{0}(t),\; D_{0}^{\varepsilon}(t)\right) +\left(H_{0}^{\varepsilon}(t),\; B_{0}^{\varepsilon}(t)\right)
\end{array}
\end{equation}
\begin{remark}
Inequality (\ref{19}) suggests we can find a solution of (\ref{1}) as a distribution $U_{\varepsilon} \in \mathcal{D}'\left(]0,\; T[,\; \mathcal{D}(\mathcal{A})  \right)$ such that $U_{\varepsilon}\in L^{\infty} \big(0, T;\; \mathbb{L}^{2}\left(\Omega\times \Lambda\right)^{2}  \big)$
\end{remark}
\par As another consequence of estimation (\ref{19}) one has the following lemma:
\begin{lemma}\label{l3.3}
From hypothesis  on $F_{\varepsilon}$, $E^{0}_{\varepsilon}$, $H^{0}_{\varepsilon}$ (see Section 1), it results from estimation (\ref{19}) that
\begin{equation}\nonumber
\begin{array}{l}
\|U^{\varepsilon}\|_{L^{2}\left(0, T; H_{\mathrm{curl}}(\Omega; L^{2}(\Lambda))^{2}\right)} \leq C,\\
\|U^{\varepsilon}\|_{L^{\infty}\left(0, T; \mathbb{L}^{2}(\Omega\times \Lambda)^{2}  \right)} \leq C,\\
\|\mathcal{M}^{\varepsilon}\left(U^{\varepsilon} \right)\|_{L^{2}\left(0, T ; \mathbb{L}^{2}(\Omega\times \Lambda)^{2} \right)} \leq C.
\end{array}
\end{equation}
\end{lemma}
\begin{proof}
Hypothesis on $F_{\varepsilon}$, $H^{0}_{\varepsilon}$, $H^{0}_{\varepsilon}$ and (\ref{19}) yield $\int^{T}_{0}\|E_{\varepsilon}(t)\|^{2}_{L^{2}(\Omega\times \Lambda)^{3}}dt \leq C$, \, $\|E^{\varepsilon}\|_{L^{\infty}\left(0, T; \mathbb{L}^{2}(\Omega\times \Lambda)  \right)} \leq C$, $\|H^{\varepsilon}\|_{L^{\infty}\left(0, T; \mathbb{L}^{2}(\Omega\times \Lambda)  \right)} \leq C$, ($C$ a constant independent from $\varepsilon$) and next  \\ 
$\|\mathcal{M}^{\varepsilon}\left(U^{\varepsilon} \right)\|_{L^{2}\left(0, T ; \mathbb{L}^{2}(\Omega\times \Lambda)^{2} \right)} \leq C$, 
(see \ref{6}). Proceeding as \cite{17} we get
\begin{equation}\nonumber
\begin{array}{l}
\partial_{t}\left(\mathcal{N}^{\varepsilon}\partial_{t}U_{\varepsilon}(t),\; \partial_{t}U_{\varepsilon}(t)  \right)+ \left( D_{\xi}\mathcal{M}^{\varepsilon}\left(U_{\varepsilon}(t)\right)\partial_{t}U_{\varepsilon}(t), \partial_{t}U_{\varepsilon}(t)\right)\\
=\left(\partial_{t}G^{\varepsilon}(t), \partial_{t}U_{\varepsilon}(t) \right).
\end{array}
\end{equation}
Because $\left(D_{\xi}\mathcal{M}^{\varepsilon}\left(U_{\varepsilon}(t)\right)\partial_{t}U_{\varepsilon}(t), \partial_{t}U_{\varepsilon}(t)\right) \geq 0$ (see \ref{8}) the following inequality is obtained after integration with respect to the variable $t$ and  integration by parts in the second member:
\begin{equation}\nonumber
\begin{array}{l}
\frac{1}{2}\left(\partial_{t}E_{\varepsilon}(T_{1}),\; \partial_{t}D_{\varepsilon}(T_{1}) \right) + \frac{1}{2}\left(\partial_{t}H_{\varepsilon}(T_{1}),\; \partial_{t}B_{\varepsilon}(T_{1})\right)\leq \\
\|\partial_{t}F^{\varepsilon}(T_{1})\|^{2}_{L^{2}(\Omega\times \Lambda)^{3}}\|E^{\varepsilon}(t)\|^{2}_{L^{2}(\Omega\times \Lambda)^{3}}+\|\partial_{t}F^{\varepsilon}(0)\|^{2}_{L^{2}(\Omega\times \Lambda)^{3}}\|E^{\varepsilon}(0)\|^{2}_{L^{2}(\Omega\times \Lambda)^{3}}\\
\bigg( \int^{T_{1}}_{0}\big\|\partial^{2}_{t}F^{\varepsilon}(t) \big\|^{2}_{L^{2}(\Omega\times \Lambda)^{3}}dt \bigg)^{\frac{1}{2}} \times \bigg( \int^{T_{1}}_{0}\big\|E_{\varepsilon}(t) \big\|^{2}_{L^{2}(\Omega\times \Lambda)^{3}}dt \bigg)^{\frac{1}{2}} \\
\frac{1}{2}\left(\partial_{t}E_{\varepsilon}(0),\; \partial_{t}D_{\varepsilon}(0) \right) + \frac{1}{2}\left(\partial_{t}H_{\varepsilon}(0),\; \partial_{t}B_{\varepsilon}(0)\right), \; \hbox{for all } T_{1}\in (0,\; T).
\end{array}
\end{equation}
The hypothesis on the source $F_{\varepsilon}$ and the strong convergence of $U_{\varepsilon}(0)$ leads to $\partial_{t}E_{\varepsilon}$ and $\partial_{t}H_{\varepsilon}$ are bounded in $L^{\infty}\big(0,T; \; L^{2}(\Omega\times \Lambda)^{3}\big)$. Taking in account the first two equations in (\ref{1}) we also get curl $E_{\varepsilon}$ bounded in $L^{\infty}\big(0, T; \; L^{2}(\Omega\times \Lambda)^{3}\big)$ and $\mathrm{curl }H_{\varepsilon}$ bounded
 in $L^{2}\big(0,T; \; L^{2}(\Omega\times \Lambda)^{3}\big)$; which ends the proof.
\end{proof}
\subsection{\textbf{Approximated problems.}} We intend to prove the existence of solution of problem (\ref{1}) by Faedo-galerkin method combined with Stone-Weierstrass theorem \cite{20, 28}. We use as notation [,] for both scalar product associated to $\mathcal{D}(\mathcal{A})$ norm and duality between $\mathcal{D}(\mathcal{A})$ and $\mathcal{D}(\mathcal{A})'$ and (,) is used for scalar product in $\mathbb{L}^{2}(\Omega\times \Lambda)$. We consider the spectral problem
$$[w_{i},\; v]=\lambda_{i}(w_{i},\; v) \hbox{ for all } v\in \mathcal{D}(\mathcal{A})$$
and set $V_{n}$ the subspace of dimension $n$ generated by $\{w_{1}, \dots , w_{n}\}$. For $\varepsilon > 0$ fixed, we define an approximated solution of order $n$, $n\in \mathbb{N}^{*}$, by:
\begin{equation}\label{3.4}
U_{\varepsilon_{n}}(t)=\sum^{n}_{i=1}\psi_{ni}(t)w_{i}
\end{equation}
\begin{equation}\label{3.5}
\left[\mathcal{N}^{\varepsilon}U'_{\varepsilon_{n}}(t),\; w\right] +\left[\mathcal{A}U_{\varepsilon_{n}}(t),\; w  \right] + \left[\mathcal{M}^{\varepsilon}
\left(U_{\varepsilon_{n}}(t),\; w  \right)  \right]=\left[G_{\varepsilon}(t),\; w \right], \; w\in V_{n}
\end{equation}
\begin{equation}\label{3.6}
U_{\varepsilon_{n}}(0)=\left(E^{0}_{\varepsilon_{n}},\; H^{0}_{\varepsilon_{n}}\right),\; \hbox{where } \left(E^{0}_{\varepsilon_{n}},\; H^{0}_{\varepsilon_{n}}\right)\in V_{n} \hbox{ is such that }
\end{equation}
\begin{equation}\label{3.7}
\left(\left(E^{0}_{\varepsilon_{n}},\; H^{0}_{\varepsilon_{n}}\right)\right) \longrightarrow \left(E^{0},\; H^{0}\right) \hbox{in } \mathcal{D}(\mathcal{A}) \hbox{ as } n\longrightarrow +\infty. 
\end{equation}
\par Assume the function $\psi_{ni}$ are continuous on $[0,\; T]$. Set $\psi_{n}=\left(\psi_{ni}\right)$.
\par \textbf{First case:} Suppose that $\mathcal{M}^{\varepsilon}\left(U_{\varepsilon_{n}}(\cdot)\right)=\left(\theta \circ \psi_{n}\right)\sigma$ where $\sigma\circ \psi_{n} \in \mathcal{C}\left([0,\; T]  \right)$ and $\sigma= \sigma\left(w_{1},..., w_{n}\right) \in \left(\mathcal{D}(\mathcal{A})\right)'$.
\par Set $K^{n}=\psi_{n}\left([0,\; T]\right)$, then we look at the function $\theta$ as an element of $\mathcal{C}(K^{n})$, $K^{n}$ being a compact subset of $\mathbb{R}^{n}$. By Stone-Weierstrass theorem (cf.\cite{7} p.37), for all $\beta >0$, there exists a polynomial in $\lambda \in \mathbb{R}^{n},\; p(\lambda)=a+b\cdot\lambda+\cdots , (a\in \mathbb{R} \; \hbox{and } b\in \mathbb{R}^{n})$, such that $|p(\lambda)-\theta(\lambda)|< \frac{\beta}{2}$, for all $\lambda \in K^{n}$. Let $Q(\lambda)=a+b\cdot\lambda$ (polynomial of degree $1$ extracted from $p$). There exists a compact and connected subset $K^{n}_{0} \subset K^{n}$ with $\psi_{n}(0)\in K^{n}_{0}$, such that $|p(\lambda)-Q(\lambda) |<\frac{\beta}{2} \hbox{ for all } \lambda \in K^{n}_{0}$; therefore  $|\theta(\lambda)-Q(\lambda)|<\beta \hbox{ for all } \lambda \in K^{n}_{0}$. Set $\psi_{n}\left([0,\; T] \right)=K^{n}_{0}$. Then for $\beta>0$ fixed and for all $t\in [0,\; \tau_{n}]$, we have $|\theta(\psi_{n}(t))-Q(\psi_{n}(t))|<\beta$. This suggests that one defines an approximated problem in $[0, \; \tau_{n}]$ of (\ref{3.4})-(\ref{3.6}) as follows:
\begin{equation}\label{3.8}
\begin{array}{l}
\sum \limits^{n}_{i=1}\left[\mathcal{N}^{\varepsilon}w_{i},\; w_{j} \right]\psi'_{ni}(t)+\sum\limits^{n}_{i=1}\left[\mathcal{A}^{\varepsilon}w_{i},\; w_{j}\right]\psi_{ni}(t) + \sum\limits^{n}_{i=1}b_{i}\left[\sigma,\; w_{j}\right]\psi_{ni}(t)\\
=\left[G_{\varepsilon}(t)\right]-a\left[\sigma,\; w_{j}\right], \;j=1,..., n
\end{array}
 \end{equation}
with initial condition
\begin{equation}\label{3.9}
\psi_{ni}(0)=\delta_{ni},\; 1\leq i\leq n
\end{equation}
where $U_{\varepsilon_{n}}(0)=\sum\limits^{n}_{i=1}\delta_{ni}w_{i}$. But the $w_{i}$ are eigen vectors, therefore the matrix $\left[\mathcal{N}^{\varepsilon}w_{i},\; w_{j}\right]$ is invertible, and (\ref{3.8})-(\ref{3.9}) become:
\begin{equation}\nonumber
\begin{array}{l}
\psi'_{n}(t)+M\psi_{n}(t)=G_{n}(t),\; 0<t< \tau\\
\psi_{n}(0)=\delta_{n}
\end{array}
\end{equation}
where $M$ is a  square matrix of order $n$ and $\delta_{n}=(\delta_{ni}),\; 1\leq i\leq n$.
\par $\star$ If $F_{\varepsilon}\in \mathcal{C}\left(\left[0,\; T\right];\; \mathbb{L}^{2}(\Omega\times \Lambda)\right)$, there exists $0<\tau_{n}\leq \tau$ 
 such that the differential system below has as unique solution $\psi_{n}: \left[0,\; \tau^{n}_{n}\right] \longrightarrow \mathbb{R}^{n}$, given by 
\begin{equation}\label{3.10} 
\psi_{n}(t)=e^{t}\delta_{n} + \int^{t}_{0} e^{(s-t)M}G_{n}(s)ds,\; \cite{13, 39};
\end{equation}
and it follows that $\psi_{n},\; \psi'_{n}$ belong to $\mathcal{C}\left(\left[0,\; \tau^{n}_{n}  \right]; \; \mathbb{R}^{n}\right)$. This being so, using (\ref{3.8}) and (\ref{3.4}), and setting\\
 $P_{n}(t,\; w)=\left[ \mathcal{N}^{\varepsilon}U'_{\varepsilon_{n}}(t),\; w \right] +\left[\mathcal{A}^{\varepsilon}U_{\varepsilon_{n}}(t), \;w \right] + \left[\mathcal{M}^{\varepsilon}\left(U_{\varepsilon_{n}}(t)\right), \; w\right]-\left[G_{\varepsilon}(t),\; w \right]$, we get:
 \begin{equation}\nonumber
\begin{array}{l}
|P_{n}(t,\; w)|= \\
\big|\left[ \mathcal{N}^{\varepsilon}U'_{\varepsilon_{n}}(t),\; w \right] +\left[\mathcal{A}^{\varepsilon}U_{\varepsilon_{n}}(t), \;w \right] + \left[\mathcal{M}^{\varepsilon}\left(U_{\varepsilon_{n}}(t)\right), \; w\right]-\left[G_{\varepsilon}(t),\; w \right]\big|=^{(\ref{3.8})}\\
\big|\left[ \mathcal{N}^{\varepsilon}U'_{\varepsilon_{n}}(t),\; w \right] +\left[\mathcal{A}^{\varepsilon}U_{\varepsilon_{n}}(t), \;w \right] + \left[\mathcal{M}^{\varepsilon}\left(U_{\varepsilon_{n}}(t)\right), \; w\right]\\
-\big(\sum\limits^{n}_{i=1}\left[\mathcal{N}^{\varepsilon}w_{i},\; w\right]\psi'_{ni}(t) +\sum\limits^{n}_{i=1}\left[\mathcal{A}^{\varepsilon}w_{i}, w  \right]\psi_{ni}(t)+\sum\limits^{n}_{i=1}b_{i}\left[\sigma,\; w  \right]\psi_{ni}(t)+a\left[\sigma,\; w  \right]\big) \big|\\
=\big| \mathcal{M}\left(U_{\varepsilon_{n}}(t)\right)\left[\sigma,\; w  \right] -\sum\limits^{n}_{i=1}b_{i}\left[\sigma, \; w \right]\psi_{ni}(t)-a\left[\sigma,\; w \right]  \big| \\
=\big|\theta\left(\psi_{n}(t)\right)\left[\sigma, \; w \right]-Q\left(\psi_{n}(t)\right)\left[\sigma,\; w\right]  \big| \\
\leq \beta\big| \left[\sigma,\; w\right] \big| \hbox{ for all } t\in \left[ 0,\; \tau^{n}_{n}\right] \hbox{ and all } w\in V_{n}.
\end{array}
\end{equation}
As $\beta$ is fixed arbitrarily, then the solution of (\ref{3.8})-(\ref{3.9}) imply the existence of a local solution in $\left[0,\; \tau^{n}_{n}\right]$ of (\ref{3.4})
-(\ref{3.6}). Combining a priori estimate to the local solution in $\left[\tau^{n}_{n},\; \tau^{n}_{n}+\delta \right],\; \delta>0,\; \tau^{n}_{n}+\delta \leq T$, where the unknown function $U_{\varepsilon_{n}}$ is replaced by $V_{\varepsilon_{n}}$, with initial condition $V_{\varepsilon_{n}}\left(\tau^{n}_{n}  \right)=U_{\varepsilon_{n}}\left(\tau^{n}_{n}\right)$, we replace the  interval $\left[0,\; \tau^{n}_{n} \right]$ by $\left[0,\; T\right]$. Consequently, (\ref{3.4})-(\ref{3.6}) have at least one global solution whenever $\mathcal{M}^{\varepsilon}\left(U_{\varepsilon_{n}}\left( \cdot\right)\right)$ is of 
the form $\left(\theta\circ \psi_{n}\right)$ and $F_{\varepsilon}\in \mathcal{C}\big(\left[0,\; T\right];\; \mathbb{L}^{2}(\Omega\times \Lambda) \big)$. 
\begin{lemma}\label{l3.4}
The problem (\ref{3.4})-(\ref{3.6}) with $\mathcal{M}^{\varepsilon}\left(U_{\varepsilon_{n}}\left(\cdot \right)\right)=\left(\theta \circ \psi_{n}\right)\sigma,\;\theta\circ \psi_{n}\in \mathcal{C}\left(\left[0,\; T\right] \right),\;\\ \sigma=\sigma(w_{1}, ..., w_{n}) \in \big( \mathcal{D}\left(\mathcal{A}\right) \big)' $,
and $F_{\varepsilon}\in \mathcal{C}\big(\left[0, T \right];\; \mathbb{L}^{2}(\Omega \times \Lambda) \big)$ admits at least one global solution in $\left[0, T  \right].$
\end{lemma}

\begin{remark}
If $\mathcal{M}^{\varepsilon}\left(U_{\varepsilon_{n}}\left(\cdot\right)\right)=\sum_{finite}\big(\theta^{l}\circ\psi_{n}\big)\sigma^{l}$ then one shows that there exists $0< \tau_{n}<T$, and a function $U_{\varepsilon_{n}}: \left[0,\; \tau_{n} \right] \longrightarrow V_{n}$ satisfying (\ref{3.4})-(\ref{3.6}). That solution leads to the existence of a global solution by proceeding as above, as the problem is linear.
\end{remark}
\par $\star\star$ We assume now that $F_{\varepsilon}\in L^{2}\big(\left[0,\;T  \right];\; \mathbb{L}^{2}(\Omega\times \Lambda)\big)$. Then there exists 
a sequence $(F_{\varepsilon p})\subset \mathcal{C}\big(\left[0,\; T\right]; \mathbb{L}^{2}(\Omega\times\Lambda)\big)$ converging to $F_{\varepsilon}$ in
 $L^{2}\big( \left[0,\; T  \right];\; \mathbb{L}^{2}(\Omega\times \Lambda)\big)$ when $p\longrightarrow +\infty$. For all $p\in\mathbb{N}^{*}$, there exists a sequence $\big(U_{\varepsilon_{n}p}\big)_{p\in\mathbb{N}^{*}} \subset L^{2}\big( \left[0,\;T  \right];\; \mathbb{L}^{2}(\Omega\times \Lambda)\big)$ such that 
\begin{equation}\label{3.11}
\left\{\begin{array}{l}
U_{\varepsilon_{n}p}\in V_{n}\\
\left[\mathcal{N}^{\varepsilon}U'_{\varepsilon_{n}p}(t),\; w\right] + \left[\mathcal{A}^{\varepsilon}U_{\varepsilon_{n}p}(t), w\right] +\left[\mathcal{M}^{\varepsilon}\left(U_{\varepsilon_{n}p}(t),\; w\right) \right] =\left[G_{\varepsilon p}(t),\; w\right],\\
0<t<T \hbox{ for all } w\in V_{n}
\end{array}\right.
\end{equation}
with 
\begin{equation}\label{3.12}
U_{\varepsilon_{n}p}(0)=U^{0}_{\varepsilon_{n}p}
\end{equation}
where 
\begin{equation}\label{3.13}
U^{0}_{\varepsilon_{n}p} \longrightarrow U^{0}_{\varepsilon_{n}} \hbox{ in } L^{2}\big(\left[0,\; T\right];\; \mathbb{L}^{2}(\Omega\times \Lambda)   \big) \hbox{ when } p\longrightarrow +\infty.
\end{equation}
From a priori estimate and monotony of $\mu^{\varepsilon},\; \eta^{\varepsilon}$ and $\sigma^{\varepsilon}, \hbox{ for } p,q \in \mathbb{N}^{*}$ we have:
\begin{equation}\nonumber
\begin{array}{l}
\big|E_{\varepsilon_{n}p}(t)-E_{\varepsilon_{n}q} \big|^{2}+\big|H_{\varepsilon_{n}p}-H_{\varepsilon_{n}q}\big|^{2} +\delta\int^{T}_{0}\|E_{\varepsilon_{n}p}(t)-E_{\varepsilon_{n}q}(t)\|^{2}dt \\
\leq c\bigg(\|F_{\varepsilon p}-F_{\varepsilon q}\|_{L^{2}(\Omega\times\Lambda)} +\big\|U^{0}_{\varepsilon_{n}p}-U^{0}_{\varepsilon_{n}q}\big\|_{L^{2}(\Omega\times\Lambda)}  \bigg).
\end{array}
\end{equation}
Therefore $\left(U_{\varepsilon_{n}p}\right)_{p\in\mathbb{N}^{*}}$ is a Cauchy sequence in $\mathcal{C}\big(\left[0,\; T\right];\; \mathbb{L}^{2}(\Omega\times\Lambda)^{2}\big)$  and $\left(E_{\varepsilon_{n}p}  \right)_{p\in \mathbb{N}^{*}}$ is a Cauchy type in $L^{2}\big(\left[0,\; T\right];\; \mathbb{L}^{2}(\Omega\times\Lambda)\big)$. There exists $U_{\varepsilon_{n}} \in \mathcal{C}\big(\left[0,\; T\right];\; \mathbb{L}^{2}(\Omega\times\Lambda)^{2}\big)$ such that $U_{\varepsilon_{n}p} \longrightarrow U_{\varepsilon_{n}} $ in $\mathcal{C}\big(\left[0,\; T\right];\; \mathbb{L}^{2}(\Omega\times\Lambda)^{2}\big)$ when $p\longrightarrow +\infty.$ Multiplying (\ref{3.11}) by $\varphi(t),\; \varphi \in \mathcal{D}\left( \left]0,\; T\right[ \right),\; 0<t<T$, and integrating by parts, then passing to the limits when $p\longrightarrow +\infty$ we get
\begin{equation}\nonumber
\begin{array}{l}
-\int^{T}_{0}\left[\mathcal{N}^{\varepsilon}U_{\varepsilon_{n}}(t), w\right]\varphi'(t)dt +\int^{T}_{0}\left[\mathcal{A}U_{\varepsilon_{n}}(t), w\right]\varphi(t)dt +\int^{T}_{0}\left[\mathcal{M}^{\varepsilon}\left(U_{\varepsilon_{n}}(t)\right), w \right]\varphi(t)dt\\
=\int^{T}_{0}\left[G_{\varepsilon}(t), w  \right]\varphi(t)dt.
\end{array}
\end{equation}
As $\varphi$ is arbitrarily fixed, (\ref{3.5}) follows. One shows without difficulties that (\ref{3.6}) happens.
\par \textbf{Second case:} $\mathcal{M}^{\varepsilon}\left(U_{\varepsilon_{n}}(\cdot)\right)$ is any element in $L^{2}\big(\left[0, T \right];\; \mathbb{L}^{2}(\Omega\times\Lambda) \big)$.
\par Let $\beta>0. $ There exists a function $\mathcal{M}^{\varepsilon}_{0}\left(U_{\varepsilon_{n}}(\cdot)\right)$ with  $\mathcal{M}^{\varepsilon}_{0}\left(U_{\varepsilon_{n}}(\cdot)\right)=\sum_{finite}\big(\theta^{l}\circ \psi_{n} \big)\sigma^{l},\; (0\leq t \leq T)$ with 
$\theta^{l}\circ \psi_{n}\in \mathcal{C}\left(\left[0, T\right]\right)$ and $\sigma^{l}\in \mathcal{D}(\mathcal{A})$ such that \\$\sup_{0\leq t\leq T}\big\|\mathcal{M}^{\varepsilon}\left(U_{\varepsilon_{n}}(t)- \mathcal{M}^{\varepsilon}_{0}\left(U_{\varepsilon_{n}}(t)\right)  \right)  \big\|_{\mathbb{L}^{2}(\Omega\times\Lambda)^{2}}< \beta$ where $U_{\varepsilon_{n}}(t)=\sum\limits^{n}_{i=1}\psi_{ni}w_{i}$ is a solution of (\ref{3.4})-(\ref{3.6}) when $\mathcal{M}^{\varepsilon}$ is replaced by $\mathcal{M}^{\varepsilon}_{0}$. Then, setting $Q_{n}(t, w)=\left[\mathcal{N}^{\varepsilon}U'_{\varepsilon_{n}}(t), w\right] + \left[ \mathcal{A}^{\varepsilon}U_{\varepsilon_{n}}(t), w \right] +\left[\mathcal{M}^{\varepsilon}\left(U_{\varepsilon_{n}}(t)\right),w  \right]-\left[ G_{\varepsilon}(t),w \right]$ we have
\begin{equation}\nonumber
\begin{array}{l}
\left|Q_{n}(t, w)  \right|=\big|\left[\mathcal{M}^{\varepsilon}\left( U_{\varepsilon_{n}}(t) \right), w  \right] -\left[\mathcal{M}^{\varepsilon}_{0}\left(U_{\varepsilon_{n}}(t)  \right),w  \right] \big|\\
\leq \big\|\mathcal{M}^{\varepsilon}\left(U_{\varepsilon_{n}}(t)  \right)-\mathcal{M}^{\varepsilon}_{0}\left(U_{\varepsilon_{n}}(t)  \right)  \big\|_{\mathbb{L}^{2}(\Omega\times \Lambda)^{2}} \|w\|_{\mathbb{L}^{2}(\Omega\times \Lambda)^{2}}\\
\leq \beta\|w\|_{\mathbb{L}^{2}(\Omega\times\Lambda)^{2}} \hbox{ for all } w\in V_{n},
\end{array}
\end{equation}
which ensures the existence of a solution of (\ref{3.4})-(\ref{3.6}).
\begin{lemma}\label{l3.5}
The problem (\ref{3.4})-(\ref{3.6}) admits at least a global solution in $\left[0, T\right]$.
\end{lemma}
\subsection{\textbf{Existence of solution of (\ref{1}).}}
Let $\left(U_{\varepsilon_{n}} \right)$ the sequence define by (\ref{3.4})-(\ref{3.6}). From apriori estimate and properties of $\mathcal{A}$ and $\mathcal{M}^{\varepsilon}: \sup_{n\in \mathbb{N}^{*}}\bigg(\|U_{\varepsilon_{n}}\|_{L^{\infty}\left(0, T; \mathbb{L}^{2}(\Omega\times\Lambda)^{2}  \right)}  \bigg)< +\infty,\;\sup_{n\in \mathbb{N}^{*}}\bigg(\|\mathcal{A}U_{\varepsilon_{n}}\|_{L^{2}\left(0, T; \mathbb{L}^{2}(\Omega\times\Lambda)^{2}  \right)}  \bigg)< +\infty, \; \sup_{n\in \mathbb{N}^{*}}\bigg(\|\mathcal{M}^{\varepsilon}U_{\varepsilon_{n}}\|_{L^{2}\left(0, T; \mathbb{L}^{2}(\Omega\times\Lambda)^{2}  \right)}  \bigg)< +\infty,$  and $\sup_{n\in \mathbb{N}^{*}}\bigg(\|U_{\varepsilon_{n}}(T)\|_{L^{2}\left(0, T; \mathbb{L}^{2}(\Omega\times\Lambda)^{2}  \right)}  \bigg)< +\infty,$  therefore we have the following results:
\begin{lemma}\label{l3.6}
One can extract a subsequence $\left(U\varepsilon_{k_{n}}\right)_{n\in \mathbb{N}^{*}}$ such that when $n\longrightarrow +\infty$, we have:
\begin{equation}\nonumber
\begin{array}{l}
U\varepsilon_{k_{n}} \longrightarrow U_{\varepsilon_{n}} \hbox{ in } L^{2}\bigg(0, T; \mathbb{L}^{2}\left(\Omega\times\Lambda \right)\bigg)-weak,\\
\mathcal{A}U\varepsilon_{k_{n}}\longrightarrow W_{\varepsilon} \hbox{ in } L^{2}\bigg(0, T; \mathbb{L}^{2}\left(\Omega\times \Lambda\right)^{2}\bigg)-weak,\\
 \mathcal{M}^{\varepsilon}U\varepsilon_{k_{n}}\longrightarrow Z_{\varepsilon} \hbox{ in } L^{2}\bigg(0, T; \mathbb{L}^{2}\left(\Omega\times \Lambda\right)^{2}\bigg)-weak,\\
and\\
U\varepsilon_{k_{n}}(T)\longrightarrow g_{\varepsilon} \hbox{ in } \mathbb{L}^{2}\left(\Omega\times\Lambda \right)^{2}-weak.
\end{array}
\end{equation}
\end{lemma}

\begin{lemma}\label{l3.7}
The problem
\begin{equation}\nonumber
\left\{\begin{array}{l}
Find \; U_{\varepsilon}\in L^{2}\bigg(0,T;H_{\hbox{curl}}\left(\Omega; L^{2}(\Lambda)\right)^{2}\bigg),\\
such\; that \; \partial_{t}U_{\varepsilon}\in L^{2}\bigg(0, T; \mathbb{L}^{2}\left(\Omega\times\Lambda\right)^{2}  \bigg),\\
\mathcal{N}^{\varepsilon}\partial_{t}U_{\varepsilon}+\mathcal{A}U_{\varepsilon}+\mathcal{M}^{\varepsilon}(U_{\varepsilon})=G_{\varepsilon}\; in\; \left]0,T \right[ \\
U_{\varepsilon}(0)=U^{\varepsilon}_{0}=\left(E^{0}_{\varepsilon}, H^{0}_{\varepsilon}\right) \; and \; U^{\varepsilon}(T)=g_{\varepsilon}
\end{array}\right.
\end{equation}
\end{lemma}
admits a unique solution.
\begin{proof}
For unicity,  assume there exists two solutions $U_{\varepsilon 1}$ and $U_{\varepsilon 2}$ and set $U_{\varepsilon}=U_{\varepsilon 1}-U_{\varepsilon 2}$.
 Then
\begin{equation}\nonumber
\begin{array}{l}
\mathcal{N}^{\varepsilon}\partial_{t}U_{\varepsilon}+\mathcal{A}U_{\varepsilon}+\mathcal{M}(U_{1}^{\varepsilon})-\mathcal{M}^{\varepsilon}(U_{2}^{\varepsilon})=0 \hbox{ with } U_{\varepsilon}(0)=U_{\varepsilon}(T)=0\\
\hbox{therefore}\\
\left(\mathcal{N}^{\varepsilon}\partial_{t}U_{\varepsilon} +\mathcal{A}U_{\varepsilon} +\mathcal{M}^{\varepsilon}(U^{\varepsilon}_{1})-\mathcal{M}^{\varepsilon}(U_{2}^{\varepsilon}), U_{\varepsilon}          \right)=0, \hbox{ and then }\\
\frac{1}{2}\partial_{t}\left(\mathcal{N}^{\varepsilon}U_{\varepsilon}(t), U_{\varepsilon}(t)\right)+\delta|U_{\varepsilon}(t)|^{2}\leq 0.
\end{array}
\end{equation}
For all $T_{1}\in (0, T)$, we get: 
\begin{equation}\nonumber
\begin{array}{l}
c'\|U_{\varepsilon}(T_{1})\|^{2}_{\mathbb{L}^{2}(\Omega\times\Lambda)^{2}}+\delta \int^{T_{1}}_{0}\|U_{\varepsilon}(t)\|^{2}_{\mathbb{L}^{2}(\Omega\times\Lambda)^{2}}dt \leq \left[\mathcal{N}^{\varepsilon}U_{\varepsilon}(T_{1}), U_{\varepsilon}(T_{1})\right] + \\ \delta\int^{T_{1}}_{0}\|U_{\varepsilon}(t)\|^{2}_{\mathbb{L}^{2}(\Omega\times\Lambda)^{2}}dt\leq 0,
\end{array}
\end{equation}
hence $\|U_{\varepsilon}(t)\|_{\mathbb{L}^{2}(\Omega\times\Lambda)^{2}}=0$ for almost all $t \in (0, T)$ and unicity follows.
\par For existence, let $U\varepsilon_{k_{n}}$ the sequence of lemma (3.6) we have $\left[\mathcal{A}U\varepsilon_{k_{n}}, \phi\right]=-\left[U\varepsilon_{k_{n}}, \mathcal{A}\phi\right]$ for all $\phi\in \left[\mathcal{D}(\Omega)\otimes\mathcal{C}^{\infty}(\Lambda)\right]^{6}$ which at the limit leads to $\left[W_{\varepsilon}, \phi \right]=-\left[U_{\varepsilon}, \mathcal{A}\phi\right]$ for all $\phi\in \left[\mathcal{D}(\Omega)\otimes\mathcal{C}^{\infty}(\Lambda)\right]^{6}$, that is $W_{\varepsilon}=\mathcal{A}U_{\varepsilon}$ in the sense of distribution then in the sense of almost every where as $W_{\varepsilon},\; U_{\varepsilon} \in L^{2}\bigg(0, T; \mathbb{L}^{2}(\Omega\times \Lambda)^{2} \bigg)$ and it follows that $U_{\varepsilon} \in L^{2}\bigg(0, T; H_{\hbox{curl}}\left(\Omega; L^{2}(\Lambda) \right)^{2}  \bigg)$. One knows that  
\begin{equation}\nonumber
\begin{array}{l}
\int^{T}_{0}\partial_{t}\left(\mathcal{N}^{\varepsilon}U^{\varepsilon}_{k_{n}}(t), \psi(t)v \right)dt= \left(\mathcal{N}^{\varepsilon}U^{\varepsilon}_{k_{n}}(T), \psi(T)v  \right)-\left(\mathcal{N}^{\varepsilon}U^{\varepsilon}_{k_{n}}(0),\psi(0)v  \right)\\
=\int^{T}_{0}\left(\partial_{t}\mathcal{N}^{\varepsilon}U^{\varepsilon}_{k_{n}}(t), \psi(t)v\right)dt +\int^{T}_{0}\left(\mathcal{N}^{\varepsilon}U^{\varepsilon}_{k_{n}}(t),\psi'(t)v\right)dt\\
\hbox{ for all } \psi\in\mathcal{C}^{1}\left(\left[0, T  \right]  \right) \hbox{ and } v\in \mathbb{L}^{2}\left(\Omega\times\Lambda \right)^{2};
\end{array}
\end{equation}
then taking in account (\ref{3.4})-(\ref{3.6}), we have:
\begin{equation}\nonumber
\begin{array}{l}
\int^{T}_{0}\left(G_{\varepsilon}(t)-\mathcal{A}U\varepsilon_{k_{n}}(t)-\mathcal{M}^{\varepsilon}\left(U\varepsilon_{k_{n}}(t), \psi(t)v\right)\right)dt + \int^{T}_{0}\left(\mathcal{N}^{\varepsilon}U\varepsilon_{k_{n}}(t),\psi'(t)v\right)dt \\ 
\left(\mathcal{N}^{\varepsilon}U\varepsilon_{k_{n}}(T), \psi(T)v \right)-\left(\mathcal{N}^{\varepsilon}U^{\varepsilon}_{0k_{n}}, \psi(0)v  \right),\; \psi\in \mathcal{C}^{1}\left(\left[0, T  \right]\right) \hbox{ and } v\in \mathcal{D}(\mathcal{A}).
\end{array}
\end{equation}
Hence passing to limits when $n\longrightarrow +\infty$ and using density of $\mathcal{D}(\mathcal{A})$ in $\mathbb{L}^{2}(\Omega\times\Lambda)^{2}$, we are lead to:
\begin{equation}\label{3.14}
\begin{array}{l}
\left(\mathcal{N}^{\varepsilon}g_{\varepsilon}, \psi(T)v\right)-\left(\mathcal{N}^{\varepsilon}U^{\varepsilon}_{0}, \psi(0)v\right)=\\
\int^{T}_{0}\left[\left(G_{\varepsilon}-W_{\varepsilon}-Z_{\varepsilon}, \psi(t)v\right)+ \left(\mathcal{N}^{\varepsilon}U_{\varepsilon}(t), \psi'(t)v  \right) \right]dt \\
\psi \in \mathcal{C}^{1}\left(\left[0, T\right]  \right) \hbox{ and } v\in \mathbb{L}^{2}(\Omega\times\Lambda)^{2}.
\end{array}
\end{equation}
For $\psi\in \mathcal{D}(0, T),\; v\in \mathbb{L}^{2}(\Omega\times\Lambda)^{2}$, the preceding equality becomes:
\[\int^{T}_{0}\left(G_{\varepsilon}-W_{\varepsilon}-Z_{\varepsilon}, v \right)\psi(t)+ \left(\mathcal{N}^{\varepsilon}U_{\varepsilon}(t), v\right)\psi'(t)dt=0\]
that is;
\begin{equation}\label{3.15}
\mathcal{N}^{\varepsilon}\partial_{t}U_{\varepsilon}(t)+\mathcal{A}U_{\varepsilon}(t)+Z_{\varepsilon}(t)=G_{\varepsilon}(t)
\end{equation}
and $\mathcal{N}^{\varepsilon}\partial_{t}U^{\varepsilon}(t)\in L^{2}\bigg(0, T; \mathbb{L}^{2}(\Omega\times\Lambda)^{2}\bigg)$. Properties of $\mathcal{N}^{\varepsilon}$ leads to $\left(\partial_{t}E_{\varepsilon},\; \partial_{t}H_{\varepsilon} \right) \in L^{2}\bigg(0, T; \mathbb{L}^{2}(\Omega\times \Lambda)^{2}\bigg)$ with $U_{\varepsilon} \in L^{2}\bigg(0, T; H_{\hbox{curl}}\left(\Omega; L^{2}(\Lambda)\right)^{2}\bigg)$ implies\\ 
$U^{\varepsilon} \in \mathcal{C}\bigg(\left[0, T\right]; \mathbb{L}^{2}(\Omega\times\Lambda)^{2}\bigg)$ and integration by parts formula is written:
\begin{equation}\nonumber
\begin{array}{l}
\left(\mathcal{N}^{\varepsilon}U^{\varepsilon}(T), \psi(T)v\right)-\left(\mathcal{N}^{\varepsilon}U^{\varepsilon}(0), \psi(0)v\right)=\\
\int^{T}_{0}\left(\partial_{t}\mathcal{N}^{\varepsilon}U^{\varepsilon}(t), \psi(t)v \right) + \left(\mathcal{N}^{\varepsilon}U^{\varepsilon}(t), \psi'(t)v  \right)dt; \; \psi \in \mathcal{C}^{1}\left(\left[0, T  \right]  \right) \hbox{ and } v\in \mathbb{L}^{2}(\Omega\times\Lambda)^{2}.
\end{array}
\end{equation}
From the above equality combined with (\ref{3.14}) and with (\ref{3.15}), one gets:
\begin{equation}\nonumber
\begin{array}{l}
\left(\mathcal{N}^{\varepsilon}U^{\varepsilon}(T), \psi(T)v\right)-\left(\mathcal{N}^{\varepsilon}U^{\varepsilon}(0), \psi(0)v\right)= \left(\mathcal{N}^{\varepsilon}g^{\varepsilon}, \psi(T)v \right)-\left(\mathcal{N}^{\varepsilon}U^{\varepsilon}_{0}, \psi(0)v\right);\\
\psi\in \mathcal{C}^{1}\left(\left[0, T  \right]  \right) \hbox{ and } v\in \mathbb{L}^{2}(\Omega\times\Lambda)^{2}.
\end{array}
\end{equation}
For $\psi\in\mathcal{C}^{1}\left(\left[0, T\right]  \right) $ such that $\psi(0)=0,\; \psi(T)=1$ we deduce that $U_{\varepsilon}(T)=g_{\varepsilon}$; for 
$\psi\in\mathcal{C}^{1}\left(\left[0, T\right]  \right) $ such that $\psi(0)=1,\; \psi(T)=0$ we deduce that $U_{\varepsilon}(T)=U^{0}_{\varepsilon}$. Remains to show that $Z_{\varepsilon}=\mathcal{M}^{\varepsilon}(U_{\varepsilon})$. For that (\ref{3.4})-(\ref{3.6}) for sequence $k_{n}$ applied to $U\varepsilon_{k_{n}}$ followed by an integration by parts imply:
\begin{equation}\nonumber
\begin{array}{l}
\frac{1}{2}\left[\big(\mathcal{N}^{\varepsilon}U\varepsilon_{k_{n}}(T), U\varepsilon_{k_{n}}(T)\big)-\big(\mathcal{N}^{\varepsilon}U\varepsilon_{k_{n}}(0), U\varepsilon_{k_{n}}(0)\big)\right]=\int^{T}_{0}\big(\partial_{t}\mathcal{N}^{\varepsilon}U\varepsilon_{k_{n}}(t), U\varepsilon_{k_{n}}(t)\big)dt\\
=\int^{T}_{0}\big(G_{\varepsilon}(t)-\mathcal{A}U\varepsilon_{k_{n}}(t)-\mathcal{M}^{\varepsilon}\big(U\varepsilon_{k_{n}}(t)\big), U\varepsilon_{k_{n}}(t)\big)dt\\
=\int^{T}_{0}\big(G_{\varepsilon}(t)-\mathcal{M}^{\varepsilon}\big(U\varepsilon_{k_{n}}(t)\big), U\varepsilon_{k_{n}}(t)\big)dt\\
=\int^{T}_{0}\big(F_{\varepsilon}(t), E\varepsilon_{k_{n}}(t)\big)dt - \int^{T}_{0}\big(\mathcal{M}^{\varepsilon}_{1}\big(E\varepsilon_{k_{n}}(t)\big), E\varepsilon_{k_{n}}(t)\big)dt.
\end{array}
\end{equation}
As $v\longrightarrow \left(\mathcal{N}^{\varepsilon} v, v  \right)^{\frac{1}{2}}$ define a norm on $\mathbb{L}^{2}(\Omega\times \Lambda),\; v\mapsto \left(\mathcal{N}^{\varepsilon}v,v  \right)$ is weakly lower semi continuous, and $U\varepsilon_{k_{n}}(T)\longrightarrow g_{\varepsilon}$ in $\mathbb{L}^{2}(\Omega\times\Lambda)^{2}$-weak, therefore $\left(\mathcal{N}^{\varepsilon}U_{\varepsilon}(T), U_{\varepsilon}(T) \right) \leq \lim \inf_{n\longrightarrow +\infty} \left(\mathcal{N}^{\varepsilon}U^{\varepsilon}_{k_n}(T), U^{\varepsilon}_{k_{n}}(T)\right)$. Then knowing that $\lim \sup_{n\longrightarrow +\infty}\left(-a_{n} \right)=-\lim\inf _{n\longrightarrow +\infty}(a_{n}); \;a_{n}$ being an real sequence we have:
\begin{equation}\label{3.16}
\begin{array}{l}
\lim\sup_{n\longrightarrow +\infty}\int^{T}_{0}\left(\mathcal{M}^{\varepsilon}_{1}\left(E\varepsilon_{k_{n}}(t)\right), E\varepsilon_{k_{n}}(t)\right)dt=\lim\sup_{n\longrightarrow +\infty}\\
\big[\int^{T}_{0}\left(F_{\varepsilon}(t), E\varepsilon_{k_{n}}(t)\right)dt-\frac{1}{2}\left(\mathcal{N}^{\varepsilon}U\varepsilon_{k_{n}}(T), U\varepsilon_{k_{n}}(T)  \right)+\frac{1}{2}\left(\mathcal{N}^{\varepsilon}U\varepsilon_{k_{n}}(0), U\varepsilon_{k_{n}}(0)\right)\big] \\
=\lim\sup_{n\longrightarrow +\infty}\int^{T}_{0}\left(F_{\varepsilon}(t), E\varepsilon_{k_{n}}(t)\right)dt-\frac{1}{2}\lim\inf_{n\longrightarrow +\infty}\left( \mathcal{N}^{\varepsilon}U\varepsilon_{k_{n}}(T), U\varepsilon_{k_{n}}(T)\right)\\
+\frac{1}{2}\lim\sup_{n\longrightarrow +\infty}\left(\mathcal{N}^{\varepsilon}U\varepsilon_{k_{n}}(0),U\varepsilon_{k_{n}}(0)\right)\leq \int^{T}_{0}\left(F_{\varepsilon}(t), E_{\varepsilon}(t)\right)dt \\
+\frac{1}{2}\left[\left(\mathcal{N}^{\varepsilon}U^{0}_{\varepsilon}, U^{0}_{\varepsilon} \right)-\left(\mathcal{N}^{\varepsilon}g_{\varepsilon}, g_{\varepsilon}\right)\right]\leq \int^{T}_{0}\left(Z_{\varepsilon}, U_{\varepsilon}\right)dt=\int^{T}_{0}\left(Z_{\varepsilon}1, E_{\varepsilon}\right)dt;\; \hbox{with } Z_{\varepsilon}=(Z_{\varepsilon}1,\; 0) \; \hbox{ and }\\ \mathcal{M}^{\varepsilon}(U)=(\mathcal{M}^{\varepsilon}_{1}(U_{1}), 0) \hbox{ for }
 U=(U_{1}, U_{2}).
\end{array}
\end{equation}
In fact \ref{3.14} implies
\begin{equation}\nonumber
\begin{array}{l}
\left[\left(\mathcal{N}^{\varepsilon}g_{\varepsilon}, g_{\varepsilon}\right)-\left(\mathcal{N}^{\varepsilon}U^{0}_{\varepsilon}, U^{0}_{\varepsilon}\right)    \right]=\\
\int^{T}_{0}\left[\left(G_{\varepsilon}-W_{\varepsilon}-Z_{\varepsilon}, U_{\varepsilon}\right)+\left(\mathcal{N}^{\varepsilon}U_{\varepsilon}(t), \partial_{t}U_{\varepsilon}\right)  \right]dt\\
=\int^{T}_{0}\left[\left(\mathcal{N}^{\varepsilon}\partial_{t}U_{\varepsilon}(t), U_{\varepsilon}\right)+\left(\mathcal{N}^{\varepsilon}U_{\varepsilon}(t), \partial_{t}U_{\varepsilon}\right) \right]dt\\
=2\int^{T}_{0}\left(\mathcal{N}^{\varepsilon}\partial_{t}U_{\varepsilon}(t), U_{\varepsilon}\right)dt.
\end{array}
\end{equation}
For $\varepsilon >0$ fixed, let now $(E_{\varepsilon}p)_{p}$ a sequence converging strongly to $E_{\varepsilon}$ in $L^{2}\left(0, T;\; \mathbb{L}^{2}\left(\Omega\times \Lambda \right)  \right)$, let $v\in \mathbb{L}^{2}\left(\Omega\times \Lambda  \right)$ and $r$ a real different from zero. From monotonicity (\ref{7}) of $\sigma$ we have:
\begin{equation}\label{3.17}
\begin{array}{l}
\big(\mathcal{M}^{\varepsilon}_{1}\big(E\varepsilon_{k_{n}}\big)-\mathcal{M}^{\varepsilon}_{1}\big(E_{\varepsilon p}+rv \big), E\varepsilon_{k_{n}}-(E_{\varepsilon p}+rv )  \big)=\\
\big(\mathcal{M}^{\varepsilon}_{1}\big(E\varepsilon_{k_{n}}  \big), E\varepsilon_{k_{n}}   \big)-\big(\mathcal{M}^{\varepsilon}_{1}\big(E\varepsilon_{k_{n}}\big), E_{\varepsilon p}+rv  \big)-\big(\mathcal{M}^{\varepsilon}_{1}(E_{\varepsilon p}+rv), E\varepsilon_{k_{n}}\big)\\
+ \left(\mathcal{M}^{\varepsilon}_{1}(E_{\varepsilon p}+rv), E_{\varepsilon p}+rv  \right)\geq 0.
\end{array}
\end{equation}
As
\begin{equation}\nonumber
\begin{array}{l}
\mathcal{M}^{\varepsilon}_{1}\big(E\varepsilon_{k_{n}}\big)\longrightarrow Z_{1 \varepsilon} \hbox{ in } L^{2}\big(0, T;\; \mathbb{L}^{2}(\Omega\times \Lambda)^{2}\big)\hbox{ -weak},\\
E\varepsilon_{k_{n}}\longrightarrow E_{\varepsilon} \hbox{ in } L^{2}\big(0, T;\; \mathbb{L}^{2}(\Omega\times \Lambda)^{2}\big) \hbox{ -weak and }\\
E_{\varepsilon p}\longrightarrow E_{\varepsilon} \hbox{ in } L^{2}\big(0, T; \; \mathbb{L}^{2}(\Omega\times \Lambda)^{2}\big) \hbox{ -strong}
\end{array}
\end{equation}
it follows that 
\begin{equation}\nonumber
\begin{array}{l}
\lim\sup_{n\longrightarrow +\infty}\int^{T}_{0}\big(\mathcal{M}^{\varepsilon}_{1}\big(E\varepsilon_{k_{n}}\big), E_{\varepsilon p}+rv\big)dt=\int^{T}_{0}\left(Z_{\varepsilon 1}, E_{\varepsilon p}+rv  \right)dt \hbox{ and }\\
\lim\sup_{n\longrightarrow +\infty} \int^{T}_{0}\big(\mathcal{M}^{\varepsilon}_{1}\left(E_{\varepsilon p}+rv \right), E\varepsilon_{k_{n}}\big)dt=\int^{T}_{0}\big(\mathcal{M}^{\varepsilon}_{1}\left(E_{\varepsilon p}+rv\right), E_{\varepsilon}\big)dt.
\end{array}
\end{equation}
Therefore, from inequality (\ref{3.17}) combined with (\ref{3.16}) and to the two upper bound limits above, we get
\begin{equation}\nonumber
\begin{array}{l}
\int^{T}_{0}\left(Z_{1\varepsilon}, E_{\varepsilon}\right)dt-\int^{T}_{0}\left(Z_{1\varepsilon}, E_{\varepsilon p}+rv  \right)dt-\int^{T}_{0}\big(\mathcal{M}^{\varepsilon}_{1}\left(E_{\varepsilon p}+rv\right)  \big)dt\\
+\int^{T}_{0}\big(\mathcal{M}^{\varepsilon}_{1}\left(E_{\varepsilon p}+rv \right), E_{\varepsilon p}+rv \big)dt \geq 0.
\end{array}
\end{equation}
Taking in account lemma (\ref{17}) we have:
\begin{equation}\nonumber
\begin{array}{l}
0\leq \lim_{p\longrightarrow +\infty}\int^{T}_{0}\left(Z_{1\varepsilon}, E_{\varepsilon}\right)dt-\int^{T}_{0}\left(Z_{1\varepsilon}, E_{\varepsilon p}+rv  \right)dt-\int^{T}_{0}\left(\mathcal{M}^{\varepsilon}_{1}\left(E_{\varepsilon p}+rv \right), E_{\varepsilon}\right)dt\\
+\int^{T}_{0}\left(\mathcal{M}^{\varepsilon}_{1}\left(E_{\varepsilon p}+rv\right), E_{\varepsilon p}+rv\right)dt dt=\\
\int^{T}_{0}\left(Z_{1\varepsilon}, E_{\varepsilon}  \right)dt-\int^{T}_{0}\left(Z_{1\varepsilon}, E_{\varepsilon}+rv \right)dt-\int^{T}_{0}\left(\mathcal{M}^{\varepsilon}_{1}\left(E_{\varepsilon}+rv\right), E_{\varepsilon}\right)dt\\
+\int^{T}_{0}\left(\mathcal{M}^{\varepsilon}_{1}\left(E_{\varepsilon}+rv\right), E_{\varepsilon}+rv \right)dt=-r\int^{T}_{0}\left(Z^{\varepsilon}_{1}, v  \right)dt+r\int^{T}_{0}\left(\mathcal{M}^{\varepsilon}_{1}\left(E_{\varepsilon}+rv  \right),v  \right)dt.
\end{array}
\end{equation}
Multiplying by $\frac{1}{r}\;(r>0)$ and taking the limit as $(r\longrightarrow 0^{+})$, one obtain $\int^{T}_{0}\left(Z_{1\varepsilon}, v  \right)dt \leq \int^{T}_{0}\left(\mathcal{M}^{\varepsilon}_{1}(E_{\varepsilon}), v\right)dt, \;\forall v \in \mathbb{L}^{2}(\Omega\times \Lambda)$; next multiplying by $\frac{1}{r}\; (\hbox{with } r<0)$ and taking limit as $(r\longrightarrow 0^{-})$, one gets $\int^{T}_{0}\left(Z_{1\varepsilon}, v  \right)dt \geq \int^{T}_{0}\left(\mathcal{M}^{\varepsilon}_{1}(E_{\varepsilon}), v  \right)dt,\; \forall v\in \mathbb{L}^{2}(\Omega\times \Lambda)$. Consequently, $Z_{1\varepsilon}=\mathcal{M}^{\varepsilon}_{1}(E_{\varepsilon})$ then $Z_{\varepsilon}=\mathcal{M}^{\varepsilon}(U^{\varepsilon})$. 
\end{proof}
\begin{theorem}\label{t3.1}
For all $\varepsilon >0$, there exists one and only one couple of functions $(E_{\varepsilon}, H_{\varepsilon})$ from $Q\times \Lambda$ to $\mathbb{R}^{3}\times \mathbb{R}^{3}$ defined by (\ref{1}) and satisfying at $(E_{\varepsilon}, H_{\varepsilon}) \in L^{2}\bigg(0, T; H_{\mathrm{curl}}\big(\Omega; L^{2}(\Lambda) \big)  \bigg)$ and $(\partial_{t}E_{\varepsilon}, \partial_{t}H_{\varepsilon})\in L^{2}\bigg(0, T; \mathbb{L}^{2}\left(\Omega\times \Lambda  \right)^{2}  \bigg)$.
\end{theorem}
\begin{proof}
It is a consequence of lemmas \ref{l3.3}, \ref{l3.5}, \ref{l3.6} and \ref{l3.7}.
\end{proof}
\begin{lemma}\label{l3.8}
The sequence of solutions of (\ref{1}) verify
\begin{equation}\nonumber
\begin{array}{c}
\max \bigg(\sup\limits_{0<\varepsilon \leq 1}\|E_{\varepsilon}\|_{L^{\infty}(0, T; \mathbb{L}^{2}(\Omega\times\Lambda))},\sup\limits_{0< \varepsilon \leq 1}\|H_{\varepsilon}\|_{L^{\infty}(0, T; \mathbb{L}^{2}(\Omega\times\Lambda))}  \bigg)<\infty, \\
\sup\limits_{0<\varepsilon \leq 1}\|J_{\varepsilon}\|_{L^{2}(0, T; \mathbb{L}^{2}(\Omega\times\Lambda))}<\infty, \\
\max \bigg(\sup\limits_{0<\varepsilon \leq 1}\|\mathrm{curl}E_{\varepsilon}\|_{L^{\infty}(0, T; \mathbb{L}^{2}(\Omega\times\Lambda))},\sup\limits_{0< \varepsilon \leq 1}\|\mathrm{curl}H_{\varepsilon}\|_{L^{2}(0, T; \mathbb{L}^{2}(\Omega\times\Lambda))}  \bigg)<\infty, \\
and \\
\max \bigg(\sup\limits_{0<\varepsilon \leq 1}\|\partial_{t}E_{\varepsilon}\|_{L^{2}(0, T; \mathbb{L}^{2}(\Omega\times\Lambda))},\sup\limits_{0< \varepsilon \leq 1}\|\partial_{t}H_{\varepsilon}\|_{L^{\infty}(0, T; \mathbb{L}^{2}(\Omega\times\Lambda))}  \bigg)<\infty.
\end{array}
\end{equation}
\end{lemma}
\begin{proof}
The proof follows from apriori estimate (\ref{19}), (see lemma \ref{l3.3}).
\end{proof}
\section{CONVERGENCE AND HOMOGENIZATION RESULTS}
\subsection{\textbf{Preliminaries convergence results.}} The following result is a consequence of lemma \ref{l3.8} and properties of stochastic-periodic two-scale convergence.
\begin{lemma}\label{l4.1}
There exist a fundamental sequence $E$ extracted from $(0,\; 1]$, functions 
\begin{equation}\nonumber
\begin{array}{l}  
\overline{E}_{0}, \overline{H}_{0}\in L^{2}\big(Q\times \Lambda_{0}; L^{2}_{per}(\Theta)\big)^{3},\\
\overline{e}_{0}^{s},\overline{h}_{0}^{s}\in L^{2}\bigg(Q\times \Lambda_{0}; L^{2}_{per}\big(\Theta; H^{1}_{\#}(\Lambda)\big)  \bigg),\\
\overline{e}_{0}^{\sigma},\overline{h_{0}}^{\sigma}\in L^{2}\bigg(Q\times\Lambda\times \Lambda_{0}; L^{2}_{per}\big(\Theta; H^{1}_{\#}(Z)\big)  \bigg),\\
\overline{E}_{1}, \overline{H}_{1}\in H^{1}\bigg(0, T; L^{2}\big(\Omega\times \Lambda; L^{2}_{nv}\big(\Lambda_{0}; L^{2}_{per}(Z) \big)\big)^{3}
\bigg), \\
\overline{E}_{1}^{s}, \overline{H}_{1}^{s}\in L^{2}\bigg(Q\times \Lambda; L^{2}_{per}\big(Z; \mathcal{H}^{1}(\Lambda_{0})\big)^{3}\bigg), \; and\\
\overline{E}_{1}^{\sigma}, \overline{H}_{1}^{\sigma}\in L^{2}\bigg(Q\times \Lambda; L^{2}_{nv}\big(\Lambda_{0}; L^{2}_{per}\big(Z; H^{1}_{\#}(\Theta)\big)\big)^{3}\bigg), \; with\\
\overline{E}_{0}+D_{\omega}\overline{e}_{0}^{s}, \overline{H}_{0}+D_{\omega}\overline{h}^{s}_{0}\in H_{\mathrm{curl}}\big(Q; L^{2}\big(\Lambda\times\Lambda_{0}; L^{2}_{per}(\Theta)  \big)  \big),\\
\overline{J}_{0}\in L^{2}\big(Q\times\Lambda\times\Lambda_{0}; L^{2}_{per}(Y)  \big),
\end{array}
\end{equation}
such that, when $E\ni \varepsilon\longrightarrow 0$, one has:
\begin{equation}\nonumber
\begin{array}{l}
E_{\varepsilon}\rightarrow \overline{E}_{1}=\overline{E}_{0}+D_{\omega}\overline{e}^{s}_{0}+D_{y}\overline{e}^{\sigma}_{0}\; in\; L^{2}\left(Q\times\Lambda\times\Lambda_{0}\right)^{3}-weak\; stoch-2s,\\
H_{\varepsilon}\rightarrow \overline{H}_{1}=\overline{H}_{0}+D_{\omega}\overline{h}^{s}_{0}+D_{y}\overline{h}^{\sigma}_{0}\; in\; L^{2}\left(Q\times\Lambda\times\Lambda_{0}\right)^{3}-weak\; stoch-2s,\\
J_{\varepsilon} \rightarrow \overline{J}_{0}\; in \; L^{2}\left(Q\times\Lambda\times\Lambda_{0}\right)^{3}-weak\; stoch-2s\\
\partial_{t}E_{\varepsilon}\rightarrow \partial_{t}\overline{E}_{1}+D^{\omega_{0}}\overline{E}^{s}_{1}+\partial_{\tau}\overline{E}^{\sigma}_{1}\; in\; L^{2}\left(Q\times\Lambda\times\Lambda_{0}  \right)^{3}-weak\; stoch -2s\\
\partial_{t}H_{\varepsilon}\rightarrow \partial_{t}\overline{H}_{1}+D^{\omega_{0}}\overline{H}^{s}_{1}+\partial_{\tau}\overline{H}^{\sigma}_{1}\; in\; L^{2}\left(Q\times\Lambda\times\Lambda_{0}  \right)^{3}-weak\; stoch -2s
\end{array}
 \end{equation}
Moreover, by setting
\begin{equation}\nonumber
\begin{array}{l}
E_{0}=\int\!\!\int_{\Lambda_{0}\times\Theta}\overline{E}_{0}d\mu^{0}d\tau,\; H_{0}=\int\!\!\int_{\Lambda_{0}\times\Theta}\overline{H}_{0}d\mu^{0}d\tau\in L^{2}\left(0, T; H_{\mathrm{curl}}(\Omega)  \right),\\
e^{s}_{0}=\int\!\!\int_{\Lambda_{0}\times \Theta}\overline{e}^{s}_{0}d\mu^{0}d\tau, h^{s}_{0}=\int\!\!\int_{\Lambda_{0}\times\Theta}\overline{h}^{s}_{0}d\mu^{0}d\tau \in H_{\mathrm{curl}}\big(Q; H^{1}_{\#}(\Lambda)\big),\\
e^{\sigma}_{0}=\int\!\!\int_{\Lambda_{0}\times \Theta}\overline{e}^{\sigma}_{0}d\mu^{0}d\tau, h^{\sigma}_{0}=\int\!\!\int_{\Lambda_{0}\times\Theta}\overline{h}^{\sigma}_{0}d\mu^{0}d\tau \in L^{2}\big(Q\times\Lambda; H^{1}_{\#}(Z)\big),\\
E_{1}=\int_{\Lambda_{0}}\overline{E}_{1}d\mu^{0}, H_{1}=\int_{\Lambda_{0}}\overline{H}_{1}d\mu^{0}\in H^{1}\big(0, T; L^{2}\big(\Omega\times\Lambda;\; L^{2}_{per}(Z)^{3}\big)\big),\\
D^{\omega_{0}}E^{s}_{1}=\int_{\Lambda_{0}}D^{\omega_{0}}\overline{E}^{s}_{1}d\mu^{0},\; D^{\omega_{0}}H^{s}_{1}=\int_{\Lambda_{0}}D^{\omega_{0}}\overline{H}^{s}_{1}d\mu^{0},\\
E^{\sigma}_{1}=\int\!\!\int_{\Lambda_{0}\times \Theta}\overline{E}^{\sigma}_{1}d\mu^{0}d\tau\in L^{2}\big(Q\times \Lambda; L^{2}_{per}(Z)^{3}\big),\\
H^{\sigma}_{1}=\int\!\!\int_{\Lambda_{0}\times \Theta}\overline{H}^{\sigma}_{1}d\mu^{0}d\tau\in L^{2}\big(Q\times \Lambda; L^{2}_{per}(Z)^{3}\big),\\
J_{0}=\int\!\!\int_{\Lambda_{0}\times\Theta}\overline{J}_{0}d\mu^{0}d\tau,
\end{array}
\end{equation}
we have
\begin{equation}\nonumber
\begin{array}{l}
E_{1}=E_{0}+D_{\omega}e^{s}_{0}+D_{y}e^{\sigma}_{0},\; H_{1}=H_{0}+D_{\omega}h_{0}^{s}+D_{y}h^{\sigma}_{0},\\
D^{\omega_{0}}E^{s}_{1}=D^{\omega_{0}}H^{s}_{1}=\partial_{\tau}E^{\sigma}_{1}=\partial_{\tau}H^{\sigma}_{1}=O,
\end{array}
\end{equation}
and for all $f=f(x, t, \omega, z)\in L^{2}\left(Q; \mathcal{C}^{\infty}(\Lambda; \mathcal{C}_{per}(Z))  \right)^{3},$ when $E\ni \varepsilon \longrightarrow 0,$
\begin{equation}\label{4.1}
\begin{array}{l}
\int\!\!\int_{Q\times\Lambda}E_{\varepsilon}\cdot f^{\varepsilon}dxdtd\mu \rightarrow \int\!\!\int\!\!\int_{Q\times\Lambda\times Z}\left(E_{0}+D_{\omega}e^{s}_{0}
+D_{y}e^{\sigma}_{0}\right)\cdot fdxdtd\mu dz, \\

\int\!\!\int_{Q\times\Lambda}H_{\varepsilon}\cdot f^{\varepsilon}dxdtd\mu \rightarrow \int\!\!\int\!\!\int_{Q\times\Lambda\times Z}\left(H_{0}+D_{\omega}h^{s}_{0}
+D_{y}h^{\sigma}_{0}\right)\cdot fdxdtd\mu dz,\\
\int\!\!\int_{Q\times\Lambda}J_{\varepsilon}\cdot f^{\varepsilon}dxdtd\mu \rightarrow \int\!\!\int\!\!\int_{Q\times\Lambda\times Z}J_{0}\cdot fdxdtd\mu dz,\\
\int\!\!\int_{Q\times \Lambda}\partial_{t}E_{\varepsilon}\cdot f^{\varepsilon}dxdtd\mu \rightarrow \int\!\!\int\!\!\int_{Q\times\Lambda\times Z}\partial_{t}\left(E_{0}+D_{\omega}e^{s}_{0}+D_{y}e^{\sigma}_{0}\right)\cdot fdxdtd\mu dz, \\
\int\!\!\int_{Q\times\Lambda}\partial_{t}H_{\varepsilon}\cdot f^{\varepsilon}dxdtd\mu \rightarrow \int\!\!\int\!\!\int_{Q\times\Lambda\times Z}\partial_{t}\left(H_{0}+D_{\omega}h^{s}_{0}
+D_{y}h^{\sigma}_{0}\right)\cdot fdxdtd\mu dz.
\end{array}
\end{equation}
\end{lemma}
For more details we refer to \cite{17}.
\begin{lemma}\label{l4.2}
The functions $E_{1}$ and $H_{1}$ define by lemma $(\ref{l4.1})$ satisfies:
\begin{equation}\nonumber
\begin{array}{l}
\mathrm{div}_{z}\left[\mu(x, \omega, \cdot)H_{1}(x, t, \omega, \cdot) \right]=0 \; in\; Z,\; and \\
\mathrm{div}_{z}\left[\eta(x, \omega, \cdot)\partial_{t}E_{1}(x, t, \omega, \cdot)+ J_{0}(x, t, \omega, \cdot)\right]=0 \; in\; Z
\end{array}
\end{equation}
for almost all $(x,t, \omega)\in Q\times\Lambda$; and
\begin{equation}\nonumber
\begin{array}{l}
\mathrm{div}_{\omega}\big[\int_{Z}\mu(x,\cdot,z)H_{1}(x, t, \cdot, z)dz\big]=0 \; in\; \Lambda,\; and\\
\mathrm{div}_{\omega}\big[\int_{Z}\eta(x, \cdot, z)\partial_{t}E_{1}(x, t, \cdot, z)+J_{0}(x, t, \cdot, z)  \big]=0\; in\; \Lambda
\end{array}
\end{equation}
for almost all $(x, t)\in Q.$
\end{lemma}

\begin{proof}
$\mathrm{div}_{z}\left[\mu(x, \omega, \cdot)H_{1}(x,t,\omega, \cdot) \right]=0$ in $Z$ and $\mathrm{div}_{\omega}\big[\int_{Z}\mu(x,\cdot,z)H_{1}(x, t, \cdot, z)dz \big]=0$ in $\Lambda$, are obtained exactly as in \cite{17}. For the two others equalities, one replaces in the proof of lemma \ref{l4.2}, therein $\sigma^{\varepsilon}_{ij}E^{j}_{\varepsilon}$ by $\sigma^{\varepsilon}_{i}(x,.,.,E_{\varepsilon})=J^{i}_{\varepsilon}$
and taking into account lemma \ref{l4.1} see the third convergence in (\ref{4.1}), one gets the expected result.
\end{proof}
\begin{lemma}\label{l4.3}
We also have a.e in $(x,t)\in Q,$
\begin{equation}\nonumber
\begin{array}{l}
\int_{\Lambda}\mathrm{div}\int_{Z}\mu(x,\omega,z)H_{1}(x,t,\omega,z)dz g(\omega)d\mu=0,\; and\\
\int_{\Lambda}\mathrm{div}\big[\int_{Z}\left(\eta(x, \omega,z)\partial_{t}E_{1}(x, t, \omega, z)+J_{0}(x, \omega, z)\right)dz\big]g(\omega)d\mu\\
=\int_{\Lambda}\mathrm{div}F(x, t, \omega)g(\omega)d\mu
\end{array}
\end{equation}
for all $g\in \mathcal{C}^{\infty}(\Lambda)\cap L^{2}_{nv}(\Lambda).$
\end{lemma}
\begin{proof}
The first equality follows exactly as in \cite{17}. The second is obtained from third equation in (\ref{17}) when proceeding as in above case taking into account 
lemma \ref{l4.1} see third convergence therein.
\end{proof}

\begin{proposition}\label{p4.1}
For all $g=(g_{j})\in \left[\mathcal{C}^{\infty}(\Lambda)\cap L^{2}_{nv}(\Lambda)  \right]^{3}$ one has:
\begin{equation}\nonumber
\begin{array}{l}
\int_{\Lambda}\left[\int_{Z}\eta(x, \omega, z)\partial_{t}E_{1}(x,t,\omega)dz\right]\cdot g(\omega)d\mu +\\
\int_{\Lambda}\left[\int_{Z}J_{0}(x, t,\omega)dz\right]\cdot g(\omega)d\mu= \\
\int_{\Lambda}\left[\mathrm{curl}\int_{Z}H_{1}(x,t,\omega)dz  \right]\cdot g(\omega)d\mu +\int_{\Lambda}F(x,t,\omega)\cdot g(\omega)d\mu
\end{array}
\end{equation}
a.e. in Q.
\end{proposition}
\begin{proof}
It is sufficient to replace $\sigma^{\varepsilon}_{ij}E^{j}_{\varepsilon}$ by $\sigma^{\varepsilon}_{i}(x,.,.,E_{\varepsilon})=J^{i}_{\varepsilon}$ in the 
proof of proposition 4.3 in \cite{17} and next take into account lemma \ref{l4.1} see third convergence.
\end{proof}
\begin{proposition}\label{p4.2}
For all $g=(g_{j}) \in \left[\mathcal{C}^{\infty}(\Lambda)\cap L^{2}_{nv}(\Lambda)\right]^{3}$ we have
\[\int_{\Lambda}\left[\int_{Z}\mu(x, \omega,z)\partial_{t}H_{1}(x, t,\omega)dz  \right]\cdot g(\omega)d\mu = -\int_{\Lambda}\left[\mathrm{curl}\int_{Z}
E_{1}(x,t,\omega)dz  \right]\cdot g(\omega)d\mu\]
a.e. in Q
\end{proposition}
\begin{proof}
Proceed exactly as in \cite{17} proposition 4.4.
\end{proof}

\begin{proposition}\label{p4.3}
$E_{1},\; H_{1} \in C\left(\left[0, T  \right]; H_{\mathrm{curl}}\left(\Omega; L^{2}\left(\Lambda; L^{2}_{per}(Z)\right)  \right)  \right)$ with 
$$\int_{\Lambda}\left[\gamma \wedge \widetilde{E_{1}}(x,t, \omega)\right]\cdot g(\omega)d\mu=0 \; a.e.\; on\; \partial\Omega \times (0,T),$$
for all $g\in \left[\mathcal{C}^{\infty}(\Lambda)\cap L^{2}_{nv}(\Lambda)  \right]^{3},$ and 
\begin{equation}\nonumber
\begin{array}{l}
E_{1}(x,z,0,\omega)=E^{0}(x,\omega) \;a.e.\; in \Omega\times\Lambda\times Z,\\
E_{1}(x, z,0,\omega)=H^{0}(x,\omega)\;a.e.\; in \Omega\times\Lambda \times Z,
\end{array}
\end{equation}
where $(\widetilde{E^{j}_{1}})(x,t,\omega)=\int_{Z}E^{j}_{1}(x,t,\omega)dz$, and $\widetilde{E_{1}}=\big(\widetilde{E^{j}_{1}} \big)$
\end{proposition}
\begin{proof}
We proceed as in \cite{17} proof of proposition 4.5, by replacing in (4.20) $G_{1}$ by $G'_{1}$ where $G'_{1}$ is the limit in the sense of (4.11 therein) of 
$\eta* \times \sigma\left(E_{\varepsilon}\right)=\eta*\times J^{\varepsilon}$.
\end{proof}
To identify $J_{0}$ we make use of some preliminary results.
\begin{lemma}\label{l4.4}
For $g=(g_{j})\in \big[\mathcal{C}^{\infty}(\Lambda)\cap L^{2}_{nv}(\Lambda)\big]^{3}$ one has:
\begin{equation}\nonumber
\begin{array}{l}
\int_{\Omega\times\Lambda\times Z}\big[\eta_{ij}(x,\omega,z)\partial_{t}\big(E^{j}_{0}(x,t)+\partial^{\omega}_{j}e^{s}_{0}(x,\omega,t)+\partial_{yj}e^{\sigma}_{0}
(x,\omega,z,t)\big)+J_{0i}(x,\omega,t) \big]E^{i}_{0}(x,t)g_{i}(\omega)dxd\mu dzdt\\ =
\int^{T}_{0}\!\!\int_{\Omega\times\Lambda}\left(\mathrm{curl}(H_{0}+D_{\omega}h^{s}_{0}) \right)_{i}(x,\omega,t)E^{i}_{0}(x,t)g_{i}(\omega)dxd\mu
dt  \\ 
 \;\;\; \; + \int^{T}_{0}\!\!\int_{\Omega\times\Lambda}F_{i}(x,\omega,t)\left(E^{i}_{0}(x,t)\right)g_{i}(\omega)dxd\mu dt.
\end{array}
\end{equation}

\end{lemma}
\begin{proof}
One has already shown (see proposition \ref{p4.1}) that:
\begin{equation}\nonumber
\begin{array}{l}
-\int\!\!\int_{Q\times\Lambda}\eta^{\varepsilon}_{ij}\big(x, \hbox{\LARGE{$\tau$}}\big(\frac{x}{\varepsilon} \big)\omega, \frac{x}{\varepsilon^{2}}\big)
E^{j}_{\varepsilon}(x,t,\omega)\beta'(t)\varphi(x)g_{i}\big(\hbox{\LARGE{$\tau$}}\big(\frac{x}{\varepsilon} \big)\omega  \big)dxdtd\mu\\
+\int\!\!\int_{Q\times\Lambda}\sigma_{i}\big(x,\hbox{\LARGE{$\tau$}}\big(\frac{x}{\varepsilon}\big)\omega,\frac{x}{\varepsilon^{2}}, E_{\varepsilon}(x,t,\omega)\big)\beta(t)\varphi(x)g_{i}\big(\hbox{\LARGE{$\tau$}}\big(\frac{x}{\varepsilon}\big)\omega  \big)dxdtd\mu\\
=-\int\!\!\int_{Q\times\Lambda}H_{\varepsilon}(x,t,\omega)\cdot\beta(t)\mathrm{curl}\left[\varphi\otimes g\right]\big(x,\hbox{\LARGE{$\tau$}}\big(\frac{x}{\varepsilon}\big)\omega\big)dxdtd\mu\\
+\int\!\!\int_{Q\times\Lambda}F_{\varepsilon}(x,t,\omega)\cdot\beta(t)\varphi(x)g\big(\hbox{\LARGE{$\tau$}}\big(\frac{x}{\varepsilon}\big)\omega\big)dxdtd\mu,\; \beta\in\mathcal{D}\big( \left]0,T \right[\big),\; \varphi\in \mathcal{D}(\Omega).
\end{array}
\end{equation}
Let a sequence $\big((\beta\otimes\varphi)^{l} \big)_{l\in \mathbb{N}}\subset \left[\mathcal{D}(\Omega)\otimes \mathcal{D}\left(\left]0,T\right[ \right)  \right]^{3}$, converging to $E_{0}$ in $L^{2}(Q)^{3}$. Replace $\beta\varphi$ by $\big((\beta\otimes\varphi)^{l} \big)_{i}$ in the preceding equality, take limits first on $\varepsilon$, then on $l$ and the result follows.
\end{proof}
\begin{lemma}\label{l4.5}
One also have for $g=(g_{j})\in \left[\mathcal{C}^{\infty}(\Lambda)\cap L^{2}_{nv}(\Lambda)\right]^{3}$, \\
$$\int^{T}_{0}\!\!\int_{\Omega\times\Lambda}F_{i}(x,\omega,t)E^{i}_{0}(x,t)g_{i}(\omega)dxd\mu dt-\int_{\Omega\times\Lambda\times Z}J_{0 i}(x,\omega,z,t)E_{i}(x,t)g_{i}(\omega)dxd\mu dzdt$$ 
$$-\int_{\Omega\times\Lambda\times Z}\eta_{ij}(x,\omega\times z)\times \partial_{t}\big(E_{j}(x)+\partial^{\omega}_{j}e^{s}(x,\omega,t)+\partial_{zj}e^{\sigma}(x,\omega,z,t)\big)\times E^{i}_{0}(x,t)g_{i}(\omega)dxd\mu dzdt$$ 
$$- \int^{T}_{0}\!\!\int_{\Omega\times\Lambda\times Z}\mu_{ij}(x, \omega, z)\times \partial_{t}\big(H^{j}_{0}(x,t)+\partial^{\omega}_{j}h^{s}_{0}(x, \omega,t)+\partial_{zj}h^{\sigma}_{0}(x,\omega, z,t)\big)\times H^{i}_{0}(x,t)g_{i}(\omega)dxd\mu dzdt=0 $$
\end{lemma}
\begin{proof}
Let a sequence $\big((\beta \otimes \varphi)^{l}\big)_{l\in \mathbb{N}}\subset \left[\mathcal{D}(\Omega)\otimes \mathcal{D}(\left]0,T\right[)  \right]^{3}$
, converging to $H_{0}$ in $L^{2}(Q)^{3}$. Let us replace $\beta\varphi$ by $\big((\beta\otimes \varphi)^{l}\big)_{i}$ in the proof of proposition \ref{p4.2}. Taking limit on $\varepsilon$, one gets
$\int^{T}_{0}\!\!\int_{\Omega\times\Lambda\times Z}\mu_{ij}(x,\omega,z)\partial_{t}\left(H_{j}(x,t)+\partial^{\omega}_{j}h^{s}(x,\omega,t)+
\partial_{zj}h^{\sigma}(x,\omega,z,t)\right)\times \big((\beta\otimes \varphi)^{l}\big)_{i}g_{i}(\omega)dxd\mu dzdt=-\int^{T}_{0}\!\!\int_{\Omega\times\Lambda}\left(\mathrm{curl}(E_{0}+D_{\omega}e^{s}_{0}) \right)_{i}(x,t)\big((\beta\otimes \varphi)^{l} \big)_{i}g_{i}(\omega)dxd\mu dt,\; g\in \big[\mathcal{C}^{\infty}(\Lambda)\cap L^{2}_{nv}(\Lambda)\big]^{3}.$\\
Then taking limit on $l$ gives
\begin{equation}\nonumber
\begin{array}{c}
\int^{T}_{0}\!\!\int_{\Omega\times\Lambda\times Y}\mu_{ij}(x,\omega,z)\partial_{t}\big(H_{j}(x,t)+\partial_{j}^{\omega}h^{s}(x,\omega,t)+\partial_{z_{j}}h^{\sigma}(x,\omega,z,t)\big)\times H^{i}_{0}(x,t)g(\omega)dxd\mu dzdt=\\
-\int^{T}_{0}\!\!\int_{\Omega\times\Lambda}\big(\mathrm{curl}\left(E_{0}+D_{\omega}e^{s}_{0}\right)\big)_{i}(x,t)H^{i}_{0}(x,t)g(\omega)dxd\mu dt;\;
g\in\mathcal{C}^{\infty}(\Lambda)\cap L^{2}_{nv}(\Lambda).\\
\int_{\Lambda}\big[\gamma \wedge \tilde{E_{1}}(x,t,\omega)\big]\cdot g(\omega)d\mu=0 \hbox{ a.e. on }  \partial\Omega \times (0, T)\hbox{ leads to }\\
\int\!\!\int_{Q\times \Lambda}\mathrm{curl}\;\tilde{E_{1}}(x,t,\omega)\cdot\big[v_{1}^{l}b^{l}\otimes g^{l}\big](x,\omega)dxd\mu dt-\\
\int\!\!\int_{Q\times \Lambda}\tilde{E_{1}}(x,t,\omega)\cdot \mathrm{curl}\big[v_{1}^{l}b^{l}\otimes g^{l}\big](x,\omega)dxd\mu dt=0,
\end{array}
\end{equation}
taking $v_{1}^{l}b^{l}\otimes g^{l}\in \big[\mathcal{D}(\Omega)\otimes \mathcal{D}(\left]0,T\right[)\otimes \left[\mathcal{C}^{\infty}(\Lambda)\cap  I^{2}_{nv}(\Lambda)\right]  \big]^{3}$ converging to $H_{0}+D_{\omega}h^{s}_{0}$; at the limit one obtain:
\begin{equation}\nonumber
\begin{array}{l}
-\int^{T}_{0}\!\!\int_{\Omega\times\Lambda}\left(\mathrm{curl}(E_{0}+D_{\omega}e^{s}_{0})\right)(x,\omega,t)\cdot \left[H_{0}+D_{\omega}h^{s}_{0}\right]dxd\mu dt\\
+\int^{T}_{0}\!\!\int_{\Omega\times\Lambda}\left(E_{0}+D_{\omega}e^{s}_{0}\right)(x,\omega,t)\cdot \left(\mathrm{curl}\left[H_{0}+D_{\omega h^{s}_{0}}\right]\right)dxd\mu dt=0;  
\end{array}
\end{equation}
and desire result follows.
\end{proof}
\begin{lemma}\label{l4.6}
$J_{0}(x,\omega,z,t)=\sigma(x,\omega,z,E_{0}(x,\omega,t)+D^{\omega}e^{s}_{0}(x,\omega,t)+D_{z}e^{\sigma}_{0}(x,\omega,z,t))$.
\end{lemma}
\begin{proof}
Define test functions as follow
\begin{equation}\nonumber
\begin{array}{c}
v^{\varepsilon}_{i,m}=E^{i}_{0m}+\varepsilon\partial_{xi}\left(e^{s}_{0,m}\right)^{\varepsilon}+\varepsilon^{2}\partial_{xi}\left(e^{\sigma}_{0,m} \right)^{\varepsilon}+r(e_{2,i})^{\varepsilon}\\
u^{\varepsilon}_{i,m}=H^{i}_{0m}+\varepsilon\partial_{xi}\left(h^{s}_{0,m} \right)^{\varepsilon}+\varepsilon^{2}\partial_{xi}\left(h^{\sigma}_{0,m} \right)^{\varepsilon}+r(h_{2,i})^{\varepsilon},
\end{array}
\end{equation}
where, $r$ a non zero real; $E_{0m}, H_{0m}\in \mathcal{D}(Q)^{3}$ converging strongly and respectively to $E_{0}\in L^{2}\big(0,T; L^{2}(\Omega)^{3} \big)$  and to $H_{0}\in L^{2}\big(0,T;L^{2}(\Omega)^{3}\big)$ when $\mathbb{N}\ni m\longrightarrow +\infty; h^{s}_{0,m}\in \mathcal{D}\left(\Omega\times\left]0,T\right[; \mathcal{C}^{\infty}(\Lambda)\right),\; h^{\sigma}_{0,m}\in \mathcal{D}\big(\Omega\times\left]0,T \right[;\big(\mathcal{C}^{\infty}(\Lambda); \mathcal{C}^{\infty}_{per}(Z)\big)  \big)$ converging strongly and respectively to $h^{s}_{0}\in L^{2}\big(Q; H^{1}_{\#}(\Lambda) \big), h^{\sigma}_{0}\in L^{2}\big(Q\times\Lambda; H^{1}_{\#}(Z)\big)$ when $\mathbb{N}\ni m\longrightarrow +\infty; \; 
e_{2,i},h_{2,i}\in \mathcal{C}\big(\overline{Q};\big(\mathcal{C}^{\infty}(\Lambda); \mathcal{C}^{\infty}_{per}(Z)\big)\big)$. It is clear that $v^{\varepsilon}_{i,m}, u^{\varepsilon}_{i,m}\in \mathcal{C}\left(\overline{Q}; \mathcal{C}^{\infty}(\Lambda)\right)$ are test functions strongly stochastic-two-scale converging respectively to $v_{i,m}$ and $u_{i,m}$. Note also that:
\begin{equation}\nonumber
\begin{array}{ll}
\partial_{xi}\left(h^{s}_{0,m}\right)^{\varepsilon}(x,\omega,t)&=\partial_{xi}h^{s}_{0,m}\big(x, \hbox{\LARGE{$\tau$}}\big(\frac{x}{\varepsilon}\big)\omega,t\big)+\frac{1}{\varepsilon}\partial^{\omega}_{i}h^{s}_{0,m}\big(x,\hbox{\LARGE{$\tau$}}\big(\frac{x}{\varepsilon}\big)\omega,t\big);\\
\partial_{xi}\big(e^{s}_{0,m}\big)^{\varepsilon}(x,\omega,t)&=\partial_{xi}e^{s}_{0,m}\big(x,\hbox{\LARGE{$\tau$}}\big(\frac{x}{\varepsilon}\big)\omega,t \big)+\frac{1}{\varepsilon}\partial_{i}^{\omega}e^{s}_{0,m}\big(x,\hbox{\LARGE{$\tau$}}\big(\frac{x}{\varepsilon}\big)\omega,t \big);\\
\partial_{xi}\left(h^{\sigma}_{0,m}\right)^{\varepsilon}(x,\omega,t)&=\partial_{xi}h^{\sigma}_{0,m}\big(x, \hbox{\LARGE{$\tau$}}\big(\frac{x}{\varepsilon}\big)\omega,\frac{x}{\varepsilon^{2}} ,t\big)+\frac{1}{\varepsilon}\partial^{\omega}_{i}h^{\sigma}_{0,m}\big(x,\hbox{\LARGE{$\tau$}}\big(\frac{x}{\varepsilon}\big)\omega,\frac{x}{\varepsilon^{2}},t\big)  \\ 
 &  +\frac{1}{\varepsilon^{2}}\partial_{zi}h^{\sigma}_{0,m}
\big(x, \hbox{\LARGE{$\tau$}}\big(\frac{x}{\varepsilon}\big)\omega, \frac{x}{\varepsilon^{2}},t \big);\\
\partial_{xi}\big(e^{\sigma}_{0,m}\big)^{\varepsilon}(x,\omega,t)&=\partial_{xi}e^{\sigma}_{0,m}\big(x,\hbox{\LARGE{$\tau$}}\big(\frac{x}{\varepsilon}\big)\omega, \frac{x}{\varepsilon^{2}},t \big)+\frac{1}{\varepsilon}\partial_{i}^{\omega}e^{\sigma}_{0,m}\big(x, \hbox{\LARGE{$\tau$}}
\big(\frac{x}{\varepsilon}\big)\omega,\frac{x}{\varepsilon^{2}},t \big)  \\       
&  +\frac{1}{\varepsilon^{2}}\partial_{zi}e^{\sigma}_{0,m}\big(x,\hbox{\LARGE{$\tau$}}\big(\frac{x}{\varepsilon}\big)\omega,\frac{x}{\varepsilon^{2}},t \big);
\end{array}
\end{equation}
then
\begin{equation}\nonumber
\begin{array}{ll}
\varepsilon\partial_{xi}\left(h^{s}_{0,m}\right)^{\varepsilon}(x,\omega,t)&=\varepsilon\partial_{xi}h^{s}_{0,m}\big(x, \hbox{\LARGE{$\tau$}}\big(\frac{x}{\varepsilon}\big)\omega,t\big)+\partial^{\omega}_{i}h^{s}_{0,m}\big(x,\hbox{\LARGE{$\tau$}}\big(\frac{x}{\varepsilon}\big)\omega,t\big);\\
\varepsilon\partial_{xi}\big(e^{s}_{0,m}\big)^{\varepsilon}(x,\omega,t)&=\varepsilon\partial_{xi}e^{s}_{0,m}\big(x,\hbox{\LARGE{$\tau$}}\big(\frac{x}{\varepsilon}\big)\omega,t \big)+\partial_{i}^{\omega}e^{s}_{0,m}\big(x,\hbox{\LARGE{$\tau$}}\big(\frac{x}{\varepsilon}\big)\omega,t \big);\\
\varepsilon^{2}\partial_{xi}\left(h^{\sigma}_{0,m}\right)^{\varepsilon}(x,\omega,t)&=\varepsilon^{2}\partial_{xi}h^{\sigma}_{0,m}\big(x, \hbox{\LARGE{$\tau$}}\big(\frac{x}{\varepsilon}\big)\omega,\frac{x}{\varepsilon^{2}} ,t\big)+\varepsilon\partial^{\omega}_{i}h^{\sigma}_{0,m}\big(x,\hbox{\LARGE{$\tau$}}\big(\frac{x}{\varepsilon}\big)\omega,\frac{x}{\varepsilon^{2}},t\big)  \\ 
&+\partial_{zi}h^{\sigma}_{0,m}\big(x, \hbox{\LARGE{$\tau$}}\big(\frac{x}{\varepsilon}\big)\omega, \frac{x}{\varepsilon^{2}},t \big);\\
\varepsilon^{2}\partial_{xi}\big(e^{\sigma}_{0,m}\big)^{\varepsilon}(x,\omega,t)&=\varepsilon^{2}\partial_{xi}e^{\sigma}_{0,m}\big(x,\hbox{\LARGE{$\tau$}}\big(\frac{x}{\varepsilon}\big)\omega, \frac{x}{\varepsilon^{2}},t \big)+\varepsilon\partial_{i}^{\omega}e^{\sigma}_{0,m}\big(x, \hbox{\LARGE{$\tau$}}\big(\frac{x}{\varepsilon}\big)\omega,\frac{x}{\varepsilon^{2}},t \big)  \\
& +\partial_{zi}e^{\sigma}_{0,m}\big(x,\hbox{\LARGE{$\tau$}}\big(\frac{x}{\varepsilon}\big)\omega,\frac{x}{\varepsilon^{2}},t \big);
\end{array}
\end{equation}
and it follows that when $\varepsilon \longrightarrow 0,$
\begin{equation}\nonumber
\begin{array}{l}
v^{\varepsilon}_{i,m}\longrightarrow E^{i}_{0m}+\partial_{i}^{\omega}e^{s}_{0,m}+\partial_{zi}e^{\sigma}_{0,m}+re_{2,i} \;\hbox{in}\; L^{2}
\left(Q\times\Lambda\times Z \right)\hbox{-strong stoch-2s;}\\
u^{\varepsilon}_{i,m}\longrightarrow H^{i}_{0m}+\partial_{i}^{\omega}h^{s}_{0,m}+\partial_{zi}h^{\sigma}_{0,m}+rh_{2,i} \;\hbox{in}\; L^{2}
\left(Q\times\Lambda\times Z \right)\hbox{-strong stoch-2s;}
\end{array}
\end{equation}
moreover, when $\varepsilon \longrightarrow 0$ and $m\longrightarrow 0,$
\begin{equation}\nonumber
\begin{array}{l}
v^{\varepsilon}_{i,m}\longrightarrow E^{i}_{0}+\partial_{i}^{\omega}e^{s}_{0}+\partial_{zi}e^{\sigma}_{0}+re_{2,i} \;\hbox{in}\; L^{2}
\left(Q\times\Lambda\times Z \right)\hbox{-strong stoch-2s;}\\
u^{\varepsilon}_{i,m}\longrightarrow H^{i}_{0}+\partial_{i}^{\omega}h^{s}_{0}+\partial_{zi}h^{\sigma}_{0}+rh_{2,i} \;\hbox{in}\; L^{2}
\left(Q\times\Lambda\times Z \right)\hbox{-strong stoch-2s.}
\end{array}
\end{equation}
We recall that:
$$F(T)=F(0)+\int^{T}_{0}F'(t)dt,$$
and set:
\begin{equation}\nonumber
\begin{array}{l}
F(T)=\frac{1}{2}\int_{\Omega\times\Lambda}\eta^{\varepsilon}_{ij}\left(E^{j}_{\varepsilon}-v^{\varepsilon}_{j,m} \right)\left(.,.,T \right)\left(E^{i}_{\varepsilon}-v^{\varepsilon}_{i,m} \right)\left(.,.,T \right)dxd\mu \geq \\
\frac{1}{2}c\|\left(E_{\varepsilon}-v^{\varepsilon}_{m} \right)\left(.,.,T \right)\|^{2}_{L^{2}(\Omega\times\Lambda)^{3}}\geq 0\\
F_{1}(T)=\frac{1}{2}\int_{\Omega\times\Lambda}\mu^{\varepsilon}_{ij}\left(H^{j}_{\varepsilon}-u^{\varepsilon}_{j,m} \right)\left(.,.,T \right)\left(H^{i}_{\varepsilon}-u^{\varepsilon}_{i,m}\right)\left(.,.,T \right)dxd\mu\geq \\
\frac{1}{2}c\|\left(H_{\varepsilon}-u^{\varepsilon}_{m}\right)\left(.,.,T\right)\|^{2}_{L^{2}(\Omega\times\Lambda\times ]0,T[ )^{3}}\geq 0,
\end{array}
\end{equation}
$\sigma$ being monotone we have:
\begin{equation}\nonumber
\bigg(\sigma^{\varepsilon}_{i}\bigg(x, \hbox{\LARGE{$\tau$}}\bigg(\frac{x}{\varepsilon}\bigg)\omega, \frac{x}{\varepsilon^{2}}, E_{\varepsilon}\bigg)-
\sigma^{\varepsilon}_{i}\bigg(x,\hbox{\LARGE{$\tau$}}\bigg(\frac{x}{\varepsilon}\bigg)\omega, \frac{x}{\varepsilon^{2}}, v^{2}_{m}\bigg) \bigg)\big(E^{i}_{\varepsilon}-v^{\varepsilon}_{i,m}\big)\geq 0;
\end{equation}
one deduces that;
\begin{equation}\nonumber
\begin{array}{l}
0\leq \int^{T}_{0}\!\!\int_{\Omega\times\Lambda}\big(\sigma^{\varepsilon}_{i}\big(x,\hbox{\LARGE{$\tau$}}\big(\frac{x}{\varepsilon}\big)\omega,\frac{x}{\varepsilon^{2}},E_{\varepsilon} \big)-\sigma^{\varepsilon}_{i}\big(x,\hbox{\LARGE{$\tau$}}\big(\frac{x}{\varepsilon}\big)\omega,\frac{x}{\varepsilon^{2}}, v^{\varepsilon}_{m} \big)\big) \times\big(E^{\varepsilon}_{i}-v^{\varepsilon}_{i,m}\big)dxd\mu dt\\
+\frac{1}{2}\int^{T}_{0}\partial_{t}\int_{\Omega\times\Lambda}\eta^{\varepsilon}_{ij}\big(E^{j}_{\varepsilon}-v^{\varepsilon}_{j,m} \big)\big(E^{i}_{\varepsilon}-v^{\varepsilon}_{i,m} \big)dxd\mu dt\\
+\frac{1}{2}\int_{\Omega\times\Lambda}\eta^{\varepsilon}_{ij}\big(E^{j}_{\varepsilon}(.,.,0)-v^{\varepsilon}_{j,m}(.,.,0)\big)\big(E^{i}_{\varepsilon}
(.,.,0)-v^{\varepsilon}_{i,m}(.,.,0) \big)dxd\mu dt\\
+\frac{1}{2}\int^{T}_{0}\partial_{t}\int_{\Omega\times\Lambda}\mu^{\varepsilon}_{ij}\big(H^{j}_{\varepsilon}-u^{\varepsilon}_{j,m}\big)\big(H^{i}_{\varepsilon}-u^{\varepsilon}_{i,m}\big)dxd\mu\\
+\frac{1}{2}\int_{\Omega\times\Lambda}\mu^{\varepsilon}_{ij}\big(H^{j}_{\varepsilon}(.,.,0)-u^{\varepsilon}_{j,m}(.,.,0) \big)\big(H^{i}_{\varepsilon}
(.,.,0)-u^{\varepsilon}_{i,m}(.,.,0) \big)dxd\mu\\
=\int^{T}_{0}\!\!\int_{\Omega\times\Lambda}\Biggl[\underbrace{J^{i}_{\varepsilon}E^{i}_{\varepsilon}}_{1} -\underbrace{\sigma^{\varepsilon}_{i}\big(x, \hbox{\LARGE{$\tau$}}\big(\frac{x}{\varepsilon}\big)\omega,\frac{x}{\varepsilon^{2}}, v^{\varepsilon}_{m}\big)E^{\varepsilon}_{i}}_{2}-\underbrace{J^{i}_{\varepsilon}v^{\varepsilon}_{i,m}}_{3} \\
+\underbrace{\sigma^{\varepsilon}_{i}\big(x,\hbox{\LARGE{$\tau$}}\big(\frac{x}{\varepsilon}\big)\omega, \frac{x}{\varepsilon^{2}}, v^{\varepsilon}_{m}\big)v^{\varepsilon}_{i,m}}_{4}\Biggr]dxd\mu dt \\
+\underbrace{\frac{1}{2}\int^{T}_{0}\partial_{t}\int_{\Omega\times\Lambda}\left(\eta^{\varepsilon}_{i,j}E^{j}_{\varepsilon}E^{i}_{\varepsilon}+\mu^{\varepsilon}_{i,j}
H^{j}_{\varepsilon}H^{i}_{\varepsilon}\right)dxd\mu dt}_{5}\\
-\underbrace{\frac{1}{2}\int^{T}_{0}\partial_{t}\int_{\Omega\times\Lambda}\left(\eta^{\varepsilon}_{i,j}E^{j}_{\varepsilon}v^{\varepsilon}_{i,m}
+\mu^{\varepsilon}_{i,j}H^{j}_{\varepsilon}u^{\varepsilon}_{i,m}\right)dxd\mu dt}_{6}\\
+\underbrace{\frac{1}{2}\int^{T}_{0}\partial_{t}\int_{\Omega\times\Lambda}\left(\eta^{\varepsilon}_{ij}v^{\varepsilon}_{j,m}v^{\varepsilon}_{i,m}+\mu^{\varepsilon}_{ij}
u^{\varepsilon}_{j,m}u^{\varepsilon}_{i,m}\right)dxd\mu dt}_{7}\\
+\frac{1}{2}\underbrace{\int_{\Omega\times\Lambda}\eta^{\varepsilon}_{ij}\left(E^{j}_{\varepsilon}(.,.,0)-v^{\varepsilon}_{j,m}(.,.,0) \right)\big(E^{i}_{\varepsilon}(.,.,0)-v^{\varepsilon}_{i,m}(.,.,0) \big)dxd\mu}_{8}\\
+\frac{1}{2}\underbrace{\int_{\Omega\times\Lambda}\mu^{\varepsilon}_{ij}\big(H^{j}_{\varepsilon}(.,.,0)-u^{\varepsilon}_{j,m}(.,.,0) \big)\left(E^{i}_{\varepsilon}(.,.,0)-u^{\varepsilon}_{i,m}(.,.,0)\right)dxd\mu}_{9}\\
=\int^{T}_{0}\!\!\int_{\Omega\times\Lambda}\bigg[ \underbrace{J^{i}_{\varepsilon}E^{i}_{\varepsilon}}_{1}-\underbrace{\sigma^{\varepsilon}_{i}
\big(x, \hbox{\LARGE{$\tau$}}\big(\frac{x}{\varepsilon}\big)\omega, \frac{x}{\varepsilon^{2}},v^{\varepsilon}_{m}\big)E^{i}_{\varepsilon}}-
\underbrace{J^{i}_{\varepsilon}v^{\varepsilon}_{i,m}}_{3} \\
+\underbrace{\sigma^{\varepsilon}_{i}\big(x\hbox{\LARGE{$\tau$}}\big(\frac{x}{\varepsilon}\big)\omega, \frac{x}{\varepsilon^{2}},v^{\varepsilon}_{m} \big)v^{\varepsilon}_{i,m}}_{4}\bigg]dxd\mu dt\\
+\underbrace{\int^{T}_{0}\int_{\Omega\times\Lambda}\left(\eta^{\varepsilon}_{ij}\partial_{t}(E^{j}_{\varepsilon})E^{i}_{\varepsilon}+\mu^{\varepsilon}_{ij}\partial_{t}(H^{j}_{\varepsilon})H^{i}_{\varepsilon}\right)dxd\mu dt}_{5}\\
-\frac{1}{2}\underbrace{\int^{T}_{0}\!\!\int_{\Omega\times\Lambda}\eta^{\varepsilon}_{ij}\partial_{t}(E^{j}_{\varepsilon})v^{\varepsilon}_{i,m}dxd\mu dt}_{6.1}-\frac{1}{2}\underbrace{\int^{T}_{0}\!\!\int_{\Omega\times\Lambda}\eta^{\varepsilon}_{ij}E^{j}_{\varepsilon}\partial_{t}(v^{\varepsilon}_{i,m})
dxd\mu dt}_{6.2}\\
-\frac{1}{2}\underbrace{\int^{T}_{0}\!\!\int_{\Omega\times\Lambda}\mu^{\varepsilon}_{ij}\partial_{t}(H^{j}_{\varepsilon})u^{\varepsilon}_{i,m}dxd\mu dt}_{6.3}-\frac{1}{2}\underbrace{\int^{T}_{0}\!\!\int_{\Omega\times\Lambda}\mu^{\varepsilon}_{ij}H^{j}_{\varepsilon}\partial_{t}(u^{\varepsilon}_{i,m})
dxd\mu dt}_{6.4}

\end{array}
\end{equation}

\begin{flushleft}
\begin{equation*} 
\begin{array}{l}
+\underbrace{\int^{T}_{0}\!\!\int_{\Omega\times\Lambda}\left(\eta^{\varepsilon}_{ij}\partial_{t}\left(v^{\varepsilon}_{j,m}\right)v^{\varepsilon}_{i,m}
	+\mu^{\varepsilon}_{ij}\partial_{t}\left(u^{\varepsilon}_{j,m}\right)u^{\varepsilon}_{i,m}\right)dxd\mu dt }_{7} \\
+\frac{1}{2}\underbrace{\int_{\Omega\times\Lambda}\eta^{\varepsilon}_{ij}\left(E^{j}_{\varepsilon}(.,.,0)-v^{\varepsilon}_{j,m}(.,.,0)\right)\left(E^{i}_{\varepsilon}(.,.,0)-v^{\varepsilon}_{i,m}(.,.,0) \right)dxd\mu }_{8}\\
+\frac{1}{2}\underbrace{\int_{\Omega\times\Lambda}\mu^{\varepsilon}_{ij}\left(H^{j}_{\varepsilon}(.,.,0)-u^{\varepsilon}_{j,m}(.,.,0)\right)\left(H^{i}_{\varepsilon}(.,.,0)-u^{\varepsilon}_{i,m}(.,.,0) \right)dxd\mu }_{9}.
\end{array}
\end{equation*}
\end{flushleft}

For $\varepsilon \longrightarrow 0$, one gets at the limit:
\begin{equation}\nonumber
\begin{array}{l}
8\longrightarrow \frac{1}{2} \int_{\Omega\times\Lambda\times Y}\eta_{ij}r^{2}e_{2,j}(x,\omega,z,0)e_{2,i}(x,\omega,z,0)dxd\mu\\
9\longrightarrow \frac{1}{2} \int_{\Omega\times\Lambda\times Y}\mu_{ij}r^{2}h_{2,j}(x,\omega,z,0)h_{2,i}(x,\omega,z,0)dxd\mu \\
7\longrightarrow \int^{T}_{0}\!\!\int_{\Omega\times\Lambda\times Y}\eta_{ij}\partial_{t}\big(E^{j}_{0m}(x,t)+\partial^{\omega}_{j}e^{s}_{0,m}(x,\omega,t)+\partial_{zj}e^{\sigma}_{0,m}(x,\omega,z,t)+re_{2,j}(x,\omega,z,t)\big) \times \\
\big(E^{j}_{0m}(x,t)+\partial^{\omega}_{i}e^{s}_{0,m}(x,\omega,t)+\partial_{zi}e^{\sigma}_{0,m}(x,\omega,z,t)+re_{2,i}(x,\omega,z,t) \big)dxd\mu 
dzdt\\
+\int^{T}_{0}\!\!\int_{\Omega\times\Lambda\times Z}\mu_{ij}\partial_{t}\big(H^{j}_{0m}(x,t)+\partial_{j}^{\omega}h^{s}_{0}(x,\omega,t)+
\partial_{zj}h^{\sigma}_{0}(x,\omega,z,t) +rh_{2,j}(x,\omega,z,t)\big)\times \\
H^{i}_{0m}(x,t)+\partial_{i}^{\omega}h^{s}_{0}(x,\omega,t)+\partial_{zi}h^{\sigma}_{0}(x,\omega,z,t)+rh_{2,i}(x,\omega,z,t)dxd\mu dzdt \\
6.1\longrightarrow \int^{T}_{0}\!\!\int_{\Omega\times\Lambda\times Z}\eta_{ij}\partial_{t}\big(E_{0j}(x,t)+\partial^{\omega}_{j}e^{s}_{0}(x,\omega,t)+\partial_{z_{j}}e^{\sigma}_{0}(x,\omega,z,t) \big)\times\\
\big(E^{i}_{0m}(x,t)+\partial^{\omega}_{i}e^{s}_{0,m}(x,\omega,t)+\partial_{z_{i}}e^{\sigma}_{0,m}(x,\omega,z,t)+re_{2,i}(x,\omega,z,t)\big)dxd\mu dzdt\\
6.2\longrightarrow \int^{T}_{0}\!\!\int_{\Omega\times\Lambda\times Z}\eta_{ij}\big(E_{0j}(x,t)+\partial_{j}^{\omega}e^{s}_{0}(x,\omega,t)+\partial_{z_{j}}e^{\sigma}_{0}(x,\omega,z,t) \big)\times \\
\partial_{t}\big(E^{i}_{0m}(x,t)+\partial_{i}^{\omega}e^{s}_{0,m}(x,\omega,t)+\partial_{z_{i}}e^{\sigma}_{0,m}(x,\omega,z,t)+re_{2,i}(x,\omega,z,t) \big)dxd\mu dzdt \\
6.3 \longrightarrow \int^{T}_{0}\!\!\int_{\Omega\times\Lambda\times Z}\mu_{ij}\partial_{t}\big(H_{0j}(x,t)+\partial^{\omega}_{j}h^{s}_{0}(x,\omega,t)+\partial_{z_{j}}h^{\sigma}_{0}(x,\omega,z,t)\big) \times\\
\big(H^{i}_{0m}(x,t)+\partial^{\omega}_{i}h^{s}_{0}(x,\omega,t)+\partial_{z_{i}}h^{\sigma}_{0}(x,\omega,z,t)+rh_{2,i}(x,\omega,z,t) \big)dxd\mu dzdt\\
6.4\longrightarrow \int^{T}_{0}\!\!\int_{\Omega\times\Lambda}\mu_{ij}\big(H_{0j}(x,t)+\partial_{j}^{\omega}h^{s}_{0}(x,\omega,t)+\partial_{z_{j}}h^{\sigma}_{0}(x,\omega,z,t) \big)\times \\
\partial_{t}\big(H^{i}_{0m}(x,t)+\partial^{\omega}_{i}h^{s}_{0,m}(x,\omega,t)+\partial_{z_{i}}h^{\sigma}_{0,m}(x,\omega,z,t)+rh_{2,i}(x,\omega,z,t) \big)dxd\mu dzdt\\
4\longrightarrow \int^{T}_{0}\!\!\int_{\Omega\times\Lambda\times Y}\sigma_{i}\big(x,\omega,z,E_{0m}(x,t)+D^{\omega}e^{s}_{0,m}(x,\omega,t)+D_{z}e^{\sigma}_{0,m}(x,\omega,z,t)+re_{2}(x,\omega,z,t) \big) \times \\
\big(E^{i}_{0m}(x,t)+\partial^{\omega}_{i}e^{s}_{0,m}(x,\omega,t)+\partial_{z_{i}}e^{\sigma}_{0,m}(x,\omega,z,t)+re_{2,i}(x,\omega,z,t)\big)dxd\mu dzdt\\
3\longrightarrow \int^{T}_{0}\!\!\int_{\Omega\times\Lambda\times Y}J_{0i} \times \big(E^{i}_{0m}(x,t)+\partial^{\omega}_{i}e^{s}_{0,m}(x,\omega,t)+\partial_{z_{i}}e^{\sigma}_{0,m}(x,\omega,z,t)+re_{2,i}(x,\omega,z,t) \big)dxd\mu dzdt\\
2\longrightarrow \int^{T}_{0}\!\!\int_{\Omega\times\Lambda\times Y}\sigma_{i}\big(x,\omega,z, E_{0m}(x,t)+\partial^{\omega}e^{s}_{0,m}(x,\omega,t)+D_{z}e^{\sigma}_{0,m}(x,\omega,z,t)+re_{2}(x,\omega,z,t)\big)\times\\
\big(E_{0j}(x,t)+\partial_{j}^{\omega}e^{s}_{0}(x,\omega,t)+\partial_{z_{j}}e^{\sigma}_{0}(x,\omega,z,t) \big)dxd\mu dzdt \\
1+5 \longrightarrow \int^{T}_{0}\!\!\int_{\Omega\times\Lambda}F_{i}(x,\omega,t)\big(E_{0i}(x,t)+\partial^{\omega}_{i}e^{s}_{0}(x,\omega,t)\big)dxd\mu dzdt \\
\hbox{For convergence of 1+5 one takes in account the relation } \mathcal{A}(U,U)=0.
\end{array}
\end{equation}
The above convergences are justified by properties of stochastic two scale convergence presented in section 2. In particular, convergences 8 and 9 are consequences of proposition \ref{p4.3}; convergence 4 is a consequence of lemmas (\ref{17}, \ref{l3.2}) and the equality below:  
\begin{equation}\nonumber
\begin{array}{l}
\sigma_{i}\big(x,\hbox{\LARGE{$\tau$}}\big(\frac{x}{\varepsilon} \big)\omega,\frac{x}{\varepsilon^{2}},v^{\varepsilon}_{m} \big)v^{\varepsilon}_{i,m}=
\big[ \sigma_{i}\big(x,\hbox{\LARGE{$\tau$}}\big(\frac{x}{\varepsilon}\big)\omega, \frac{x}{\varepsilon^{2}},v^{\varepsilon}_{m} \big)-\\
\sigma_{i}\big(x,\hbox{\LARGE{$\tau$}}\big(\frac{x}{\varepsilon}\big)\omega, \frac{x}{\varepsilon^{2}}, E^{m}_{i}(x,t)+\partial^{\omega}_{i}e^{s}_{1,m} \big(x,\hbox{\LARGE{$\tau$}}\big(\frac{x}{\varepsilon}\big)\omega, t \big)+\\
\partial_{zi}e^{\sigma}_{1,m}\big(x, \hbox{\LARGE{$\tau$}}\big(\frac{x}{\varepsilon}\big)\omega, \frac{x}{\varepsilon^{2}},t \big) +re_{2,i}\big(x,\hbox{\LARGE{$\tau$}}\big(\frac{x}{\varepsilon} \big)\omega, \frac{x}{\varepsilon^{2}},t \big)\big)  \big]v^{\varepsilon}_{i,m}\\
+\big[\sigma_{i}\big(x,\hbox{\LARGE{$\tau$}}\big(\frac{x}{\varepsilon}\big)\omega,\frac{x}{\varepsilon^{2}},E^{m}_{i}(x,t)+\partial^{\omega}_{i}e^{s}_{1,m}
\big(x,\hbox{\LARGE{$\tau$}}\big(\frac{x}{\varepsilon} \big)\omega,t\big)+\\
\partial_{z_{i}}e^{\sigma}_{1,m}\big(x,\hbox{\LARGE{$\tau$}}\big(\frac{x}{\varepsilon}\big)\omega, \frac{x}{\varepsilon^{2}},t \big)+re_{2,i}\big(x,\hbox{\LARGE{$\tau$}}\big(\frac{x}{\varepsilon}\big)\omega, \frac{x}{\varepsilon^{2}},t \big) \big)\big]v^{\varepsilon}_{i,m}.
\end{array}
\end{equation}
A similar process gives the convergence of 2. Next, take limit as $m\longrightarrow \infty$, taking in account lemmas (\ref{l4.2}, \;\ref{l4.3} and \ref{l4.5}), one gets: 
\begin{equation}\nonumber
\begin{array}{l}
\int^{T}_{0}\!\!\int_{\Omega\times\Lambda\times Z}\big[\sigma\big(x,\omega,z,E_{0}(x,t)+D^{\omega}e^{s}_{0}(x,\omega,t)+D_{z}e^{\sigma}_{0}(x,\omega,z,t)+re_{2}(x,\omega,z,t) \big)\\-J_{0}(x,\omega,z,t) \big]\times re_{2}(x,\omega,z,t)dxd\mu dzdt\\
+\int_{\Omega\times\Lambda\times Z}\eta_{ij}r^{2}e_{2,j}(x,\omega,z,0)e_{2,i}(x,\omega,z,0)dxd\mu dz \\
\frac{1}{2}\int^{T}_{0}\partial_{t}\int_{\Omega\times\Lambda\times Z}\eta_{ij}r^{2}e_{2,j}(x,\omega,z,t)e_{2,i}(x,\omega,z,t)dxd\mu dzdt\\
+\int_{\Omega\times\Lambda\times Z}\mu_{ij}r^{2}h_{2,j}(x,\omega,z,0)h_{2,i}(x,\omega,z,0)dxd\mu dz \\
+\frac{1}{2}\int^{T}_{0}\partial_{t}\int_{\Omega\times\Lambda\times Z}\mu_{ij}r^{2}h_{2,j}(x,\omega,z,t)h_{2,i}(x,\omega,z,t)dxd\mu dzdt\geq 0.

\end{array}
\end{equation}
We multiply by $\frac{1}{r}(r>0)$ and letting r go to zero to obtain:
\begin{equation}\nonumber
\begin{array}{l}
\int_{\Omega\times\Lambda\times Z}\left(\sigma\left(x,\omega,z,E_{0}(x,t)+D^{\omega}e^{s}_{0}(x,\omega,t)+D_{z}e^{\sigma}_{0}(x,\omega,z,t) \right)
-J_{0}(x,\omega,z,t)\right)\times\\
e_{2}(x,\omega,z,t)dxd\mu dzdt\geq 0,\; \forall e_{2}\in \mathcal{D}\big(Q;\big(\mathcal{C}^{\infty}(\Lambda); \mathcal{C}^{\infty}_{\#}(Z) \big) \big)^{3}.
\end{array}
\end{equation}
Restarting the above procedure for $(r<0)$ leads to:
\begin{equation}\nonumber
\begin{array}{l}
\int_{\Omega\times\Lambda\times Z}\left(\sigma\left(x,\omega,z,E_{0}(x,t)+D^{\omega}e^{s}_{0}(x,\omega,t)+D_{z}e^{\sigma}_{0}(x,\omega,z,t) \right)
-J_{0}(x,\omega,z,t)\right)\times\\
e_{2}(x,\omega,z,t)dxd\mu dzdt\geq 0,\; \forall e_{2}\in \mathcal{D}\big(Q;\big(\mathcal{C}^{\infty}(\Lambda); \mathcal{C}^{\infty}_{\#}(Z) \big) \big)^{3}.
\end{array}
\end{equation}
Then,
\begin{equation}\nonumber
\begin{array}{l}
\sigma(x,\omega,z,E_{0}(x,t)+D^{\omega}e^{s}_{0}(x,\omega,t)+D_{z}e^{\sigma}_{0}(x,\omega,z,t))=J_{0}(x,\omega,z,t)
\end{array}
\end{equation}
almost everywhere in $Q\times\Lambda\times Z$.
\end{proof}

\subsection{\textbf{Homogenization of Results.}}
\begin{theorem}\label{t4.1}
Let $(E_{\varepsilon},H_{\varepsilon})_{0<\varepsilon \leq 1}$ the sequence defined by theorem \ref{t3.1}, assume that \hbox{\LARGE{$\tau$}} is ergodic. Then, 
when $\varepsilon \longrightarrow 0$, we have:
\begin{equation}\nonumber
\left\{\begin{array}{l}
E_{\varepsilon} \longrightarrow E_{0} \\
H_{\varepsilon} \longrightarrow  H_{0}
\end{array}
in\; L^{\infty}(0, T; H_{\mathrm{curl}}(\Omega))\hbox{ -weak*},
\right.
\end{equation}
and
\begin{equation}\nonumber
\left\{\begin{array}{l}
\partial_{t}E_{\varepsilon}\longrightarrow \partial_{t}E_{0}\\
\partial_{t}H_{\varepsilon} \longrightarrow \partial_{t}H_{0}
\end{array}
in\; L^{\infty}\big(0,T; L^{2}(\Omega)^{3}\big)\hbox{ -weak*},
\right.
\end{equation}
the above quantities being linked by the Maxwell system:
\begin{equation}\nonumber
\left\{\begin{array}{r c l}
\partial_{t}D_{0}(x,t)+J_{0}(x,t)&=&\mathrm{curl}H_{0}(x,t)+\underline{F}(x,t)\\
\partial_{t}B_{0}(x,t)&=&-\mathrm{curl}E_{0}(x,t)\\
\mathrm{div}B_{0}(x,t)&=&0\\
\mathrm{div}D_{0}(x,t)&=&Q_{0}(x,\omega,t)\\
\gamma\wedge E_{0}(x,t)&=&0 \; on\; \partial\Omega\times ]0,T[ 
\end{array}
\right.
\end{equation}
with initial conditions
\begin{equation}\nonumber
\left\{\begin{array}{l}
E_{0}(x, 0)=\underline{E}^{0}(x),\; and\\
H_{0}(x,0)=\underline{H}^{0}(x)
\end{array}
\right.
\end{equation}
where
\begin{equation}\nonumber
\begin{array}{l}
B^{i}_{0}(x,t)=\int_{\Lambda}\!\int_{Z}\mu_{ij}(x,\omega,z)\big[H^{j}_{0}(x,t)+\partial^{\omega}_{j}h^{s}_{0}(x,\omega,t)+\partial_{z_{j}}h^{\sigma}
_{0}(x,\omega,z,t) \big]dzd\mu \\
J^{i}_{0}(x,t)=\int_{\Lambda}\!\int_{Z}\sigma_{i}(x,\omega,z,E_{0}(x,t)+D^{\omega}e^{s}_{0}(x,\omega,t)+D_{z}e^{\sigma}_{0}(x,\omega,z,t))dzd\mu\\
D^{i}_{0}(x,t)=\int_{\Lambda}\!\int_{Z}\eta_{ij}(x,\omega,z)\big[E^{j}_{0}(x,t)+\partial^{\omega}_{j}e^{s}_{0}(x,\omega,t)+\partial_{z_{j}}e^{\sigma}
_{0}(x,\omega,z,t) \big]dzd\mu.
\end{array}
\end{equation}
with in addition,
\begin{equation}\nonumber
\begin{array}{l}
\int_{Z}\big[\eta_{ij}(x,\omega,z)\partial_{t}\big(E^{j}_{0}(x,t)+\partial^{\omega}_{j}e^{s}_{0}(x,\omega,t)+\partial_{z_{j}}e^{\sigma}_{0}(x,\omega,z,t) \big)\\
+\sigma_{i}(x,\omega,z, E_{0}(x,t)+\partial^{\omega}e^{s}_{0}(x,\omega,t)+D_{z}e^{\sigma}_{0}(x,\omega,z,t)) \big]\times \partial_{z_{i}}v_{3}(z)dz=0,\\ 
for \; all \; v_{3}\in W^{1,2}_{\#}(Z); \\\\

\int\!\int_{\Lambda\times Z}\big[ \eta_{ij}(x,\omega,z)\partial_{t}\big(E^{j}_{0}(x,t)+\partial^{\omega}_{j}e^{s}_{0}(x,\omega,t)+\partial_{z_{j}}e^{\sigma}_{0}(x,\omega,z,t) \big)\\
+\sigma_{i}(x,\omega,z, E_{0}(x,t)+\partial^{\omega}e^{s}_{0}(x,\omega,t)+D_{z}e^{\sigma}_{0}(x,\omega,z,t)) \big]\times \partial^{\omega}_{j}v_{2}dzd\mu=0 \\
for\; all\; v_{2}\in \mathcal{H}_{\#}(\Lambda);
\end{array}
\end{equation}
and
\begin{equation}\nonumber
\begin{array}{l}
\int_{Z}\mu_{ij}(x,\omega,z)\times \big(H^{j}(x,t)+\partial^{\omega}_{j}h^{s}(x,\omega,t)+\partial_{z_{j}}h^{\sigma}(x,\omega,z,t)\big)\partial_{z_{i}}
v_{3}(z)dz=0,\\
for\; all\; v_{3}\in W^{1,2}_{\#}(Z);\\
\int\!\int_{\Lambda\times Z}\mu_{ij}(x,\omega,z)\times \big(H^{j}(x,t)+\partial^{\omega}_{j}h^{s}(x,\omega,t)+\partial_{z_{j}}h^{\sigma}(x,\omega,z,t) \big)\partial_{i}^{\omega}v_{2}dzd\mu=0,\\
for\; all\; v_{2}\in \mathcal{H}_{\#}(\Lambda).
\end{array}
\end{equation}
 \end{theorem}

\begin{proof}
It is sufficient to realize that in the proof of theorem (4.2) below in the case the dynamic system, is ergodic, $v_{2}\in \mathcal{C}^{\infty}(\Lambda)\cap I^{2}_{nv}(\Lambda)$ is a constant.
\end{proof}

\begin{theorem}\label{t4.2}
Let $(E_{\varepsilon}, H_{\varepsilon})_{0<\varepsilon\leq 1}$ the sequence defined by  theorem \ref{t3.1}, assume \hbox{\LARGE{$\tau$}} is not ergodic.
Then, when $\varepsilon \longrightarrow 0$, we get:
\begin{equation}\nonumber
\left\{\begin{array}{l}
E_{\varepsilon} \longrightarrow E_{0}+D_{\omega}e^{s}_{0}\\
H_{\varepsilon} \longrightarrow H_{0}+D_ {\omega}h^{s}_{0}
\end{array}
in \;L^{\infty}\big(0,T;H_{\mathrm{curl}}\left(\Omega;L^{2}_{nv}(\Lambda)\right) \big)\hbox{ -weak*},
\right.
\end{equation}
and
\begin{equation}\nonumber
\left\{\begin{array}{l}
\partial_{t}E_{\varepsilon}\longrightarrow \partial_{t}[E_{0}+D_{\omega}e^{s}_{0}]\\
\partial_{t}H_{\varepsilon}\longrightarrow \partial_{t}[H_{0}+D_{\omega}h^{s}_{0}]
\end{array}
in\; L^{\infty}\big(0,T;L^{2}\left(\Omega;L^{2}_{nv}(\Lambda)\right)^{3} \big)\hbox{ -weak*},
\right.
\end{equation}
the above quantities being linked by the Maxwell system:
\begin{equation}\label{4.2}
\left\{\begin{array}{r c l}
\partial_{t}D_{0}(x,\omega,t)+J_{0}(x,\omega,t)&=&\mathrm{curl}\left(H_{0}(x,t)+D^{\omega}h^{s}_{0}(x,\omega,t) \right)+F(x,\omega,t)\; in\; Q\times \Lambda\\
\partial_{t}B_{0}(x,\omega,t)&=&-\mathrm{curl}\left(E_{0}(x,t)+D^{\omega}e^{s}_{0}(x,\omega,t) \right)\; in\; Q\times \Lambda \\
\mathrm{div}B_{0}(x,\omega,t)&=&0\; in\; Q\times \Lambda \\
\mathrm{div}D_{0}(x,\omega,t)&=&Q_{0}(x,\omega,t)\; in\; Q\times \Lambda\\
\gamma\wedge\left(E_{0}(x,t)+D^{\omega}e^{s}_{0}(x,\omega,t) \right)&=&0\; on\; \partial\Omega\times\Lambda\times ]0,T[

\end{array}
\right.
\end{equation}
with initial conditions
\begin{equation}\nonumber
\left\{\begin{array}{l}
[E_{0}+D_{\omega}e^{s}_{0}](x,0)=E^{0}(x),\; and\\

[H_{0}+D_{\omega}h^{s}_{0}](x,0)=H^{0}(x)\; on\; L^{2}_{nv}(\Lambda; \mathbb{R})
\end{array}
\right.
\end{equation}
where
\begin{equation}\label{4.3}
\begin{array}{r c l}
B^{i}_{0}(x,\omega,t)&=&\int_{Z}\mu_{ij}(x,\omega,z)\big[H^{j}_{0}(x,t)+\partial^{\omega}_{j}h^{s}_{0}(x,\omega,t)+\partial_{z_{j}}h^{\sigma}_{0}
(x,\omega,z,t)\big]dz\\
J^{i}_{0}(x,\omega,t)&=&\int_{Z}\sigma_{i}(x,\omega,z, E_{0}(x,t)+D^{\omega}e^{s}_{0}(x,\omega,t)+D_{z}e^{\sigma}_{0}(x,\omega,z,t))dz\\
D^{i}_{0}(x,\omega,t)&=&\int_{Z}\eta_{ij}(x,\omega,z)\big[E^{j}_{0}(x,t)+\partial^{\omega}_{j}e^{s}_{0}(x,\omega,t)+\partial_{z_{j}}e^{\sigma}_{0}
(x,\omega,z,t)\big]dz,\\
\end{array}
\end{equation}
with in addition,
\begin{equation}\label{4.4}
\begin{array}{l}
\int_{Z}\big[\eta_{ij}(x,\omega,z)\partial_{t}\big(E^{j}_{0}(x,t)+\partial^{\omega}_{j}e^{s}_{0}(x,\omega,t)+\partial_{z_{j}}e^{\sigma}_{0}(x,\omega,z,t) \big)\\
+\sigma_{i}(x,\omega,z,E_{0}(x,t)+\partial^{\omega}e^{s}_{0}(x,\omega,t)+D_{z}e^{\sigma}_{0}(x,\omega,z,t)) \big]\times \partial_{z_{i}}v_{3}(z)dz=0,\\
for\; all\; v_{3}\in W^{1,2}_{\#}(Z);\\\\

\int\!\int_{\Lambda\times Z}\big[ \eta_{ij}(x,\omega,z)\partial_{t}\big(E^{j}_{0}(x,t)+\partial^{\omega}_{j}e^{s}_{0}(x,\omega,t)+\partial_{z_{j}}e^{\sigma}_{0}(x,\omega,z,t) \big)\\
+\sigma_{i}(x,\omega,z, E_{0}(x,t)+\partial^{\omega}e^{s}_{0}(x,\omega,t)+D_{z}e^{\sigma}_{0}(x,\omega,z,t)) \big]\times \partial^{\omega}_{j}v_{2}dzd\mu=0 \\
for\; all\; v_{2}\in \mathcal{H}_{\#}(\Lambda);
\end{array}
\end{equation}
and
\begin{equation}\label{4.5}
\begin{array}{l}
\int_{Z}\mu_{ij}(x,\omega,z)\times \big(H^{j}_{0}(x,t)+\partial^{\omega}_{j}h^{s}_{0}(x,\omega,t)+\partial_{z_{j}}h^{\sigma}_{0}(x,\omega,z,t)\big)\partial_{z_{i}}
v_{3}(z)dz=0,\\
for\; all\; v_{3}\in W^{1,2}_{\#}(Z);\\
\int\!\int_{\Lambda\times Z}\mu_{ij}(x,\omega,z)\times \big(H^{j}_{0}(x,t)+\partial^{\omega}_{j}h^{s}_{0}(x,\omega,t)+\partial_{z_{j}}h^{\sigma}_{0}(x,\omega,z,t) \big)\partial_{i}^{\omega}v_{2}dzd\mu=0,\\
for\; all\; v_{2}\in \mathcal{H}_{\#}(\Lambda).
\end{array}
\end{equation}
\end{theorem}

\begin{proof}
The theorem is a summary of result proven above. We will focus on unicity of couple of solution. Assume there are two solutions 
\begin{equation}\nonumber
\begin{array}{l}
\left(E_{1}(x,t)+D^{\omega}e^{s}_{1}(x,\omega,t)+D_{z}e^{\sigma}_{1}(x,\omega,z,t), E_{2}(x,t)+D^{\omega}e^{s}_{2}(x,\omega,t)+D_{z}e^{\sigma}
_{2}(x,\omega,z,t)\right)\\
\hbox{ and }\\
\left(H_{1}(x,t)+D^{\omega}h^{s}_{1}(x,\omega,t)+D_{z}h^{\sigma}_{1}(x,\omega,z,t), H_{2}(x,t)+D^{\omega}h^{s}_{2}(x,\omega,t)+D_{z}h^{\sigma}
_{2}(x,\omega,z,t)\right)
\end{array}
\end{equation}
their differences 
\begin{equation}\nonumber
\begin{array}{l}
E(x,t)+D^{\omega}e^{s}(x,\omega,t)+D_{z}e^{\sigma}(w,\omega,z,t)=\\
E_{1}(x,t)-E_{2}(x,t)+D^{\omega}(e^{s}_{1}(x,\omega,t)-e^{s}_{2}(x,\omega,t))\\
+D_{z}(e^{\sigma}_{1}(x,\omega,z,t)-e^{\sigma}_{2}(x,\omega,z,t)),
\end{array}
\end{equation}
and
\begin{equation}\nonumber
\begin{array}{l}
H(x,t)+D^{\omega}h^{s}(x,\omega,t)+D_{z}h^{\sigma}(w,\omega,z,t)=\\
H_{1}(x,t)-H_{2}(x,t)+D^{\omega}(h^{s}_{1}(x,\omega,t)-h^{s}_{2}(x,\omega,t))\\
+D_{z}(h^{\sigma}_{1}(x,\omega,z,t)-h^{\sigma}_{2}(x,\omega,z,t)),
\end{array}
\end{equation}
verify:
\begin{equation}\nonumber
\begin{array}{l}
E(x,0)+D^{\omega}e^{s}(x,\omega,0)+D_{z}e^{\sigma}(x,\omega,z,0)=0\\
H(x,0)+D^{\omega}h^{s}(x,\omega,0)+D_{z}h^{\sigma}(x,\omega,z,0)=0.
\end{array}
\end{equation}
From limits of the first equation of the system (\ref{1}) after difference one gets:
\begin{equation}\nonumber
\begin{array}{l}
\int^{T}_{0}\!\int_{\Omega\times\Lambda\times Z}\big[\eta_{ij}(x,\omega,z)\partial_{t}\big(E^{j}(x,t)+\partial^{\omega}_{j}e^{s}(x,\omega,t)+\partial_{z_{j}}e^{\sigma}(x,\omega,z,t) \big) +\\
\left.\begin{array}{l}
\sigma_{i}\left(x,\omega,z,E^{1}(x,t)+\partial^{\omega}e^{s}_{1}(x,\omega,t)+D_{z}e^{\sigma}_{1}(x,\omega,z,t) \right) \\
-\sigma_{i}\left(x,\omega,z,E^{2}(x,t)+\partial^{\omega}e^{s}_{2}(x,\omega,t)+D_{z}e^{\sigma}_{2}(x,\omega,z,t) \right)
\end{array}
\right] \times \\
E^{i}(x,t)v_{2}(\omega)dxd\mu dzdt= \\
\int^{T}_{0}\!\int_{\Omega\times \Lambda}\left(\mathrm{curl}\left(H(x,t)+\partial^{\omega}h^{s}(x,\omega,t) \right)\right)_{i}(x,t)E^{i}(x,t)v_{2}(\omega)dxd\mu dt.
\end{array}
\end{equation}
In a similar way, the second equation of (\ref{1}) leads to:
\begin{equation}\label{4.6}
\begin{array}{l}
\int^{T}_{0}\!\int_{\Omega\times\Lambda \times Z}\mu_{ij}(x,\omega,z)\partial_{t}\big(H^{j}(x,t)+\partial^{\omega}_{j}h^{s}(x,\omega,t)+\partial_{z_{i}}h^{\sigma}(x,\omega,z,t) \big)\times H^{i}(x,t)v_{2}dxd\mu dzdt=\\
-\int^{T}_{0}\!\int_{\Omega\times \Lambda}\left(\mathrm{curl}\; E(x,t)+\partial^{\omega}e^{s}(x,\omega,t) \right)_{i}H^{i}(x,t)v_{2}(\omega)dxd\mu dt,\; v_{2}\in \mathcal{C}^{\infty}(\Lambda) \cap I^{2}_{nv}(\Lambda).
\end{array}
\end{equation}
The second equality of lemma \ref{l4.2} applied to $e^{\sigma}_{1}-e^{\sigma}_{2}$ writes as follows:
\begin{equation}\nonumber
\begin{array}{l}
\int_{Z}\big[\eta_{ij}(x,\omega,z)\partial_{t}\big(E^{j}_{1}(x,t)+\partial^{\omega}_{j}e^{s}_{1}(x,\omega,t)+\partial_{z_{j}}e^{\sigma}_{1}(x,\omega,z,t) \big)\\
+\sigma_{i}(x,\omega,z, E_{1}(x,t)+\partial^{\omega}e^{s}_{1}(x,\omega,t)+D_{z}e^{\sigma}_{1}(x,\omega,z,t))  \big]\times \\
\partial_{z_{i}}(e^{\sigma}_{1}-e^{\sigma}_{2})(z)dz=0;\\\\

\int_{Z}\!\big[\eta_{ij}(x,\omega,z)\partial_{t}\big(E^{j}_{2}(x,t)+\partial^{\omega}_{j}e^{s}_{2}(x,\omega,t)+\partial_{z_{j}}e^{\sigma}_{2}(x,\omega,z,t) \big) \\
+\sigma_{i}(x,\omega,z, E^{2}(x,t)+\partial^{\omega}e^{s}_{2}(x,\omega,t)+D_{z}e^{\sigma}_{2}(x,\omega,z,t))  \big]\times \\
\partial_{z_{i}}(e^{\sigma}_{1}-e^{\sigma}_{2})(z)dz=0;
\end{array}
\end{equation}
The second equation in (\ref{1}) leads to $(\mathrm{div}_{z}[\mu(x,\omega,\cdot)\partial_{t}H_{1}(x,t,\omega,\cdot)])=0$ in $Z$, and applying this at $h^{\sigma}_{1}-h^{\sigma}_{2}$ it follows that 
\begin{equation}\nonumber
\begin{array}{l}
\int_{Z}\mu_{ij}(x,\omega,z) \times\\
\partial_{t}\big(H^{j}_{1}(x,t)+\partial^{\omega}_{j}h^{s}_{1}(x,\omega,t)+\partial_{z_{j}}h^{\sigma}_{1}(x,\omega,z,t) \big)\partial_{z_{i}}(h^{\sigma}_{1}-h^{\sigma}_{2})(z)dz=0,\\\\

\int_{Z}\mu_{ij}(x,\omega,z) \times\\
\partial_{t}\big(H^{j}_{2}(x,t)+\partial^{\omega}_{j}h^{s}_{2}(x,\omega,t)+\partial_{z_{j}}h^{\sigma}_{2}(x,\omega,z,t) \big)\partial_{z_{i}}(h^{\sigma}_{1}-h^{\sigma}_{2})(z)dz=0,
\end{array}
\end{equation}
and the difference gives:
\begin{equation}\label{4.7}
\begin{array}{l}
\int_{Z}\big[\eta_{ij}(x,\omega,z)\partial_{t}\big(E^{j}(x,t)+ \partial^{\omega}_{j}e(x,\omega,t)+\partial_{z_{j}}e(x,\omega,z,t) \big) \\
+\left.\begin{array}{l}
\sigma_{i}(x,\omega,z, E_{1}(x,t)+\partial^{\omega}e^{s}_{1}(x,\omega,t)+D_{z}e^{\sigma}_{1}(x,\omega,z,t))-\\
\sigma_{i}(x,\omega,z, E_{2}(x,t)+\partial^{\omega}e^{s}_{2}(x,\omega,t)+D_{z}e^{\sigma}_{2}(x,\omega,z,t))
\end{array}
\right] \\
\times \partial_{z_{i}}(e^{\sigma}_{1}-e^{\sigma}_{2})(z)dz=0;
\end{array}
\end{equation}
and
\begin{equation}\nonumber
\begin{array}{l}
\int_{Z}\mu_{ij}(x,\omega,z) \times\\
\partial_{t}\big(H^{j}(x,t)+\partial^{\omega}_{j}h^{s}(x,\omega,t)+\partial_{z_{j}}h^{\sigma}(x,\omega,z,t) \big)\partial_{z_{i}}(h^{\sigma}_{1}-h^{\sigma}_{2})(z)dz=0.
\end{array}
\end{equation}
Restarting the above process taking in account the two other equalities of lemma \ref{l4.2}, with the quantities $e^{\sigma}_{1}-e^{\sigma}_{2}$ and $h^{s}_{1}-h^{s}_{2}$  gives us:
\begin{equation}\label{4.8}
\begin{array}{l}
\int_{\Lambda}\big[\eta_{ij}(x,\omega,z)\partial_{t}\big(E^{j}(x,t)+ \partial^{\omega}_{j}e(x,\omega,t)+\partial_{z_{j}}e(x,\omega,z,t) \big) \\
+\left.\begin{array}{l}
\sigma_{i}(x,\omega,z, E_{1}(x,t)+\partial^{\omega}e^{s}_{1}(x,\omega,t)+D_{z}e^{\sigma}_{1}(x,\omega,z,t))-\\
\sigma_{i}(x,\omega,z, E_{2}(x,t)+\partial^{\omega}e^{s}_{2}(x,\omega,t)+D_{z}e^{\sigma}_{2}(x,\omega,z,t))
\end{array}
\right] \\
\times \partial^{\omega}_{i}(e^{s}_{1}-e^{s}_{2})(\omega)d\mu=0;
\end{array}
\end{equation}
and
\begin{equation}\nonumber
\begin{array}{l}
\int_{\Lambda\times Z}\mu_{ij}(x,\omega,z) \times\\
\partial_{t}\big(H^{j}(x,t)+\partial^{\omega}_{j}h^{s}(x,\omega,t)+\partial_{z_{j}}h^{\sigma}(x,\omega,z,t) \big)\partial^{\omega}_{i}(h^{\sigma}-h^{\sigma}_{2})(\omega)d\mu=0.
\end{array}
\end{equation}
(\ref{4.6}), (\ref{4.7}) and (\ref{4.8}) lead after summation to:
\begin{equation}\nonumber
\begin{array}{l}
\int_{\Omega\times\Lambda\times Z}\big[\eta_{ij}(x,\omega,z)\partial_{t}\big(E^{j}(x,t)+ \partial^{\omega}_{j}e^{s}(x,\omega,t)+\partial_{z_{j}}e^{\sigma}(x,\omega,z,t) \big) \\
+\left(\left.\begin{array}{l}
\sigma(x,\omega,z, E_{1}(x,t)+\partial^{\omega}e^{s}_{1}(x,\omega,t)+D_{z}e^{\sigma}_{1}(x,\omega,z,t))-\\
\sigma(x,\omega,z, E_{2}(x,t)+\partial^{\omega}e^{s}_{2}(x,\omega,t)+D_{z}e^{\sigma}_{2}(x,\omega,z,t))
\end{array}
\right)\right] \\
\times \big[E^{i}(x,t)+\partial^{\omega}_{i}e^{s}(x,\omega,t)+\partial_{z_{i}}e^{\sigma}(x,\omega,z,t)  \big]dxd\mu dz + \\
\int_{\Omega\times\Lambda\times Z}\mu_{ij}(x,\omega,z)\times \partial_{t}\big(H^{j}(x,t)+\partial^{\omega}_{j}h^{s}(x,\omega,t)+\partial_{z_{j}}h^{\sigma}(x,\omega,z,t) \big)\times \\
\big(H^{j}(x,t)+\partial^{\omega}_{j}h^{s}(x,\omega,t)+\partial_{z_{j}}h^{\sigma}(x,\omega,z,t)\big)dxd\mu dz=0.
\end{array}
\end{equation}
Replace $t$ by $s\in (0,T)$ and integrate on $(0,T),\; t\in (0,T)$. Monotonicity of $\sigma$, symmetry of $\eta$ and $\mu$ and null initial values  lead therefore to the fact that a summation of integral of quantities all positive is negative. It follows then that each of them is null, that is 
\begin{equation}\nonumber
\begin{array}{l}
 E_{j}(x,t)+\partial^{\omega}_{j}e^{s}(x,\omega,t)+\partial_{z_{j}}e^{\sigma}(x,\omega,z,t)=0,\\
 H_{j}(x,t)+\partial^{\omega}_{j}h^{s}(x,\omega,t)+\partial_{z_{j}}h^{\sigma}(x,\omega,z,t)=0
\end{array}
\end{equation}
and unicity is established.
\begin{remark}
 It is possible by proceeding as in \cite{17} to obtain microscopic problems in such a way that local functions $h^{\sigma}_{0}$ and $h^{s}_{0}$ are expressed in terms of the function $H_{0}$.
\end{remark}
\end{proof}


\small
\vspace{1cm} 

\textit{Current adress :} $^{\ddagger }$ The University of Bamenda, Higher Teachers Training
	College, Department of Mathematics, P.O. Box 39, Bambili Cameroon 
\par \textit{E-mail :} fotsotachago@yahoo.fr 
\vspace{0.3cm}
\par \textit{Current adress :} $^{\dagger }$University of Yaounde I, \'{E}cole Normale Sup\'{e}%
	rieure de Yaound\'{e}, P.O. Box 47 Yaounde, Cameroon. 
\par \textit{E-mail :} hnnang@uy1.uninet.cm 

\end{document}